\documentclass[a4paper]{amsart}
\usepackage[ps2pdf]{hyperref}
\usepackage{amsmath,amssymb}
\usepackage{bbold}
\usepackage{psfrag}
\usepackage{graphicx}
\usepackage{leftidx}
\usepackage[all,poly]{xy}
\usepackage[latin1]{inputenc}

\usepackage[normalem]{ulem}
\usepackage{tikz-cd}
\usepackage[dvipsnames]{xcolor}

\newtheorem{thm}{Theorem}[section]
\newtheorem{cor}[thm]{Corollary}
\newtheorem{lem}[thm]{Lemma}

\newcommand{\mylistlabel}[1]{{(#1)}}
\newenvironment{mylist}[1]%
 {\begin{list}{}{\settowidth{\labelwidth}{\mylistlabel{#1}}%
  \settowidth{\labelsep}{\ }
  \setlength{\topsep}{0pt}
  \setlength{\itemsep}{.2em}
  \setlength{\leftmargin}{\labelwidth}\addtolength{\leftmargin}{\labelsep}%
  }}{\end{list}}

\newcommand{\ns}{{\mkern-1.5mu}}
\newcommand{\ps}{{\mkern2mu}}
\newcommand{\B}{X}
\newcommand{\x}{\theta}
\newcommand{\e}{\lambda}
\newcommand{\up}{\mathcal{U}}
\newcommand{\low}{\mathcal{L}}
\newcommand{\I}{\mathcal{I}}

\newcommand{\ant}{\Psi}%{\mathrm{S}}
\newcommand{\pfpart}[1]{\noindent {\it #1.}}

\newcommand{\ab}{\mathrm{ab}}

\newcommand{\CM}{\chi}

\newcommand{\co}{\colon}
\newcommand{\id}{\mathrm{id}}

\newcommand{\opp}{\mathrm{op}}

\newcommand{\cop}{\mathrm{cop}}

\newcommand{\cc}{\mathcal{C}}
\newcommand{\tg}{\mathcal{G}}

\newcommand{\aaa}{\mathcal{A}}

\newcommand{\te}{\mathcal{G}_1^1}

\newcommand{\pp}{\mathcal{P}}

\newcommand{\uu}{\mathcal{U}}

\newcommand{\vv}{\mathcal{V}}

\newcommand{\CO}{C}

\newcommand{\Ker}{\mathrm{Ker}}
\newcommand{\Coker}{\mathrm{Coker}}

\newcommand{\kk}{\Bbbk}
\newcommand{\lact}[2]{\leftidx{^{#1}}{{#2}}{}}
\newcommand{\elact}[2]{\leftidx{^{#1}}{{\hspace{-.08em} #2}}{}}
\newcommand{\kt}{$\Bbbk$\nobreakdash-\hspace{0pt}}

\newcommand{\ti}{\mbox{-}}
\newcommand{\N}{\mathbb{N}}
\newcommand{\Z}{\mathbb{Z}}

\newcommand{\Aut}{\mathrm{Aut}}

\newcommand{\Hom}{\mathrm{Hom}}

\newcommand{\Ima}{{\mathrm{Im}}}

\newcommand{\mmod}{\mathrm{mod}}

\newcommand{\scaledraw}[1]{A}
\newcommand{\scaleraisedraw}[2]{A}
\newcommand{\rsdraw}[3]{\raisebox{-#1\height}{\scalebox{#2}{\includegraphics{#3.eps}}}}

\newcommand{\labeli}{\renewcommand{\labelenumi}{{\rm (\roman{enumi})}}}

\makeatletter
\newcommand*\rel@kern[1]{\kern#1\dimexpr\macc@kerna}
\newcommand*\widebar[1]{%
  \begingroup
  \def\mathaccent##1##2{%
    \rel@kern{0.8}%
    \overline{\rel@kern{-0.8}\macc@nucleus\rel@kern{0.2}}%
    \rel@kern{-0.2}%
  }%
  \macc@depth\@ne
  \let\math@bgroup\@empty \let\math@egroup\macc@set@skewchar
  \mathsurround\z@ \frozen@everymath{\mathgroup\macc@group\relax}%
  \macc@set@skewchar\relax
  \let\mathaccentV\macc@nested@a
  \macc@nested@a\relax111{#1}%
  \endgroup
}
\makeatother

\begin{document}
\title[Quantum invariants of flat 2-bundles over 3-manifolds]{Quantum  invariants of flat 2-bundles\\ over 3-manifolds}
\author{K\"{u}r\c{s}at S\"{o}zer}
\address{K\"{u}r\c{s}at S\"{o}zer\newline
\indent  Univ. Lille, CNRS, UMR 8524 - Laboratoire Paul Painlev\'e, F-59000 Lille, France\\}
\email{kursat.sozer@univ-lille.fr}
\author{Alexis Virelizier}
\address{Alexis Virelizier\newline
\indent Univ. Lille, CNRS, UMR 8524 - Laboratoire Paul Painlev\'e, F-59000 Lille, France\\}
\email{alexis.virelizier@univ-lille.fr}

\subjclass[2020]{57K16, 57K31, 55R15, 18G45, 16T05}
\date{\today}

\begin{abstract}
We construct a scalar invariant of flat principal 2-bundles over 3\ti manifolds, with structure 2-group $\tg$, from an involutory Hopf algebra graded by~$\tg$. Expressing $\tg$ in terms of a crossed module $\CM$ and using the classification of such $2$-bundles via the  classifying space $B\CM$, this amounts to constructing a homotopy invariant of maps from 3-manifolds to $B\CM$. The construction of the invariant relies on a combinatorial description of such maps by $\CM$-colored Heegaard diagrams. When the corresponding map to $B\CM$ is nullhomotopic or, equivalently, when the associated flat principal $\tg$-bundle is trivializable, the invariant reduces to the Kuperberg invariant of the underlying $3$-manifold. 
\end{abstract}

\maketitle

\setcounter{tocdepth}{1}
\tableofcontents

\section{Introduction}\label{sec-Intro}
Since the 1990s, deep connections between low dimensional topology and Hopf algebras have emerged. Notably, the first quantum invariants of 3-manifolds were constructed from the representation theory of quantum groups by Reshetikhin-Turaev~\cite{RT} (using surgery presentations of 3-manifolds) and by Turaev-Viro~\cite{TV} (using triangulations of 3-manifolds). At the same time, Kuperberg~\cite{Ku} derived an invariant of 3-manifolds directly from a finite-dimensional involutory Hopf algebra   (using  Heegaard diagrams of 3-manifolds). Since then, these invariants have been generalized in several directions: using more elaborate algebraic input (such as finite tensor categories) or considering 3-manifolds with extra structure (such as bundle structures). In particular, following the program initiated by Turaev called homotopy quantum field theory, the second author constructed in~\cite{Vi3} an invariant of flat principal bundles over 3-manifolds from group graded involutory Hopf algebras, which generalizes  the Kuperberg invariant. In this paper, we generalize it further: we construct an invariant of flat principal 2-bundles over 3-manifolds from 2-group graded involutory Hopf algebras.

In higher gauge theory, principal $2$-bundles categorify ordinary principal bundles by replacing topological groups with topological $2$-groups (see Appendix~\ref{sect-principal-2-bundles}). Given a topological 2-group $\tg$, principal $\tg$-bundles over a topological space can be described by local transition data, namely by $\tg$-valued \v{C}ech cocycles (see Appendix~\ref{sect-bundle-to-cocycle}). Observe that if $\tg$ is discrete, then $\tg$-valued \v{C}ech cocycles are locally constant and principal $\tg$-bundles are automatically flat and have a unique flat structure (see Appendix~\ref{sect-flat}).

Crossed modules describe  (discrete) 2-groups and model connected homotopy 2\ti types. Recall that a crossed module is a group homomorphism $\CM \co E \to H$ together with a left action of $H$ on $E$ satisfying $H$-equivariance and the Peiffer identity. Every crossed module $\CM$ canonically determines a (strict and discrete) $2$-group~$\tg_{\CM}$ and has   
a classifying space $B\CM$ which is a connected homotopy $2$-type (meaning that $\pi_k(B\CM)=0$ for all $k \geq 3$).
It follows from Baez-Stevenson~\cite{BS} that there is a canonical bijection between equivalence classes of principal $\tg_{\CM}$-bundles (which are all flat) and homotopy classes of maps to $B\CM$. In particular, constructing an invariant of principal $\tg_{\CM}$-bundles over $3$-manifolds amounts to constructing an invariant of $\CM$-manifolds, which are pairs $(M,g)$ where $M$ is a $3$\ti manifold and $g$ is a homotopy class of maps $M \to B\CM$. 

To extend the Kuperberg invariant of 3-manifolds to an invariant of $\CM$-manifolds, we need to extend the Heegaard presentation of 3-manifolds. To this end, we first associate a crossed module (of groupoids) with any (generalized) Heegaard diagram~$D$ of a $3$-manifold $M$, and we express within this crossed module the taut identities for~$D$ (corresponding to the attaching maps of the 3-cells, see~\cite{Si}).  Then a $\CM$\ti labeling of $D$ is a morphism from this crossed module to $\CM$ which is compatible with the taut identities (see Section~\ref{sect-Xi-labelings}). We prove that  the set $[M,B\CM]$ of homotopy classes of maps $M \to B\CM$ is canonically identified with the set of gauge equivalence classes of $\CM$-labelings of~$D$ (see Theorem~\ref{thm-gauge-group-labelings}). 
Thus, $\CM$-Heegaard diagrams (which are Heegaard diagrams endowed with a $\CM$-labeling) give a combinatorial description of $\CM$-manifolds (and therefore of principal $\tg_{\CM}$-bundles over 3\ti manifolds). We illustrate this description with explicit computations for lens spaces and the Poincar\'e homology sphere. Furthermore, we give a colored version of the Reidemeister-Singer theorem:  two $\CM$-colored Heegaard diagrams represent equivalent $\CM$-manifolds if and only if they are related by a finite sequence of $\CM$\ti moves (see Theorem~\ref{thm-colored-Reidemeister}).

Hopf algebras graded by the 2-group  $\tg_{\CM}$, called Hopf $\CM$-coalgebras and introduced in~\cite{SV2}, serve as algebraic input for our construction of invariants of $\CM$-manifolds. These are Hopf $H$-coalgebras (in the sense of \cite{Vi2}) endowed with an action of~$\tg_\CM$. We prove that any finite-type involutory Hopf $\CM$-coalgebra $A$ (over a field $\kk$) has canonical two-sided integrals (see Lemma~\ref{lem-existence-Xi-integrals}) and we use them to assign a scalar to each $\CM$-colored Heegaard diagram (see Section~\ref{sect-def-invariant}). Our main result (Theorem~\ref{thm-colored-Kuperberg}) establishes that this scalar is independent of all choices and defines an invariant 
$$
K_A(M,g) \in \kk
$$
of $\CM$-manifolds, and thus an invariant of (flat) principal $\tg_{\CM}$-bundles over $3$-manifolds. We give examples showing that this invariant is nontrivial. In the special case where~$E=1$, it corresponds to the invariant of (flat) principal $H$-bundles over $3$-manifolds defined in~\cite{Vi3}, while for a trivializable  $\tg_{\CM}$-bundle it reduces to the Kuperberg invariant (associated with the neutral component of $A$) of the underlying 3-manifold.

The paper is organized as follows. In Section~\ref{sect-crossed-classifying}, we first recall background on crossed modules and their classifying spaces. Then we relate principal $\tg_{\CM}$\ti bundles to $\CM$-manifolds. In Section~\ref{sect-Heegaard-diagrams-maps}, we encode $\CM$-manifolds using $\CM$-Heegaard diagrams and their $\CM$-moves. In Section~\ref{sect-Hopf-crossed-module-coalgebras}, we review Hopf $\CM$-coalgebras and give some examples. In Section~\ref{sect-invariant-Xi-manifolds}, we construct from a finite-type  involutory Hopf $\CM$-coalgebra a topological  invariant of $\CM$-manifolds (via $\CM$-Heegaard diagrams).
Appendix~\ref{sect-2-bundles-appendix} reviews the principal $2$-bundles over topological spaces, including their \v{C}ech description and flatness. In Appendices~\ref{sect-crossed-modules-groupoids} and~\ref{sect-HVThm}, we collect necessary background on crossed modules and crossed complexes of groupoids, culminating in the homotopy classification theorem which relates maps to classifying spaces with morphisms of crossed complexes.

We fix throughout the paper a field~$\kk$. We denote by $|A|$ the cardinality of a finite set $A$. For sets $A$ and $B$, we denote by $ \mathrm{Map}(A,B)$ the set of maps $A \to B$. For topological spaces $X$ and $Y$, we denote by $[X,Y]$ the set of homotopy classes of continuous maps $X \to Y$. By an $n$-manifold, we mean a smooth manifold of dimension $n$.

\section{Crossed modules, crossed manifolds, and crossed bundles}\label{sect-crossed-classifying}

In this section, we first review crossed modules and their classifying spaces.  Then we introduce the equivalent notions of crossed manifolds (which are 3-manifolds endowed with a homotopy class  of maps to the classifying space of a crossed module) and crossed bundles (which are principal 2-bundles over 3-manifolds with structure 2-group induced by the crossed module).

\subsection{Crossed modules}\label{sect-crossed-modules-def}
Crossed modules encode (strict) 2-groups, see Appendix~\ref{sect-2-groups}.
A \emph{crossed module} is a group homomorphism $\CM \co E \to H$ together with a left  action of $H$ on $E$ (by group automorphisms), denoted by
$$
(x,e) \in H \times E  \mapsto \elact{x}{e} \in E,
$$ 
such that $\CM$ is equivariant with respect to the conjugation action of $H$ on itself and satisfies the Peiffer identity, that is, for all $x \in H$ and $e,f \in E$,
$$
\CM(\elact{x}{e})=x\CM(e)x^{-1} \quad \text{and} \quad \lact{{\CM(e)}}{\!f}=efe^{-1}.
$$
These axioms imply that the image $\Ima(\CM)$ is normal in $H$ and that the kernel $\Ker(\CM)$ is central in $E$ and is acted on trivially by $\Ima(\CM)$. In particular,  $\Ker(\CM)$ inherits an action of $H/\Ima(\CM)=\Coker(\CM)$.

A \emph{morphism} from a crossed module $\CM \co E \to H$ to a crossed module $\mu \co F \to K$ is a pair $(\psi \co E \to F,\varphi \co H \to K)$
of group homomorphisms such that
$$
\mu\bigl(\psi(e) \bigr)=\varphi \bigl( \CM(e) \bigr) \quad \text{and} \quad \psi(\elact{x}{e})=\lact{\varphi(x)}{\psi(e)}
$$
for all $e\in E$ and $x \in H$.

\subsection{Classifying spaces of crossed modules}\label{sect-crossed-modules-classifying-spaces}
A fundamental geometric example of a crossed module is due to Whitehead: if $(X,Y)$ is a pair of pointed topological spaces, then the homotopy boundary map $$\Pi_2(X,Y)=\partial \co \pi _{2}(X,Y)\rightarrow \pi _{1}(Y),$$
together with the standard action of $\pi _{1}(Y)$ on $\pi _{2}(X,Y)$,  is a crossed module.

Conversely, there exists a classifying space functor~$B$ (see \cite{BH1}) that  assigns to each crossed module $\CM \co E \to H$
a connected,  reduced\footnote{A CW-complex is \emph{reduced} if it has a single 0-cell, which then serves as a basepoint.} CW-complex
$B\CM$, containing a canonical subcomplex $BH$ which is a classifying space of the group $H$, such that
    $$
    \Pi_2(B\CM,BH)=\CM.
    $$
The classifying space $B\CM$ is a homotopy 2-type:
$$
\pi_1(B\CM)=\Coker(\CM), \quad \pi_2(B\CM)=\Ker(\CM), \quad \pi_k(B\CM)=0 \quad \text{for $k \geq 3$.}
$$
If $X$ is a reduced CW-complex, then there is a map $X \to B\Pi_2(X,X^1)$ inducing isomorphisms on $\pi_1$ and $\pi_2$, where $X^1$ is the 1-skeleton of $X$.
It follows that crossed modules model all pointed connected homotopy 2-types (as originally proved by MacLane and Whitehead \cite{MLW}).

\subsection{Examples}\label{sect-crossed-modules-ex}
1. Given any  normal subgroup $E$ of a group $H$, the inclusion $E\hookrightarrow H$ is a crossed module with the conjugation action of $H$ on $E$. Its classifying space is an Eilenberg-MacLane space of type $K(H/E,1)$.

2. If $E$ is an abelian group, then the trivial map $E \to 1$ is a crossed module. Its classifying space is an Eilenberg-MacLane space of type $K(E,2)$.

3. Let $H$ be a group acting (by group automorphisms) on an abelian group $E$. Then the trivial morphism $\CM \co E \to H$, defined by $\CM(e)=1$ for all $e \in E$, is a crossed module. Note that the action of $\pi_1(B\CM)=H$ on $\pi_2(B\CM)=E$ is the given action of $H$ on $E$, and so $B\CM$ is homotopy equivalent to $K(H,1) \times K(E,2)$ if and only if this action is trivial.

4. For any group $E$, the homomorphism $\CM\co E \to \Aut(E)$ sending any element of $E$ to the corresponding inner automorphism is a crossed module. Note that $\pi_1(B\CM)=\mathrm{Out}(E)$ and $\pi_2(B\CM)=Z(E)$.

\subsection{Crossed manifolds}\label{sect-Xi-manifolds} 
Let $\CM$ be a crossed module. A \emph{$\CM$-manifold} is a pair $(M,g)$ where $M$ is a closed connected oriented 3-manifold and $g  \in [M,B\CM]$ is a  homotopy class  of continuous maps $M \to B\chi$. 

Two $\CM$-manifolds $(M,g)$ and $(M',g')$ are \emph{equivalent} if there is an orientation preserving diffeomorphism $\psi\co M \to M'$ such that $g'\psi=g$.

By a \emph{topological invariant} of $\CM$-manifolds, we mean a quantity associated to $\CM$-manifolds which remains unchanged under the above equivalence relation.

\subsection{Crossed  bundles}\label{sect-Xi-bundles}
Let $\CM$ be a crossed module. Denote by $\tg_\CM$ its associated (discrete) $2$-group (see Appendix~\ref{sect-2-groups}). 
A \emph{$\CM$-bundle} is a pair $(M,\pi)$ where $M$ is a closed connected oriented 3-manifold and $\pi\co P \to \widebar{M}$ is a principal $\tg_\CM$-bundle over~$M$ (see Appendix~\ref{sect-principal-2-bundles}). Note that any  principal $\tg_\CM$-bundle is flat and has a unique flat structure (see Appendix~\ref{sect-flat}).

Two $\CM$-bundles $(M,\pi)$ and $(M',\pi')$ are \emph{equivalent} if there is an orientation preserving diffeomorphism $\psi\co M \to M'$ such that the principal $\tg_\CM$-bundles $\psi_*\pi$ and~$\pi'$ over $M'$ are equivalent (see Appendix~\ref{sect-principal-2-bundles}). Here the principal $\tg_\CM$-bundle $\psi_*\pi$ is defined as the composition of $\pi\co P \to \widebar{M}$ with the continuous functor $(\psi,\psi)\co \widebar{M}\to \widebar{M'}$  (see Appendix~\ref{sect-2-spaces}).

By a \emph{topological invariant} of $\CM$-bundles, we mean a quantity associated to $\CM$\ti bun\-dles which remains unchanged under the above equivalence relation.

\subsection{Crossed manifolds vs crossed bundles}\label{sect-flat-2-bundles}
Let $\CM$ be a crossed module and~$\tg_\CM$ be its associated $2$-group.
Given a 3-manifold $M$, denote by $\pp(M,\tg_\chi)$ the set of equivalence classes of (flat) principal $\tg_\chi$-bundles over~$M$. Recall that $[M,B\chi]$ denotes the set of homotopy classes of continuous maps $M \to B\chi$. Since $M$ is a paracompact Hausdorff space admitting good covers, it follows from Appendix~\ref{sect-flat} that there is a canonical bijection
$$
\pp(M,\tg_\chi) \cong [M,B\chi].
$$
This induces a canonical bijection between the set of equivalence classes of $\CM$-bundles and the set of equivalence classes of $\CM$-manifolds. In particular, constructing a topological invariant of $\CM$-bundles corresponds to constructing a topological invariant of $\CM$-manifolds.

\section{Colored Heegaard diagrams}\label{sect-Heegaard-diagrams-maps} 

Throughout this section, $\CM \co E \to H$ is a crossed module. We introduce the notion of a $\CM$-labeling of a (generalized) Heegaard diagram of a 3-manifold. To this end, we first associate a crossed module (of groupoids) with a Heegaard diagram and provide an algorithm to extract taut identities. We then prove that gauge equivalence classes of $\CM$-labelings encode 
homotopy classes of maps from 3\ti manifolds to the classifying space of $\CM$ (see Theorem~\ref{thm-gauge-group-labelings}). In particular, $\CM$-Heegaard diagrams  (which are Heegaard diagrams endowed with a $\CM$-labeling)  give a combinatorial description of $\CM$-manifolds (and hence  of $\CM$-bundles). Finally, we prove  that two $\CM$\ti Heegaard diagrams represent equivalent $\CM$-manifolds if and only if they are related by a finite sequence of $\CM$-colored Heegaard moves (see Theorem~\ref{thm-colored-Reidemeister}).

\subsection{Heegaard diagrams}\label{sect-Heegaard-diag}
By a  \emph{planar surface}, we mean a 2-dimensional manifold which can be embedded in the plane. 
A \emph{Heegaard diagram} is a triple $D=(\Sigma,\up,\low)$, where:
\begin{itemize}
\item $\Sigma$ is a closed, connected, oriented surface,
\item  $\up$  is a finite set of pairwise disjoint simple closed curves on $\Sigma$, called the \emph{upper circles},
\item  $\low$  is a finite set of pairwise disjoint simple closed curves on $\Sigma$, called the \emph{lower circles},
\end{itemize}
such that  both $\Sigma \setminus \up$ and $\Sigma \setminus \low$  are a disjoint union of planar surfaces and 
the upper and lower circles intersect transversely at a finite set of points, called the \emph{intersection points} of $D$.
For $u \in \up$, we denote by $|u|$ the number of intersection points of $D$ lying in $u$. Likewise, for $l \in \low$, we denote by $|l|$ the number of intersection points of $D$ lying in $l$. Then the total number of intersection points of $D$ is computed by
$$
\sum_{u \in \up} |u|=\sum_{l \in \low} |l|.
$$

In the literature, it is often also required that both $\Sigma \setminus \up$ and $\Sigma \setminus \low$ are connected (or equivalently that the number of upper circles and the number of lower circles are equal to the genus of $\Sigma$). The (generalized) definition given above follows that of \cite{Ku} and is necessary for our purposes (see Section~\ref{sect-rem-connected}).

A Heegaard diagram is \emph{oriented} if all its upper and lower circles are oriented. The intersection points of an oriented Heegaard diagram $D=(\Sigma,\up,\low)$ have a sign defined as follows.  An intersection point $s$ between an upper circle $u$ and a lower circle $l$ is \emph{positive} if $u$ intersects $l$ positively at $s$ and is \emph{negative} otherwise: 
\begin{center}
\begin{minipage}[c]{70pt}
\begin{center}
   \psfrag{l}[Bl][Bl]{\scalebox{1}{$l$}}
   \psfrag{s}[Bc][Bc]{\scalebox{1}{$s$}}
   \psfrag{u}[Br][Br]{\scalebox{1}{$u$}}
   \psfrag{S}[Br][Br]{\scalebox{1}{$\Sigma$}}
\rsdraw{.45}{.9}{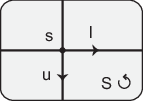}  \\[.5em]
$s$ positive
\end{center}
\end{minipage}
\hspace*{4em}
\begin{minipage}[c]{75pt}
\begin{center}
   \psfrag{l}[Bl][Bl]{\scalebox{1}{$l$}}
   \psfrag{s}[Bc][Bc]{\scalebox{1}{$s$}}
   \psfrag{u}[Br][Br]{\scalebox{1}{$u$}}
   \psfrag{S}[Br][Br]{\scalebox{1}{$\Sigma$}} 
\rsdraw{.45}{.9}{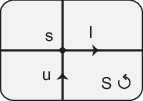}  . \\[.5em]
$s$ negative \phantom{.}
\end{center}
\end{minipage}
\end{center}
\smallskip

A Heegaard diagram is  \emph{pointed} if each lower circle is equipped with a basepoint distinct from the intersection points.

\subsection{Crossed modules from Heegaard diagrams}\label{sect-CM-from-Heegaard}
We associate to any  oriented pointed Heegaard diagram $D=(\Sigma,\up,\low)$ a crossed module of groupoids (see Appendix~\ref{sect-crossed-modules-groupoids}) in the following way.

Any upper circle $u \in \up$ determines two (possibly coinciding) connected components  $c_-^u$ and $c_+^u$ of $\Sigma \setminus \up$ which are adjacent to $u$ and where the subscript $\pm$ is determined by the orientation of $u$ and $\Sigma$ as follows:
$$
   \psfrag{o}[Bl][Bl]{\scalebox{1}{$c^u_+$}}
   \psfrag{c}[Br][Br]{\scalebox{1}{$c^u_-$}}
   \psfrag{u}[Br][Br]{\scalebox{1}{$u$}}
   \psfrag{S}[Br][Br]{\scalebox{1}{$\Sigma$}}
\rsdraw{.45}{.9}{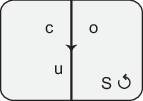} \; .
$$
Consider the oriented graph $\Gamma_{\Sigma,\up}$ whose set of vertices is the set $\pi_0(\Sigma \setminus \up)$ of connected components of  $\Sigma \setminus \up$ and whose set of edges is $\up$, where each upper circle~$u$ is viewed as an oriented edge from $c^u_-$ to $c^u_+$. Note that $\Gamma_{\Sigma,\up}$ is connected.
Denote by $\mathcal{F}_{\Sigma,\up}$ the free groupoid on $\Gamma_{\Sigma,\up}$ (see Appendix~\ref{sect-free-groupoids}).

Each lower circle $l$ determines an endomorphism $\omega_l$ of $\mathcal{F}_{\Sigma,\up}$ as follows. Let $u_1, \dots, u_k$ be the upper circles encountered by $l$  when making  a round trip along~$l$ starting from its basepoint and following its orientation, and let $\nu_1, \dots, \nu_k$ be the signs of the corresponding intersection points (see Section~\ref{sect-Heegaard-diag}). Denote by $c_l$ the connected component  of $\Sigma \setminus \up$ containing the basepoint of $l$.
Then
$$ 
\omega_l=u_1^{\nu_1} \cdots u_k^{\nu_k}  \in \mathcal{F}_{\Sigma,\up}(c_l).
$$
By Appendix~\ref{sect-free-crossed-modules-over-groupoids}, the induced map $\omega \co \low \to \mathcal{F}_{\Sigma,\up}$ gives rise to the free crossed module on $\omega$, denoted here as:
$$
\nu_D\co \mathcal{F}_D=\{\mathcal{F}_D(c)\}_{c \in \pi_0(\Sigma \setminus \up)} \to \mathcal{F}_{\Sigma,\up}.
$$
Recall that $\mathcal{F}_D$ is a quotient of the free groupoid $\mathcal{P}_D=\{\mathcal{P}_D(c)\}_{c \in \pi_0(\Sigma \setminus \up)}$, where $\mathcal{P}_D(c)$ is the free group on the set of pairs  $(x,l)$ with $l  \in \low$ and $x \in \mathcal{F}_{\Sigma,\up}(c,c_l)$. Explicitly, for $c\in \pi_0(\Sigma \setminus \up)$, the group $\mathcal{F}_D(c)$ is the quotient of $\mathcal{P}_D(c)$ by its subgroup generated by the Peiffer commutators (see Appendix~\ref{sect-free-crossed-modules-over-groupoids}).
Let us denote by $\langle f\rangle \in \mathcal{F}_D(c)$  the class of an element $f \in \mathcal{P}_D(c)$. Then
$$
\nu_D(\langle x,l\rangle)=x\omega_l x^{-1} \quad \text{and} \quad \lact{y}{\langle x,l\rangle}=\langle yx,l\rangle
$$
for all  generators $(x,l)$ of $\mathcal{P}_D(c)$ and all $y \in \mathcal{F}_{\Sigma,\up}$ with $t(y)=s(x)$. 
By construction, the crossed module $\nu_D\co \mathcal{F}_D  \to \mathcal{F}_{\Sigma,\up}$ is free  
with free basis $ \{\mathcal{B}_2=\{b_l\}_{l \in \low},\, \mathcal{B}_1=\up\}$, where $b_l=\langle (1_{c_l},l)\rangle \in \mathcal{F}_D(c_l)$.

Consider the 2-dimensional CW-complex $X_D$ associated with the above presentation of $\nu_D$. Its 0-skeleton is the set $X_D^0=\pi_0(\Sigma \setminus \up)$, its 1-skeleton~$X_D^1$ is the geometric realization of the graph  $\Gamma_{\Sigma,\up}$, and $X_D$ is obtained by gluing for each lower circle $l$ a 2-cell to $X_D^1$ according to $\omega_l$. Note that $X_D$ is connected and that there is a canonical isomorphism of crossed modules (of groupoids) between $\nu_D\co \mathcal{F}_D  \to \mathcal{F}_{\Sigma,\up}$ and the fundamental crossed module 
$$
\pi_2(X_D,X_D^1)=\{\pi_2(X_D,X_D^1,c)\}_{c \in X_D^0}  \xrightarrow{ \partial_2}   \pi_1(X_D^1)=\{\pi_1(X_D^1,c)\}_{c \in X_D^0}
$$
of the triple $(X_D,X_D^1,X_D^0)$ induced by the usual boundary maps.

\subsection{Heegaard diagrams of 3-manifolds}\label{sect-Heegaard-diag-3man}
Let $\Sigma$ be a closed connected oriented surface and  $\low$  be a finite set of pairwise disjoint simple closed curves on $\Sigma$ 
such that   $\Sigma \setminus \low$ is a disjoint union of planar surfaces. The pair $(\Sigma,\low)$ gives rise to an oriented handlebody $\mathcal{H}_{\Sigma,\low}$ with boundary $\partial(\mathcal{H}_{\Sigma,\low})=\Sigma$  as follows. Consider the 3-manifold $\Sigma \times [0,1]$. Glue a 2-handle along a tubular neighborhood of each circle $l \times \{0\}$ with $l \in \low$.  The result is a compact connected 3-manifold whose boundary is $\Sigma \times \{1\} \cong \Sigma$ plus a  disjoint union of $| \pi_0(\Sigma \setminus \low)|$ spheres. We eliminate these spherical boundary components by gluing in 3-balls to obtain $\mathcal{H}_{\Sigma,\low}$. We endow $\mathcal{H}_{\Sigma,\low}$ with the orientation extending the product orientation of $\Sigma \times [0,1]$. 

Any Heegaard diagram $D=(\Sigma,\up,\low)$ gives rise to the closed connected oriented 3-manifold
$$
M_D=\mathcal{H}_{\Sigma,\low}\bigcup_\Sigma (-\mathcal{H}_{\Sigma,\up}),
$$ 
where $-\mathcal{H}_{\Sigma,\up}$ denotes $\mathcal{H}_{\Sigma,\up}$ with the opposite orientation. In particular, the surface $\Sigma$ is a splitting surface for $M_D$, meaning that it cuts $M_D$ into two handlebodies.

A \emph{Heegaard diagram of a  closed connected oriented 3-manifold} $M$ is a Heegaard diagram $D=(\Sigma,\up,\low)$ together with an embedding of  $\Sigma$ in $M$ which extends to an orientation preserving diffeomorphism $M_D \cong M$. 

Any closed connected oriented 3-manifold has a Heegaard diagram. Moreover, the Reidemeister-Singer theorem asserts that two Heegaard diagrams give rise to diffeomorphic oriented 3-manifolds if and only if one can be obtained from the other by a finite sequence of Heegaard moves, see Section~\ref{sect-Xi-moves}.

\subsection{Taut identities of Heegaard diagrams}\label{sect-Taut}
Let $D=(\Sigma,\up,\low)$  be an oriented pointed Heegaard diagram. Denote by $M=M_D$ the  closed connected oriented 3-manifold associated with $D$ (see Section~\ref{sect-Heegaard-diag-3man}).
Following~\cite[Chapter 8]{HMS}, the  2-dimensional CW-complex $X_D$  associated with $D$ (see Section~\ref{sect-CM-from-Heegaard}) embeds in $M$ in such a way that $M\setminus X_D$ is a disjoint union of open 3-balls. This gives rise to a CW-decomposition of $M$ with $2$-skeleton~$X_D$. 
In particular, this decomposition of~$M$ has $| \pi_0(\Sigma \setminus \up)|=1+|\up|-\mathrm{genus}(\Sigma)$ 0-cells,  $|\up|$ 1-cells, $|\low|$ 2-cells, and $| \pi_0(\Sigma \setminus \low)|=1+|\low|-\mathrm{genus}(\Sigma)$ 3-cells. 
Consider the free crossed complex (of groupoids)
$$
\Pi M^\ast=\bigl ( \cdots \to 1 \to 1 \to  \pi_3(M,M^2) \xrightarrow{\partial_3}  \pi_2(M^2,M^1) \xrightarrow{\partial_2}   \pi_1(M^1)   \bigr).
$$
associated with the skeletal filtration of $M^*=\{M^0 \subseteq M^1 \subseteq M^2 \subseteq M\}$ of $M$ (see Appendix~\ref{exa-crossed-complexes-of-CW-complexes}). 
In particular, $\pi_3(M,M^2)$ is a free $\pi_1(M)$-module with free basis $\mathcal{B}_3=\{h_\sigma\}_{\sigma \in \pi_0(\Sigma \setminus \low)}$ obtained by picking a basepoint on the boundary of each  3-cell and then taking the homotopy classes of the characteristic maps of the 3-cells of~$M$.
It follows from Section~\ref{sect-CM-from-Heegaard} that  the cellular homeomorphism $X_D\cong M^2$ induces an 
isomorphism of crossed modules (of groupoids) 
$$
\xymatrix@R=.7cm @C=.9cm{
\pi_2(M^2,M^1) \ar@{->}[r]^-{\partial_2}\ar@{->}[d]_-{\eta}^-{\mathbin{\rotatebox[origin=c]{-90}{\,$\simeq$}}}      & \pi_1(M^1) \ar@{->}[d]^-{\mathbin{\rotatebox[origin=c]{-90}{\,$\simeq$}}}     \\
 \mathcal{F}_D \ar@{->}[r]^-{\nu_D}  &  \mathcal{F}_{\Sigma,\up}.
}
$$
Recall that $\mathcal{F}_D$ is the quotient of the free groupoid $\mathcal{P}_D$ by the Peiffer commutators and that the class of an element $f \in \mathcal{P}_D$ is denoted by $\langle f\rangle \in \mathcal{F}_D$.
By a \emph{representative set of taut identities} for $D$, we mean a family  $\tau=\{\tau_\sigma\}_{\sigma \in \pi_0(\Sigma \setminus \low)}$ of elements of $\mathcal{P}_D$ such that for all $\sigma \in \pi_0(\Sigma \setminus \low)$,
$$
\langle\tau_\sigma\rangle=\eta\partial_3(h_\sigma) \in \mathcal{F}_D.
$$ 
Note that any two taut identities associated with $\sigma$ are related by a Peiffer commutator (since they represent the same class in $\mathcal{F}_D$) and by the action of $\mathcal{F}_{\Sigma,\up}$ (due to the choices of basepoints for the  3-cells to define $\mathcal{B}_3$).

The following procedure explains how to find a taut identity associated with a component $\sigma \in \pi_0(\Sigma \setminus \low)$. Pick a point $m$ in the interior of $\sigma$ distinct from the upper circles (which serves as a basepoint for the 3-cell associated with $\sigma$).
Each connected component~$b$ of the boundary $\partial \sigma$ of $\sigma$ comes from the cutting of $\Sigma$ along some lower circle $l_b$. It inherits an orientation and a basepoint from $l_b$. Pick an arc~$\gamma_b$ in $\sigma$ starting at $m$  and ending at the basepoint of $b$:
$$
\psfrag{E}[Bc][Bc]{\scalebox{1.5}{$\leftarrow$}}
\psfrag{a}[Bc][Bc]{\scalebox{.9}{{\color{red}$\gamma_{b}$}}}
\psfrag{b}[Br][Br]{\scalebox{.9}{$b$}}
\psfrag{h}[Br][Br]{\scalebox{.9}{$l_b$}}
\psfrag{m}[Bc][Bc]{\scalebox{.9}{{\color{red}$m$}}}
\psfrag{S}[Bc][Bc]{\scalebox{.9}{$\Sigma$}}
\psfrag{R}[Bc][Bc]{\scalebox{.9}{$\sigma$}}
\rsdraw{.45}{.9}{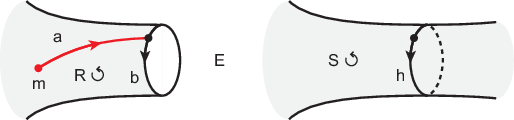} \;.
$$
We require that the arcs $\{\gamma_b\}_{b \in \pi_0(\partial \sigma)}$  intersect the upper  circles transversally and intersect each other only at their initial point $m$.  Set $\varepsilon_b=+1$ if the orientation of $b$ coincides with that induced\footnote{We use the ``outward vector first'' convention for the induced orientation of the boundary.} by the orientation of~$\sigma$ (inherited from $\Sigma$) and $\varepsilon_b=-1$ otherwise. Denote by $c_m$ the connected component of $\Sigma \setminus \up$ containing $m$. Recall that 
$c_{l_b}$  denotes the connected component  of $\Sigma \setminus \up$ containing the basepoint of $l_b$.
Set  
$$ 
r_b =u_1^{\nu_1} \cdots u_s^{\nu_s}  \in \mathcal{F}_{\Sigma,\up}(c_m,c_{l_b})
$$
where $u_1, \dots, u_s$ are the upper circles encountered by $\gamma_b$ when traveling along it from~$m$, and $\nu_1, \dots, \nu_s$
are the signs of the intersection points of the upper circles with~$\gamma_b$:  
\begin{center}
\begin{minipage}[c]{70pt}
\begin{center}
   \psfrag{l}[Bl][Bl]{\scalebox{1}{{\color{red}$\gamma_b$}}}
   \psfrag{u}[Br][Br]{\scalebox{1}{$u_i$}}
   \psfrag{S}[Br][Br]{\scalebox{1}{$\sigma$}}
\rsdraw{.45}{.9}{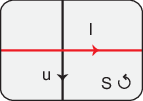}  \\[.5em]
$\nu_i=+1$  
\end{center}
\end{minipage}
\hspace*{4em}
\begin{minipage}[c]{75pt}
\begin{center}
   \psfrag{l}[Bl][Bl]{\scalebox{1}{{\color{red}$\gamma_b$}}}
   \psfrag{u}[Br][Br]{\scalebox{1}{$u_i$}}
   \psfrag{S}[Br][Br]{\scalebox{1}{$\sigma$}} 
\rsdraw{.45}{.9}{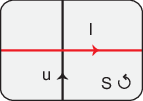}  . \\[.5em]
$\nu_i=-1$ \phantom{.}
\end{center}
\end{minipage}
\end{center}
\smallskip
Enumerate the boundary components of $\sigma$ by $b_1, \dots,b_n$
so that $\gamma_{b_1},\dots, \gamma_{b_n}$ are the arcs successively encountered while traversing a small loop negatively encircling $m$ and starting from any point on that loop:
$$
\psfrag{a}[Bc][Bc]{\scalebox{.9}{{\color{red}$\gamma_{b_1}$}}}
\psfrag{c}[Br][Br]{\scalebox{.9}{{\color{red}$\gamma_{b_2}$}}}
\psfrag{e}[Bl][Bl]{\scalebox{.9}{{\color{red}$\gamma_{b_n}$}}}
\psfrag{m}[Bc][Bc]{\scalebox{.9}{{\color{red}$m$}}}
\psfrag{S}[Bc][Bc]{\scalebox{.9}{$\sigma$}}
\rsdraw{.45}{.9}{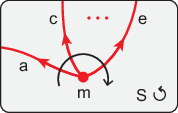} \;.
$$
Then
$$
\tau_\sigma= \prod_{k=1}^n (r_{b_k},l_{b_k})^{\varepsilon_{b_k}} \in \mathcal{P}_D(c_m).
$$
Note that any other choice of the point $m$, of the curves $\gamma_b$, and of the point on the loop used to enumerate the boundary components of $\sigma$ gives rise to another taut relation which is related to the above by Peiffer commutators and the action of~$\mathcal{F}_{\Sigma,\up}$. We illustrate this procedure in the next two examples.

\subsection{Example}(Poincar\'e's homology sphere)\label{ex-Poincare}
Consider the Heegaard diagram on a genus 2 surface $\Sigma$ used by Poincar\'e to define  his celebrated homology 3-sphere~$\mathbb{P}$  (see \cite{Po}). It has two upper circles $u_1,u_2$ and two lower circles $l_1,l_2$. The only connected component $\sigma$ of~$\Sigma \setminus \{l_1,l_2\}$  is the planar surface $\sigma$ depicted as:
$$
\psfrag{1}[Bc][Bc]{\scalebox{.9}{$1$}}
\psfrag{2}[Bc][Bc]{\scalebox{.9}{$2$}}
\psfrag{3}[Bc][Bc]{\scalebox{.9}{$3$}}
\psfrag{4}[Bc][Bc]{\scalebox{.9}{$4$}}
\psfrag{5}[Bc][Bc]{\scalebox{.9}{$5$}}
\psfrag{6}[Bc][Bc]{\scalebox{.9}{$6$}}
\psfrag{7}[Bc][Bc]{\scalebox{.9}{$7$}}
\psfrag{a}[Bc][Bc]{\scalebox{.9}{$a$}}
\psfrag{b}[Bc][Bc]{\scalebox{.9}{$b$}}
\psfrag{c}[Bc][Bc]{\scalebox{.9}{$c$}}
\psfrag{d}[Bc][Bc]{\scalebox{.9}{$d$}}
\psfrag{e}[Bc][Bc]{\scalebox{.9}{$e$}}
\psfrag{r}[Bc][Bc]{\scalebox{.9}{{\color{red}$\gamma_1$}}}
\psfrag{o}[Bc][Bc]{\scalebox{.9}{{\color{red}$\gamma_2$}}}
\psfrag{n}[Bc][Bc]{\scalebox{.9}{{\color{red}$\gamma_3$}}}
\psfrag{z}[Bc][Bc]{\scalebox{.9}{{\color{red}$\gamma_4$}}}
\psfrag{m}[Bc][Bc]{\scalebox{.9}{{\color{red}$m$}}}
\psfrag{l}[Bc][Bc]{\scalebox{.9}{$l_1^-$}}
\psfrag{k}[Bc][Bc]{\scalebox{.9}{$l_1^+$}}
\psfrag{x}[Bc][Bc]{\scalebox{.9}{$l_2^+$}}
\psfrag{y}[Bc][Bc]{\scalebox{.9}{$l_2^-$}}
\psfrag{u}[Bc][Bc]{\scalebox{.9}{$u_1$}}
\psfrag{v}[Bc][Bc]{\scalebox{.9}{$u_2$}}
\psfrag{S}[Bc][Bc]{\scalebox{.9}{$\sigma$}}
\rsdraw{.45}{.9}{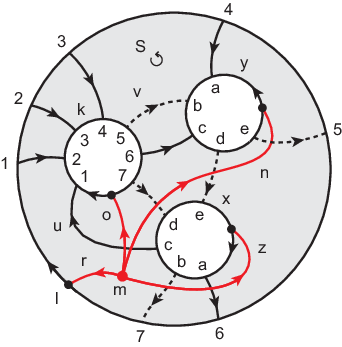} \;.
$$
The arrows and the dots represent the orientations and  basepoints of the circles. The points $1,\cdots, 7$ (respectively, $a, \cdots, e$) are the intersection points between the upper circles and the lower circle $l_1$ (respectively, $l_2$). The four connected components of the boundary of $\sigma$  are denoted $l_1^+,l_1^-,l_2^+,l_2^-$ in such a way that~$l_i^\pm$ arises from cutting~$\Sigma$ along $l_i$ and $\varepsilon_{l_i^\pm}=\pm 1$. The elements associated with the lower circles (as in Section~\ref{sect-CM-from-Heegaard}) are
$$
\omega_{l_1}=u_1^4 u_2^{-1} u_1^{-1} u_2^{-1} \quad  \text{and} \quad \omega_{l_2}=u_1^{-1} u_2^{-1}u_1^{-1} u_2^2.
$$
For the choice of  $\gamma$-curves depicted in red, we get
$$
r_{l_1^+}=u_1^{-1}, \quad r_{l_1^-}=1,  \quad r_{l_2^+}=u_2u_1, \quad r_{l_2^-}=u_1^{-1}u_2^3.
$$
Enumerating the curves from $\gamma_1$ to $\gamma_4$ (in accordance with the opposite orientation of $\sigma$), we obtain the taut identity
$$
\tau_\sigma= (1,l_1)^{-1} (u_1^{-1},l_1)  (u_1^{-1}u_2^3,l_2)^{-1}(u_2u_1,l_2).
$$

\subsection{Example}(Lens spaces)\label{ex-Lens}
Let $p$ and $q$ be coprime integers with $1 \leq q < p$. The lens space $L(p,q)$  has an oriented pointed Heegaard diagram on a genus 1 surface~$\Sigma$ with one upper circle $u$ and one lower circle $l$. The only connected component $\sigma$ of~$\Sigma \setminus \{l\}$  is the planar surface $\sigma$ depicted as:
$$
\psfrag{1}[Bc][Bc]{\scalebox{.9}{$1$}}
\psfrag{2}[Bc][Bc]{\scalebox{.9}{$2$}}
\psfrag{p}[Bc][Bc]{\scalebox{.9}{$p$}}
\psfrag{q}[Bc][Bc]{\scalebox{.9}{$q$}}
\psfrag{t}[Br][Br]{\scalebox{.9}{$q+1$}}
\psfrag{r}[Br][Br]{\scalebox{.9}{$q+2$}}
\psfrag{b}[Bc][Bc]{\scalebox{.9}{$p-q$}}
\psfrag{l}[Bc][Bc]{\scalebox{.9}{$l^-$}}
\psfrag{k}[Bc][Bc]{\scalebox{.9}{$l^+$}}
\psfrag{u}[Bc][Bc]{\scalebox{.9}{$u$}}
\psfrag{S}[Bc][Bc]{\scalebox{.9}{$\sigma$}}
\rsdraw{.45}{.9}{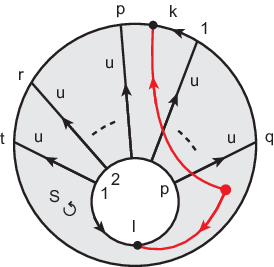} \;.
$$
The arrows and the dots represent the orientations and  basepoints of the circles. The points $1,\cdots, p$ are the intersection points between $u$ and $l$.  The two connected components of the boundary of $\sigma$  are denoted $l^+$ and $l^-$ so that 
$\varepsilon_{l^\pm}=\pm 1$. The element associated with the lower circle (as in Section~\ref{sect-CM-from-Heegaard}) is
$$
\omega_{l}=u^p.
$$
For the choice of  $\gamma$-curves depicted in red,  we obtain the taut identity
$$
\tau_\sigma= (1,l)^{-1}(u^q,l).
$$

\subsection{Labelings of Heegaard diagrams}\label{sect-Xi-labelings}
Recall that  $\CM \co E \to H$ is a crossed module of groups (and so of groupoids by viewing any group as a groupoid with one object).
In this section, we introduce $\CM$-labeled Heegaard diagrams and prove that they encode homotopy classes of maps from closed 3-manifolds to the classifying space $B\CM$ of $\CM$. 

Let $D=(\Sigma,\up,\low)$ be  an oriented pointed Heegaard diagram. 
Consider the groupoid $\mathcal{F}_{\Sigma,\up}$ freely generated by $\up$ (see Section~\ref{sect-CM-from-Heegaard}). Any map $\alpha \co \up \to H$  determines a unique groupoid morphism $\tilde{\alpha} \co \mathcal{F}_{\Sigma,\up} \to H$  such that $\tilde{\alpha}(u)=\alpha(u)$ for all $u \in \up$. 

A \emph{$\CM$-labeling} of $D$ is a pair  $(\alpha, \beta) \in \mathrm{Map}(\up,H) \times \mathrm{Map}(\low,E)$ satisfying the following two conditions:
\begin{enumerate}
\labeli
\item For any lower circle $l \in \low$,
$$
\CM(\beta(l))=\tilde{\alpha}(\omega_l)
$$  
where $\omega_l \in \mathcal{F}_{\Sigma,\up}$ is defined in Section~\ref{sect-CM-from-Heegaard}.
\item For every element $\prod_{k=1}^n (r_k,l_k)^{\varepsilon_k}$ of a representative set of taut identities for $D$,
    $$
    \prod_{k=1}^{n}\left(\lact{\tilde{\alpha}(r_k)}{\beta(l_k)}\right)^{\!\varepsilon_k} =1.
    $$
\end{enumerate}
We prove below (see Theorem~\ref{thm-gauge-group-labelings}) that this definition does not depend on the choice of the representative set of taut identities in (ii). For examples of $\CM$-labelings, see Sections~\ref{ex-Lens-labelings} and~\ref{ex-Poincare-labelings} below.

The \emph{gauge group} of  $D$ is $$\tg_D=\mathrm{Map}(\pi_0(\Sigma\setminus \up),H) \ltimes \mathrm{Map}(\up,E)$$
with product defined by
$(a,d) (a',d')=(a'',d'')$ where
$$
a''(c)=a(c)a'(c) \quad \text{and} \quad 
d''(u)= \Bigl (\lact{a'(c^u_-)^{-1}}{\!\!d(u)}\!\Bigr ) d'(u)
$$ 
for all $c \in \pi_0(\Sigma\setminus \up)$  and  $u \in \up$. (The notation $c^{\pm}_u$ and $c_l$ associated with upper and lower circles is defined in Section~\ref{sect-CM-from-Heegaard}.)  
The group $\tg_D$ acts (on the left) on the set of $\CM$-labelings of~$D$ as follows: for any  $(a,d) \in \tg_D$ and any $\CM$-labeling $(\alpha,\beta)$ of~$D$, the $\CM$-labeling $(\alpha',\beta')=(a,d) \cdot (\alpha,\beta)$ of $D$ is defined for all $u \in \up$ and $l \in \low$ by
$$
\alpha'(u) = a(c^u_-) \CM(d(u)) \alpha(u) a(c^u_+)^{-1} \quad \text{and} \quad
\beta'(l) =\lact{a(c_l)}{\!{\big(d_\alpha(\omega_l)\beta(l)\big)}}.
$$
Here $d_\alpha$ is the derivation over the groupoid morphism $\tilde{\alpha}\co  \mathcal{F}_{\Sigma,\up} \to H$ extending $d$, that is, the (unique) map $d_\alpha \co \mathcal{F}_{\Sigma,\up} \to E$ such that 
$$
d_\alpha(u)=d(u) \quad \text{and} \quad d_\alpha(ww')=d_\alpha(w)\, \lact{\tilde{\alpha}(w)}{\!d_\alpha(w')}
$$
for all $u \in \up$ and $w,w'\in \mathcal{F}_{\Sigma,\up}$ with $t(w)=s(w')$. Note that  $d_\alpha(\emptyset)=1$ and $d_\alpha(u^{-1})=\lact{\alpha(u)^{-1}}{\!\!\left (d(u)^{-1}\right)}$.

Two $\CM$-labelings of $D$ are \emph{gauge equivalent} if they belong to the same orbit under the action of the gauge group $\tg_D$.

\begin{thm}\label{thm-gauge-group-labelings}
Let $M$ be a closed connected oriented 3-manifold and let $D$ be  an oriented pointed Heegaard diagram of $M$. 
The notion of a $\CM$-labeling of $D$ is well-defined, as well as the action of the gauge group~$\tg_D$ on the $\CM$-labelings of~$D$.
Moreover, there is a canonical bijection  
$$
[M,B\CM] \cong  \{ \text{$\CM$-labelings of $D$} \}/\tg_D
$$
between the set of homotopy classes of maps $M \to B\CM$ and the set of gauge equivalence classes of $\CM$-labelings of $D$. 
\end{thm}

\begin{proof}
We use the notation of Sections~\ref{sect-CM-from-Heegaard} and~\ref{sect-Taut}. Recall from Section~\ref{sect-Heegaard-diag-3man} that the embedding of the surface underlying $D$ into $M$ extends to a diffeomorphism $M_D\cong M$. This diffeomorphism transports the CW-decomposition of $M_D$ defined in Section~\ref{sect-Taut} into a  CW-decomposition of $M$,  whose associated skeletal filtration is denoted by~$M^*$.  It follows from Section~\ref{sect-Taut} that the free crossed complex $\Pi M^\ast$ (see Appendix~\ref{exa-crossed-complexes-of-CW-complexes}) is canonically isomorphic to the free crossed complex
$$
\Pi_{M,D}=\bigl ( \cdots \to 1 \to 1 \to  \pi_3(M,M^2) \xrightarrow{\eta\partial_3}  \mathcal{F}_D \xrightarrow{\nu_D}   \mathcal{F}_{\Sigma,\up}  \bigr).
$$
Consider the free basis
$$
\mathcal{B}_*=\bigl\{\mathcal{B}_3=\{h_\sigma\}_{\sigma \in \pi_0(\Sigma \setminus \low)},\, \mathcal{B}_2=\{b_l\}_{l \in \low},\, \mathcal{B}_1=\up\bigr\}
$$
of $\Pi_{M,D}$, where $\{\mathcal{B}_2, \mathcal{B}_1\}$ is the free basis of $\nu_D$ defined in Section~\ref{sect-CM-from-Heegaard} and $\mathcal{B}_3$ is the basis of $\pi_3(M,M^2)$ defined in Section~\ref{sect-Taut}. Note that the crossed module $\CM$ induces the crossed complex 
$$
\widetilde{\CM}= ( \cdots \to 1 \to 1 \to  E \xrightarrow{\CM }  H).
$$
By the homotopy classification theorem (see Appendix~\ref{defn-of-pointed-crossed-things}),  $[M,B\CM]$ is in canonical bijection with $[\Pi M^\ast,\widetilde{\CM}]$ and so with $[\Pi_{M,D},\widetilde{\CM}]$. (Here, we denote by $[C,C']$ the set of homotopy classes of morphisms of crossed complexes from $C$ to $C'$, see Appendix~\ref{sect-groupoid-of-homotopies}.) 
Since~$\widetilde{\CM}$ has a single object, it follows from Appendix~\ref{sect-free-crossed-complexes-def} that a morphism $f=\{f_n\}_{n \geq 1}$ of crossed complexes $\Pi M^\ast \to \widetilde{\CM}$ is fully determined by maps $\alpha \co \up \to H$ and  $\beta \co \low \to E$, with $\alpha(u)=f_1(u)$ and  $\beta(l)=f_2(b_l)$, such that for all $l \in \low$ and $\sigma \in \pi_0(\Sigma \setminus \low)$,
$$
\CM(\beta(l))=f_1(\nu_D(b_l)) \quad \text{and} \quad  f_2(\eta\partial_3(h_\sigma))=1.
$$
Since $f_1=\tilde{\alpha}$, $\nu_D(b_l)=\omega_l$, and $\eta\partial_3(h_\sigma)=\langle\tau_\sigma\rangle$, these two conditions correspond to Axioms (i)-(ii) of a $\CM$-labeling of~$D$. Consequently, there is a canonical bijective correspondence  between morphisms $\Pi M^\ast \to \widetilde{\CM}$ of crossed complexes and $\CM$-labelings of $D$. Now a homotopy between two morphisms $\Pi M^\ast \to \widetilde{\CM}$ is fully determined by its values on $\mathcal{B}_*$ and so by an element of~$\tg_D$. Moreover, the groupoid structure of homotopies and its action on  morphisms of crossed complexes correspond to the group structure of $\tg_D$ and its action on $\CM$\ti labelings of~$D$ (which is then well-defined). Hence there is a canonical  bijection between the set $[M,B\CM]$ and the set of gauge equivalence classes of $\CM$-labelings of $D$.

It remains to prove that Axiom (ii) in the definition of a $\CM$-labeling of $D$ does not depend on the choice of the representative set of taut identities for $D$. Any different choice of basepoints for the 3-cells of $M$ induces another basis $\mathcal{B}'_3=\{h'_\sigma\}_{\sigma \in \pi_0(\Sigma \setminus \low)}$ of the free $\pi_1(M)$-module $\pi_3(M,M^2)$. Let  $\tau$ and $\tau'$ be representative sets of taut identities  for $D$ with respect to $\mathcal{B}_3$ and $\mathcal{B}'_3$, respectively. Given $\sigma \in \pi_0(\Sigma \setminus \low)$, we need to prove that $f_2(\langle\tau'_\sigma\rangle)=1$ if and only if  $f_2(\langle\tau_\sigma\rangle)=1$. 
By considering a path connecting the basepoints of the 3-cell corresponding to $\sigma$ and using that 
 $\pi_1(M)\cong \mathcal{F}_{\Sigma,\up}/\mathrm{Im}(\nu_D)$, there exists $w \in \mathcal{F}_{\Sigma,\up}$ such that $h'_\sigma=\lact{w}{h}_\sigma$.
Then 
$$
\langle\tau'_\sigma\rangle=\eta\partial_3(h'_\sigma)=\eta\partial_3(\lact{w}{h}_\sigma)
=\lact{w}{\ns\bigl(}\eta\partial_3(h_\sigma)\bigr)=\lact{w}{\ns\langle}\tau_\sigma\rangle
$$
and so 
$$
f_2(\langle\tau'_\sigma\rangle)=f_2(\lact{w}{\ns\langle}\tau_\sigma\rangle)=\lact{f_1(w)}{\ns f_2}(\langle\tau_\sigma\rangle).
$$
We conclude using that $H$ acts on $E$ by group automorphisms.
\end{proof}

\subsection{Particular cases}\label{sect-partic-cases-thm1}
1. Given a group $G$, the group homomorphism $1 \to G$ is a crossed module and  $B(1\to G)=BG$ is the classifying space of $G$. We recover from Theorem~\ref{thm-gauge-group-labelings} that for any closed connected oriented 3-manifold $M$, the set $[M,BG]$ of homotopy classes of maps $M \to BG$ is in bijection with the set of conjugacy classes of group homomorphisms $\pi_1(M) \to G$. 

2. Let $M$ be a closed connected oriented 3-manifold and $E$ be an abelian group. Recall that $E \to 1$ is a crossed module. Pick an oriented pointed Heegaard diagram $D=(\Sigma,\up,\low)$ of $M$ such that both $\Sigma \setminus\up$ and $\Sigma \setminus\low$ are connected. Then the set of~$(E \to 1)$-labelings of $D$ is in bijective correspondence with the set $\mathrm{Map}(\low,E)$. The gauge group of $D$ is~$\tg_D=\mathrm{Map}(\up,E)$ with pointwise multiplication.  It acts on $\mathrm{Map}(\low,E)$ by  $d \cdot \beta=\beta'$ with $\beta'(l)=\tilde{d}(\omega_l)\beta(l)$ for all $l \in \low$, where  $\tilde{d}\co \mathcal{F}_{\Sigma,\up} \to E$ is the  group homomorphism extending $d$. The fact that  $B(E \to 1)$ is a $K(E,2)$-space and Theorem~\ref{thm-gauge-group-labelings} imply that there are bijections  
$$
H^2(M,E) \cong [M,B(E \to 1)] \cong \mathrm{Map}(\low,E)/\mathrm{Map}(\up,E).
$$

\subsection{Relation with 2-bundles and \v{C}ech cohomology}
Let $M$ be a closed connected oriented 3-manifold and $\tg_\CM$ be the $2$-group associated with the crossed module $\CM$ (see Appendix~\ref{sect-2-groups}).  Pick an oriented pointed Heegaard diagram $D$ of~$M$.    Consider the bijections given in Appendix~\ref{sect-flat} between the set $\pp(M,\tg_\chi)$ of equivalence classes of principal $\tg_\chi$-bundles over~$M$, the \v{C}ech cohomology $\check{H}(M,\tg_\CM)$ with coefficients in $\tg_\CM$, and the set $[M,B\CM]$.  Composing it with that of Theorem~\ref{thm-gauge-group-labelings}  yields canonical bijections
$$
\pp(M,\tg_\chi) \cong \check{H}(M,\tg_{\CM}) \cong \mathfrak{L}_D/\tg_D,
$$
where $\mathfrak{L}_D$ is the set of $\CM$-labelings of $D$. 

\subsection{Example}\label{ex-Lens-labelings}
Consider the lens space $L(p,q)$ and its Heegaard diagram $D$ given in Section~\ref{ex-Lens}. A $\CM$-labeling of $D$  is encoded by a pair $(x,e) \in H \times E$ such that $\CM(e)=x^p$ and $e^{-1} \bigl(\lact{x^q}{\!e}\bigr)=1$. Since $\lact{x^p}{\!e}=\lact{\CM(e)}{e}=e$ and the integers $p$ and $q$ are coprime, the latter condition is equivalent to $\lact{x}{e}=e$. Then the set of  $\CM$-labelings of~$D$ is in bijection with the set
$$
\mathfrak{L}_D=\bigl\{ (x,e) \in H \times E\; \big | \; \CM(e)=x^p \;\, \text{and} \;\, \lact{x}{e}=e \bigr\}.
$$ 
The gauge group of $D$ is $\tg_D=H \ltimes E$ with product $(a,d) (a',d')=\bigl(a a', (\lact{a'^{-1}}{\!\!d})d'\bigr )$. It acts on $\mathfrak{L}_D$ by 
$$
(a,d) \cdot (x,e)= \Bigl(a \CM(d) x a^{-1}, \lact{a}{\!\Bigl(}d \bigl(\lact{x}{\!d}\bigr) \cdots \bigl(\lact{x^{(p-1)}}{\!\!d}\bigr) e\Bigr) \! \Bigr).
$$
By Theorem~\ref{thm-gauge-group-labelings}, the set of homotopy classes of maps  $[L(p,q), B\CM]$ is in bijection with the set of orbits $\mathfrak{L}_D/\tg_D$. Denote by $[x,e]\in \mathfrak{L}_D/\tg_D$ the class of $(x,e)\in \mathfrak{L}_D$ and by $g_{[x,e]} \in [L(p,q), B\CM]$ the corresponding homotopy class, so that
$$
[L(p,q), B\CM]=\bigl\{ g_{[x,e]} \, \big | \,  [x,e]\in \mathfrak{L}_D/\tg_D\bigr \}.
$$
Note that when $H$ is trivial, and so $E$ is abelian, it follows from Section~\ref{sect-partic-cases-thm1}  that
$$
H^2(L(p,q),E)\cong E/E^p
$$
where $E^p=\{e^p \, | \, e \in E\}$ acts on $E$ by left multiplication.

\subsection{Example}\label{ex-Lens-labelings-Z4Z-22Z}
Given a lens space $L(p,q)$, let us describe the computations of Section~\ref{ex-Lens-labelings} for the  following explicit crossed module. (We keep notation from that section.) Consider the (additive) group  $\Z/2\Z=\{0,1\}$ which acts on the (additive) group $\Z/4\Z=\{\bar{0},\bar{1},\bar{2},\bar{3}\}$ by $\lact{0}{\bar{n}}=\bar{n}$ and $\lact{1}{\bar{n}}=-\bar{n}$.  Then the trivial group homomorphism $\CM\co \Z/4\Z \to \Z/2\Z$, sending $\bar{n}$ to $0$, is a crossed module (see the third example of Section~\ref{sect-crossed-modules-ex}). Then
$$
\mathfrak{L}_D=\bigl\{ (x,\bar{n}) \in \Z/2\Z \times \Z/4\Z\; \big | \; px=0 \;\, \text{and} \;\, (-1)^x \bar{n}=\bar{n} \bigr\}
$$ 
and the group $\tg_D=\Z/2\Z \ltimes \Z/4\Z$ has product $(a,\bar{d}) (a',\bar{d'})=\bigl(a+a', \overline{(-1)^{a'} d+ d'}\, \bigr )$ and acts on $\mathfrak{L}_D$ by 
$$
(a,\bar{d}) \cdot (x,\bar{n})= \bigl(x , \overline{(-1)^{a}(\kappa_x d+n)}\, \bigr) \;\, \text{with} \;\, \kappa_0=p \;\, \text{and} \;\, \kappa_1=(1-(-1)^p)/2.
$$
It follows that
$$
\big |[L(p,q), B\CM] \big |=\big |\mathfrak{L}_D/\tg_D \big |=
\left\{\!\!\begin{array}{cl}
5 & \text{if $p\equiv 0 \, (\mathrm{mod} \, 4)$,} \\
4 & \text{if $p\equiv 2 \, (\mathrm{mod} \, 4)$,} \\
1 & \text{otherwise,} 
\end{array} \right.
$$
and
$$
[L(p,q), B\CM]=
\left\{\!\!\begin{array}{ll}
\bigl\{ g_{[0,\bar{0}]}, g_{[0,\bar{1}]}, g_{[0,\bar{2}]}, g_{[1,\bar{0}]}, g_{[1,\bar{2}]} \bigr \} & \text{if $p\equiv 0 \, (\mathrm{mod} \, 4)$,} \\
\bigl\{ g_{[0,\bar{0}]}, g_{[0,\bar{1}]}, g_{[1,\bar{0}]}, g_{[1,\bar{2}]} \bigr \} & \text{if $p\equiv 2 \, (\mathrm{mod} \, 4)$,} \\
\bigl\{ g_{[0,\bar{0}]} \bigr \} & \text{otherwise.} 
\end{array} \right.
$$

\subsection{Example}\label{ex-Poincare-labelings}
Consider the Poincar\'e homology sphere $\mathbb{P}$ and its Heegaard diagram~$D$ given in Section~\ref{ex-Poincare}. The set of  $\CM$-labelings of~$D$ is in bijection with the set $\mathfrak{L}_D$ of tuples $(x,y,e,f) \in H^2 \times E^2$ satisfying
$$
\CM(e)=x^4 y^{-1} x^{-1} y^{-1}, \;\; \CM(f)=x^{-1} y^{-1}x^{-1} y^2, \;\; \bigl(e^{-1}\bigr) \bigl(\lact{x^{-1}}{\!\!e}\bigr) \bigl(\lact{x^{-1}y^3}{\!\!f}\bigr)^{-1}\bigl(\lact{yx}{\!f}\bigr)=1.
$$
The gauge group of $D$ is $\tg_D=H \ltimes E^2$ with product 
$$
(a,d,l) (a',d',l')=\bigl(a a', (\lact{a'^{-1}}{\!\!d})d', (\lact{a'^{-1}}{\!l})l'\bigr ).
$$ 
It acts on $\mathfrak{L}_D$ by $(a,d,l)  \cdot (x,y,e,f)=\bigl(a \CM(d) x a^{-1}, a \CM(l) y a^{-1}, \lact{a}{(}\omega e), \lact{a}{(}\omega'f)\bigr)$ where
\begin{align*}
& \omega=  d\,  \bigl(\lact{x}{\!d}\bigr)  \bigl(\lact{x^2}{\!\!d}\bigr) \bigl(\lact{x^3}{\!\!d}\bigr) \bigl(\lact{x^4y^{-1}}{\!l}\bigr)^{-1} \bigl(\lact{x^4y^{-1}x^{-1}}{\!\!d}\bigr)^{-1} \bigl(\lact{x^4y^{-1}x^{-1}y^{-1}}{\!l}\bigr)^{-1}, \\
& \omega'=  \bigl(\lact{x^{-1}}{\!\!d}\bigr)^{-1}  \bigl(\lact{x^{-1}y^{-1}}{\!l}\bigr)^{-1} \bigl(\lact{x^{-1}y^{-1}x^{-1}}{\!\!d}\bigr)^{-1} \bigl(\lact{x^{-1}y^{-1}x^{-1}}{\!l}\bigr) \bigl(\lact{x^{-1}y^{-1}x^{-1}y}{l}\bigr).
\end{align*}
By Theorem~\ref{thm-gauge-group-labelings}, the set of homotopy classes of maps  $[\mathbb{P}, B\CM]$ is in bijection with the set of orbits $\mathfrak{L}_D/\tg_D$.

\subsection{Colored Heegaard diagrams}
A \emph{$\CM$-Heegaard diagram}  is a Heegaard diagram endowed with a $\CM$-labeling (in the sense of  Section~\ref{sect-Xi-labelings}). Any $\CM$-Heegaard diagram $(D,(\alpha,\beta))$ defines a $\CM$-manifold $(M_D,g_{\alpha,\beta})$, where $M_D$ is the 3-manifold associated with $D$ (see Section~\ref{sect-Heegaard-diag-3man})  and  $g_{\alpha,\beta} \in[M_D,B\CM]$ is the homotopy class induced by the $\CM$-labeling $(\alpha,\beta)$ of $D$ as in Theorem~\ref{thm-gauge-group-labelings}.

A \emph{$\CM$-Heegaard diagram of a $\CM$-manifold} $(M,g)$ is a $\CM$-Heegaard diagram $(D,(\alpha,\beta))$ such that $D$ is a Heegaard diagram of $M$ and the orientation preserving diffeomorphism $M_D \cong M$ (induced by the  embedding of the splitting surface into $M$) is an equivalence between $(M_D,g_{\alpha,\beta} )$ and $(M,g)$.  It follows from Theorem~\ref{thm-gauge-group-labelings} that every  $\CM$-manifold has a $\CM$-Heegaard diagram.

\subsection{Colored Heegaard moves}\label{sect-Xi-moves}
The Reidemeister-Singer theorem asserts that two Heegaard diagrams give rise to diffeomorphic oriented 3-manifolds if and only if they are related by a finite sequence of the following Heegaard moves (and their inverses):   diffeomorphism of the surface, isotopy of the diagram (two-point move), stabilization, and handle slide (sliding a circle past another). We extend this result to $\CM$-Heegaard diagrams and $\CM$-manifolds (see Theorem~\ref{thm-colored-Reidemeister} below). To this end, we first introduce a colored version of the Heegaard moves. Each of the moves below transforms a $\CM$-Heegaard diagram $(D=(\Sigma, \up, \low),(\alpha,\beta))$ into another $\CM$-Heegaard diagram  $(D',(\alpha',\beta'))$. 

\smallskip

\begin{mylist}{iii}
\item[i] \textbf{Orientation preserving diffeomorphisms of $\Sigma$.} Let $\psi \colon \Sigma \to \Sigma$ be an orientation preserving diffeomorphism. The resulting $\CM$-Heegaard diagram is $(D'=\psi(D), (\alpha',\beta'))$, where $\psi(D)=(\psi(\Sigma),\psi(\up),\psi(\low))$ and the labels are transferred via the diffeomorphism: for all $u \in \up$ and $l \in \low$,
        $$\alpha'(\psi(u))= \alpha(u) \quad \text{and} \quad \beta'(\psi(l))=\beta(l).$$ 
        
\item[ii] \textbf{Moving basepoints.} This move consists of moving the basepoint of a lower circle $l$ across an intersection point   $s$ by following its orientation:
$$
 \psfrag{v}[Bc][Bc]{\scalebox{.9}{$l$}}
 \psfrag{s}[Bl][Bl]{\scalebox{.9}{$s$}}
 \rsdraw{.49}{.9}{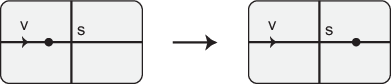} \;\;.
$$
The new label of~$l$~is
        $$
\beta'(l)=\lact{\alpha(u)^{-\nu(s)}}{\!\beta}(l)
        $$
    where $\nu_s \in \{1,-1\}$ is the sign of $s$ (see Section~\ref{sect-Heegaard-diag}). The labels of all other circles remain unchanged.
    
\item[iii] \textbf{Orientation reversal.} This move reverses the orientation of an upper or a lower circle and replaces its label by its inverse.
     
\item[iv] \textbf{Two-point move.} This move is an isotopy (in general position and away from the basepoints) between an upper circle $u \in \up$ and a lower circle $l \in \low$:
$$
 \psfrag{u}[Br][Br]{\scalebox{.9}{$u$}}
 \psfrag{v}[Br][Br]{\scalebox{.9}{$l$}}
 \rsdraw{.49}{.9}{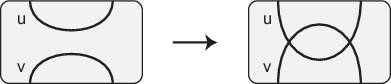} \;\;.
$$
Throughout this move, the labels of all circles remain unchanged.      Note that one of the added intersection points is positive and the other one is negative.

\item[v] \textbf{Stabilization.} 
This move removes a disk from $\Sigma$ which is disjoint from all upper and lower circles and replaces
  it by a punctured torus with one upper circle and one lower circle. One of them corresponds to the standard
  meridian and the other to the standard longitude of the added torus:
$$
 \psfrag{D}[Bc][Bc]{$D$}
 \rsdraw{.49}{.9}{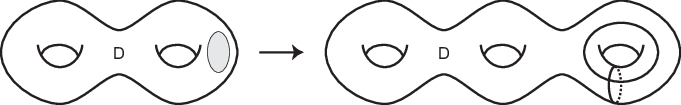} \;\;.
$$
The added lower circle is  endowed with arbitrary orientation and basepoint. The added upper circle is oriented in such a way that the added intersection point is positive. The added lower circle is labeled by an element $e \in E$ and the added upper circle  is labeled by $\CM(e) \in H$.
    
\item[vi] \textbf{Handle slide of upper circles.} This move slides an upper circle $u_1$ over another upper circle $u_2$. Pick a band~$b$ on $\Sigma$ which connects $u_1$ to $u_2$ and does not cross any other circle. We assume that the connected components~$c^{u_1}_-$ and~$c^{u_2}_-$ agree (see Section~\ref{sect-CM-from-Heegaard}) and that the band $b$ lies in this component. (This is always possible after orientation reversal of the circles.) The circle $u_1$ is replaced by the band sum $u_1' = u_1 \#_b u_2$ and inherits the orientation from~$u_1$. The circle $u_2$ is replaced by a copy $u'_2$ of itself (with the same orientation) which is slightly isotoped such that it has no point in common with~$u'_1$:
 $$
 \psfrag{C}[Bl][Bl]{\scalebox{.9}{$u_1$}}
 \psfrag{D}[Bl][Bl]{\scalebox{.9}{$u_2$}}
 \psfrag{U}[Bl][Bl]{\scalebox{.9}{$u'_1$}}
 \psfrag{V}[Br][Br]{\scalebox{.9}{$u'_2$}}
 \psfrag{b}[Bc][Bc]{$b$}
 \rsdraw{.49}{.9}{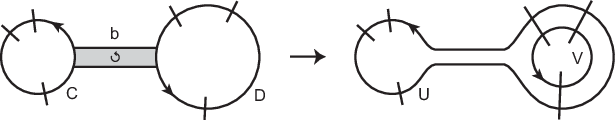}\;\; .
$$ 
Here the small arcs represent parts of the lower circles intersecting $u_1$ and~$u_2$ and the orientation of $b$ is inherited from $\Sigma$.      
The labels of new circles are
$$
\alpha'(u_1')=\alpha(u_1) \quad \text{and} \quad \alpha'(u_2')=\alpha(u_1)^{-1}\alpha(u_2).
$$
The labels of all other circles remain unchanged.

\item[vii] \textbf{Handle slide  of lower circles.} This move slides a lower circle $l_1$ over another lower circle $l_2$. Pick a band~$b$ on $\Sigma$ which connects $l_1$ to $l_2$  and does not cross any other circle. We assume that $b$ connects $l_1$ to $l_2$ just before their basepoints and that the orientations of $l_1$ and $l_2$ induce the same orientation on their band sum along~$b$. (This is always possible after moving basepoints and/or orientation reversal of the circles.) The circle $l_1$ is replaced by the band sum $l_1' = l_1 \#_b l_2$ and inherits the orientation and  basepoint from~$l_1$. The circle~$l_2$ is replaced by a copy $l'_2$ of itself (with the same orientation and basepoint) which is slightly isotoped such that it has no point in common with~$l'_1$:
 $$
 \psfrag{C}[Bl][Bl]{\scalebox{.9}{$l_1$}}
 \psfrag{D}[Bl][Bl]{\scalebox{.9}{$l_2$}}
 \psfrag{U}[Bl][Bl]{\scalebox{.9}{$l'_1$}}
 \psfrag{V}[Br][Br]{\scalebox{.9}{$l'_2$}}
 \psfrag{b}[Bc][Bc]{$b$}
 \rsdraw{.49}{.9}{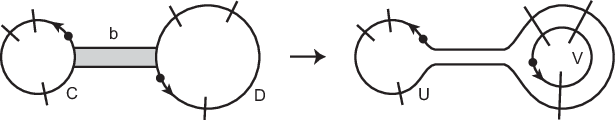}\;\; .
$$       
Here the small arcs represent parts of the upper circles intersecting $l_1$ and $l_2$. The labels of new circles are
$$
\beta'(l_1')=\beta(l_1)\beta(l_2) \quad \text{and} \quad \beta'(l_2')=\beta(l_2).
$$
The labels of all other circles remain unchanged.

\item[viii] \textbf{Adding a trivial circle.} By a trivial circle, we mean a circle that bounds a disk embedded in $\Sigma$ disjoint from all the upper and lower circles of $D$. This move adds a trivial oriented upper circle labeled by an arbitrary element of~$H$, or a trivial oriented pointed lower circle labeled by the unit $1 \in E$.\\[-.8em]
\end{mylist}
The moves (i)-(viii) and their obvious inverses (such as destabilization or removal of a trivial circle) are called \emph{$\CM$-moves}.

\begin{thm}\label{thm-colored-Reidemeister}
The $\CM$-moves are well-defined (that is, transform a $\CM$-Heegaard diagram into a $\CM$-Heegaard diagram). Moreover 
two $\CM$-Heegaard diagrams represent equivalent $\CM$-manifolds if and only if they are related by a finite sequence of $\CM$-moves.
\end{thm}

We prove Theorem~\ref{thm-colored-Reidemeister} in the next section.  In view of Section~\ref{sect-flat-2-bundles}, Theorem~\ref{thm-colored-Reidemeister} gives a way to represent $\CM$-bundles via $\CM$-Heegaard diagrams.

\subsection{Remark}\label{sect-rem-connected}
A  Heegaard  diagram $D=(\Sigma,\up,\low)$ is \emph{connected} if both $\Sigma \setminus\up$ and~$\Sigma \setminus\low$ are connected. 
Contrary to the $\CM$-move (viii), the $\CM$-moves (i)-(vii) and their inverses transform a connected $\CM$-Heegaard diagram into a connected $\CM$-Heegaard diagram. They are called \emph{connected $\CM$-moves}. Any $\CM$-manifold has a connected $\CM$-Heegaard diagram. By Theorem~\ref{thm-colored-Reidemeister}, if two connected $\CM$-Heegaard diagrams are related by a finite sequence of connected $\CM$\ti moves, then they represent equivalent $\CM$-manifolds. But the converse is false: connected $\CM$\ti Heegaard diagrams of equivalent $\CM$-manifolds are not necessarily related by connected $\CM$\ti moves.

\subsection{Proof of Theorem~\ref*{thm-colored-Reidemeister}}\label{sect-proof-colored-Reidemeister}
First, let us prove that $\CM$-moves are well-defined. We have to verify that the resulting pair $(\alpha',\beta')$ of each $\CM$-move is a $\CM$-labeling. This is obvious for the $\CM$-moves~(i), (iv), and (viii) for upper circles since under these moves, the elements $\omega_l\in \mathcal{F}_{\Sigma,\up}$ with $l \in \low$ and (an appropriate choice of) taut identities remain unchanged. Adding a trivial lower circle $l$ adds the element $\omega_l=1$, adds a taut identity~$(1,l)^{\pm 1}$, and inserts $(r,l)^{\pm 1}$ in the taut identity associated with the component of $\Sigma \setminus \low$ where the circle is added. These transformations are innocuous since the label of~$l$ is $1 \in E$. Hence the $\CM$-move (viii) for lower circles transforms a $\CM$-labeling into a $\CM$-labeling.

The $\CM$-move (ii) consists of moving the basepoint of a lower circle $l$ across an intersection point $s$ (lying on an upper circle $u$) by following its orientation. Denote by $\nu$ the sign of $s$ (that is, $\nu=1$ if $s$ is positive and $\nu=-1$ otherwise). The element $\omega_l=u^{\nu} w$ is transformed into  $\omega'_l=w\, u^{\nu}$ and a taut identity $\prod_{k=1}^n (r_k,l_k)^{\varepsilon_k}$ is transformed into $\prod_{k=1}^n (r_k',l_k)^{\varepsilon_k}$, where $r'_k=r_ku^\nu$ if $l_k=l$ and $r'_k=r_k$ otherwise.  The fact that 
$$
\beta'(l)=\lact{\alpha(u)^{-\nu}}{\!\beta}(l) \quad \text{and} \quad
 \lact{\tilde{\alpha}'(r'_k)}{\beta'}(l_k)= \lact{\tilde{\alpha}(r_k)}{\beta}(l_k)
$$
for all $k$ imply that $(\alpha',\beta')$ is a $\CM$-labeling.

Reversing the orientation of an upper circle $u$ transforms each $\omega_l$ with $l \in \low$ and each taut identity $\tau=\prod_{k=1}^n (r_k,l_k)^{\varepsilon_k}$ by replacing $u$ with $u^{-1}$ in $\omega_l$ and~$r_k$. Likewise, reversing the orientation of a lower circle $l$ transforms $\omega_l$ into $\omega_l^{-1}$, leaves the elements $\{\omega_{l'}\}_{l' \in \low\setminus\{l\}}$ unchanged, and transforms each taut identity $\tau=\prod_{k=1}^n (r_k,l_k)^{\varepsilon_k}$ by replacing $\varepsilon_k$ with $\varepsilon'_k$,  where $\varepsilon'_k=-\varepsilon_k$ if $l_k=l$  and $\varepsilon'_k=\varepsilon_k$ otherwise. These transformations are compensated by the changes of labels. Hence the $\CM$-move (iii) transforms a $\CM$-labeling into a $\CM$-labeling.

Stabilization introduces a lower circle  $l$ and an upper circle $u$. Then $\omega_l=u$ and this move inserts $(r,l)(r,l)^{-1}$ in the taut identity associated with the component of~$\Sigma \setminus \low$ where the stabilization is performed. Since the label of $u$ is the image under~$\CM$ of the label of $l$, we conclude that the $\CM$-move (v)  transforms a $\CM$-labeling into a $\CM$-labeling.

By Section~\ref{sect-Taut}, any oriented pointed Heegaard diagram $D=(\Sigma, \up, \low)$ gives rise to a free crossed complex 
$$
\Pi_{D}=\bigl ( \cdots \to 1 \to 1 \to  \pi_3(M_D,X_D) \xrightarrow{\eta\partial_3}  \mathcal{F}_D \xrightarrow{\nu_D}   \mathcal{F}_{\Sigma,\up}  \bigr)
$$
with free basis 
$\mathcal{B}_*=\{\mathcal{B}_3,\mathcal{B}_2,\mathcal{B}_1\}$, where $\mathcal{B}_2=\{b_l\}_{l \in \low}$ and $\mathcal{B}_1=\up$ (see Section~\ref{sect-CM-from-Heegaard}). 
Recall from the proof of Theorem~\ref{thm-gauge-group-labelings} that there is a bijective correspondence  between morphisms of crossed complexes $\Pi_D \to \widetilde{\CM}= ( \cdots \to 1 \to 1 \to  E \xrightarrow{\CM }  H)$ and $\CM$-labelings of~$D$. In this correspondence,  the $\CM$-labeling $(\alpha,\beta)$ associated with a morphism $f \co \Pi_D \to \widetilde{\CM}$ is given by $\alpha(u)=f_1(u)$ for $u \in \up$ and  $\beta(l)=f_2(b_l)$  for $l \in\low$. Consider now a handle slide from $D$ to $D'$. It induces  an orientation preserving diffeomorphism $M_D \to M_{D'}$ which preserves their filtration defined in Section~\ref{sect-Taut}. This diffeomorphism then induces an   isomorphism between the crossed complexes $\theta \co \Pi_{D} \to \Pi_{D'}$ and so between the sets of $\CM$-labelings of $D$ and $D'$. The image under this bijection of a $\CM$-labeling $(\alpha,\beta)$ of $D$ is the pair $(\alpha',\beta')$ described in the corresponding handle slide $\CM$-move, which is then a $\CM$-labeling of $D'$.
For the $\CM$-move (vi), this follows from the fact $\theta_2(b_l)=b_l$  for all  $l \in \low$ and 
$$
\theta_1(u_1)=u'_1, \quad \theta_1(u_2)=u'_1 u'_2, \quad \theta_1(u)=u \, \text{ for all $u \in \up \setminus \{u_1,u_2\}$.}
$$
For the $\CM$-move (vii), this follows from the fact that $\theta_1(u)=u$  for all  $u \in \up$ and 
$$
\theta_2(b_{l_1})=b_{l'_1}b_{l'_2}^{-1}, \quad \theta_2(b_{l_2})=b_{l'_2}, \quad \theta_2(b_l)=b_l \, \text{ for all $l \in \low \setminus \{l_1,l_2\}$.}
$$
This proves that the $\CM$-moves (vi) and (vii)  transform a $\CM$-labeling into a $\CM$-labeling.

Second, let us prove that two $\CM$-Heegaard diagrams related by a $\CM$-move represent equivalent $\CM$-manifolds. Given an oriented pointed Heegaard diagram $D$, we denote by  $g_{\alpha,\beta} \in [M_D,B\CM]$ the homotopy class of maps associated with the gauge class of a $\CM$-labeling $(\alpha,\beta)$ of $D$  under the canonical bijection of Theorem~\ref{thm-gauge-group-labelings}.
Let us prove that if a  $\CM$-Heegaard diagram $(D',(\alpha',\beta'))$ is obtained from a $\CM$-Heegaard diagram $(D,(\alpha,\beta))$ by applying a $\CM$-move $T$, then $g_{\alpha',\beta'}\circ  \psi_T=g_{\alpha,\beta}$, where $\psi_T \co M_D \to M_{D'}$ is the diffeomorphism canonically associated (up to isotopy) with the Heegaard move underlying $T$.  This diffeomorphism $\psi_T$ is homotopic to a cellular map which induces a morphism of crossed complexes $\theta_T \co \Pi_{D} \to \Pi_{D'}$ such that the following diagram commutes:
$$
\xymatrix@R=.5cm @C=1.2cm{
[\Pi_{D'}, \widetilde{\CM}] \ar@{->}[r]^-{? \circ\theta_T}\ar@{=}[d]^-{\mathbin{\rotatebox[origin=c]{-90}{$\sim$}}}      & [\Pi_{D}, \widetilde{\CM}] \ar@{=}[d]^-{\mathbin{\rotatebox[origin=c]{-90}{$\sim$}}}     \\
[M_{D'},B\CM] \ar@{->}[r]^-{? \circ\psi_T}  &  [M_{D},B\CM].
}
$$
Here, the square brackets in the first row denote the set of homotopy classes of morphisms of crossed complexes (see Appendix~\ref{sect-groupoid-of-homotopies}) and the vertical bijections are given by the homotopy classification theorem (see Appendix~\ref{defn-of-pointed-crossed-things}). 
As above, consider the morphisms  of crossed complexes $f_{\alpha,\beta} \co \Pi_D \to \widetilde{\CM}$ and  $f_{\alpha',\beta'} \co \Pi_{D'} \to \widetilde{\CM}$ associated with the $\CM$-labelings $(\alpha,\beta)$ and $(\alpha',\beta')$. These morphisms are mapped to $g_{\alpha,\beta}$ and $g_{\alpha',\beta'}$ under the vertical bijections. Now it follows from the definition of the change of labelings for the $\CM$-move $T$ (see Section~\ref{sect-Xi-moves}) that these morphisms are related by $[f_{\alpha',\beta'}\circ  \theta_T]=[f_{\alpha,\beta}]$. By the commutativity of the diagram, we deduce that $g_{\alpha',\beta'}\circ  \psi_T=g_{\alpha,\beta}$.

Third, let us prove that $\CM$-Heegaard diagrams representing equivalent $\CM$-manifolds are related by a finite sequence of $\CM$-moves. Let $(D,(\alpha,\beta))$ and $(D',(\alpha',\beta'))$ be two $\CM$-Heegaard diagrams whose associated $\CM$-manifolds $(M_D, g_{\alpha,\beta})$ and $(M_{D'},g_{\alpha',\beta'})$ are equivalent. By definition, there is an orientation preserving diffeomorphism $\psi \colon M_{D'} \to M_{D}$ such that $g_{\alpha',\beta'}= g_{\alpha,\beta} \circ \psi $. By the Reidemeister-Singer theorem, the  diffeomorphism $\psi$ is isotopic to a composition $\psi_{T_n} \circ \dots \circ \psi_{T_1}$ where 
$$
D'=D_0 \xrightarrow{T_1} D_1 \to \dots \xrightarrow{T_n} D_n=D
$$
is a sequence of (standard) Heegaard moves. Extend each $T_i$ to a $\CM$-move to obtain a sequence of  $\CM$-moves
\begin{gather*}
(D', (\alpha',\beta'))=(D_0,(\alpha_0,\beta_0)) \xrightarrow{T_1} (D_1,(\alpha_1,\beta_1)) \xrightarrow{T_2} \cdots \\ \cdots \xrightarrow{T_{n-1}} (D_{n-1},(\alpha_{n-1},\beta_{n-1})) \xrightarrow{T_n} (D_n,(\alpha_n,\beta_n))= (D,(\alpha'',\beta'')).
\end{gather*}
It remains to prove that $(D,(\alpha'',\beta''))$ and $(D,(\alpha,\beta))$ are also related by finitely many $\CM$-moves. 
By the  above, $g_{\alpha_{i-1},\beta_{i-1}}=g_{\alpha_i,\beta_i} \circ \psi_{T_i}$ for all $1 \leq i \leq n$. Then
$$
g_{\alpha,\beta} =g_{\alpha',\beta'}\circ \psi^{-1}=g_{\alpha'',\beta''} \circ \psi_{T_n} \circ \dots \circ \psi_{T_1} \circ \psi^{-1}=g_{\alpha'',\beta''}.
$$
Consequently, by Theorem~\ref{thm-gauge-group-labelings}, the $\CM$-labelings $(\alpha,\beta)$ and $(\alpha'',\beta'')$ of $D$ are gauge equivalent: there is $h \in \tg_D$ such that $(\alpha'',\beta'')=h\cdot(\alpha,\beta)$. We need to prove that the  $\CM$-Heegaard diagram $(D,h\cdot(\alpha,\beta))$ is obtained by applying to $(D,(\alpha,\beta))$ a finite sequence of $\CM$-moves. It suffices to verify this for the elements of the following generating set of the gauge group $\tg_D$:
$$
\bigl\{ (a_{c,x},1) \, \big| \, c \in \pi_0(\Sigma \setminus \up) \text{ and } x \in H \bigr\} \cup \bigl\{ (1,d_{u,e}) \, \big| \, u \in \up  \text{ and } e\in E \bigr\}, 
$$
where the map $a_{c,x}\co \pi_0(\Sigma \setminus \up) \to H$ assigns $x$ to $c$ and $1$ to the other components, and the map $d_{u,e} \colon \up \to E$ assigns $e$ to $u$ and $1$ to the other upper circles.\\

\pfpart{The case of $(a_{c,x},1)$} First add a trivial upper circle $u_0$ labeled by $x$ inside the component $c$. For each component of the boundary of $c$, slide $u_0$ over the upper circle corresponding to this component. Lastly isotop $u_0$ back to its initial position and remove it. The handle slide rule for upper circles ensures that the resulting $\CM$-Heegaard diagram is $(D,(a_{c,x},1) \cdot (\alpha,\beta))$. \\  

\pfpart{The case of $(1,d_{u,e})$} Apply the stabilization $\CM$-move in  $c^u_-$  by introducing a new lower circle $l_s$ labeled by $e$ and a new upper circle $u_s$ labeled by $\CM(e)$. By picking an arbitrary basepoint for $u$, the set of intersection points between $u$ and the lower circles is totally ordered by enumerating its elements  starting from this basepoint and following the orientation of $u$. Following this order, slide over $l_s$ the lower circles corresponding to the intersection points lying in $u$. Next  slide $u_s$ over $u$ and call the resulting circle $u_s'$. Finally, apply the destabilization $\CM$-move to remove the pair  $(u_s',l_s)$ and isotop the resulting $\CM$-Heegaard diagram to obtain $(D,(1,d_{u,e}) \cdot(\alpha,\beta))$.

\section{Hopf crossed module-coalgebras}\label{sect-Hopf-crossed-module-coalgebras}
Throughout this section, we let $H$ be a group (with neutral element $1$) and $\chi \co E \to H$ be a crossed module. We first review the definitions and basic properties of Hopf $H$-coalgebras and Hopf $\CM$-coalgebras (referring to \cite{Vi2} and \cite{SV2} for details). Then we study integrals of  involutory Hopf $\chi$-coalgebras of finite type.

\subsection{Hopf group-coalgebras}\label{sect-Hopf-group-coalgebras} 
A  \emph{Hopf $H$-coalgebra} (over the field $\Bbbk$) is a family of $\kk$-algebras $A = \{ A_x \}_{x \in H}$ endowed with a family of algebra homomorphisms  $\Delta=\{\Delta_{x,y} \colon A_{xy} \to A_x \otimes A_y\}_{x,y \in H}$ (called the \emph{coproduct}), an algebra homomorphism $\varepsilon \colon A_1 \to \kk$ (called the \emph{counit}),
and  a family of $\kk$-linear isomorphisms $S=\{S_x \colon A_{x^{-1}} \to A_x\}_{x \in H}$ (called the \emph{antipode}), which satisfy:
\begin{itemize}
    \item The coproduct is coassociative and counitary: for all $x,y,z \in H$, 
\begin{gather*}
(\Delta_{x,y}\otimes \id_{A_z}) \Delta_{xy,z}=(\id_{A_x} \otimes \Delta_{y,z}) \Delta_{x,yz},\\
(\id_{A_x} \otimes \varepsilon) \Delta_{x,1}=\id_{A_x}=(\varepsilon \otimes \id_{A_x}) \Delta_{1,x}.
\end{gather*}
    \item The antipode is convolution-inverse to the identities: for all $x \in H$,
    $$
\mu_x (S_x \otimes \id_{A_x}) \Delta_{x^{-1},x} =   \eta_x \varepsilon  =  \mu_x (\id_{A_x} \otimes S_x) \Delta_{x,x^{-1}},
$$
where $\mu_x\co A_x \otimes A_x \to A_x$ is the product of $A_x$ and $\eta_x \colon \Bbbk \to A_x$ is its unit map, that is, the $\kk$-linear morphism sending $1_\kk$ to the unit element $1_x$ of~$A_x$.
\end{itemize}
These axioms imply that the antipode is anti-multiplicative: for all $x \in H$,
$$
S_x \mu_{x^{-1}}=\mu_x\sigma_{A_x,A_x}(S_x \otimes S_x) \quad \text{and} \quad S_x(1_{x^{-1}})=1_x,
$$
and anti-comultiplicative:  for all $x,y \in H$,
$$
\Delta_{x,y} S_{xy}=(S_x \otimes S_y)\sigma_{A_{y^{-1}},A_{x^{-1}}} \Delta_{y^{-1},x^{-1}} \quad \text{and} \quad  \varepsilon S_1=\varepsilon.
$$
Here and below,  for $\kk$-vector spaces $U$ and $V$, the flip $\sigma_{U,V}\co U\otimes V \to V \otimes U$ is the $\kk$-linear isomorphism defined by $\sigma_{U,V}(u \otimes v)=v \otimes u$ for all $u \in U$ and $v \in V$.

Note that $A_1$ is a (usual) Hopf algebra (over $\kk$) with coproduct $\Delta_{1,1}$, counit~$\varepsilon$, and antipode $S_1$. Also, if the group  $H$ is finite, then the direct sum $\oplus_{x \in H} A_x$ is a Hopf algebra with coproduct and antipode given on $A_x$ by $\sum_{x \in H} \Delta_{h,h^{-1}x}$ and~$S_{x^{-1}}$.

\subsection{Graphical conventions}\label{sect-graphical-conventions}
We represent $\kk$-linear morphisms by diagrams to be read from bottom to top (with the composition/tensor product of morphisms consisting 
in vertical/horizontal stacking of diagrams). The product $\mu_x$, unit map~$\eta_x$,
coproduct $\Delta_{x,y}$,  counit $\varepsilon$,  and antipode $S_x$ of a Hopf $H$-coalgebra $A=\{A_x\}_{x \in H}$ are depicted as follows:
$$
\psfrag{a}[Br][Br]{\scalebox{.8}{$x$}}
\psfrag{x}[Bl][Bl]{\scalebox{.8}{$x$}}
\psfrag{y}[Br][Br]{\scalebox{.8}{$y$}}
\psfrag{n}[Bl][Bl]{\scalebox{.8}{$y$}}
\psfrag{z}[Bl][Bl]{\scalebox{.8}{$xy$}}
\psfrag{u}[Bl][Bl]{\scalebox{.8}{$1$}}
\psfrag{c}[Bl][Bl]{\scalebox{.8}{$x^{-1}$}}
\mu_x=\!\rsdraw{.45}{.9}{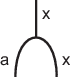} \qquad
\eta_x=\,\rsdraw{.45}{.9}{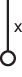} \qquad
\Delta_{x,y}=\!\rsdraw{.45}{.9}{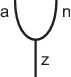}\qquad
\varepsilon=\,\rsdraw{.45}{.9}{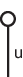} \qquad
S_x=\;\,\rsdraw{.45}{.9}{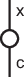}\;.
$$
Here the strand colors $x,y \in H$ are abbreviations for $A_x$ and $A_y$. The flip $\sigma_{A_x,A_y} $ is depicted as
$$
\psfrag{Z}[Bl][Bl]{\scalebox{.8}{$x$}}
\psfrag{X}[Br][Br]{\scalebox{.8}{$x$}}
\psfrag{A}[Bl][Bl]{\scalebox{.8}{$y$}}
\psfrag{Y}[Br][Br]{\scalebox{.8}{$y$}}
\sigma_{A_x,A_y}=\,\; \rsdraw{.45}{.9}{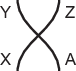} \,.
$$
For any $n \in \N$ and $x \in H$, define recursively the $n$-fold product 
$\mu^n_x \co A_x^{\otimes n} \to A_x$ by setting:
$$
\mu^0_x=\eta_x, \qquad \mu^1_x=\id_{A_x},  \qquad \mu^{n+1}_x=\mu_x(\mu^n_x \otimes \id_{A_x}).
$$
The associativity and unitality of the product imply that
$$
\mu^{n_1+\cdots +n_k}_x=\mu^k_x(\mu^{n_1}_x \otimes \cdots \otimes \mu^{n_k}_x).
$$
Similarly, we set $\Delta^0_1=\varepsilon\co A_1 \to \kk$ and, for $n \in \N_{\geq 1}$ and $x_1, \dots, x_n \in H$, we define recursively the $n$-fold coproduct 
$\Delta^n_{x_1, \dots, x_n} \co A_{x_1 \cdots x_n} \to A_{x_1} \otimes \cdots \otimes A_{x_n}$ by setting:
$$
\Delta^1_x=\id_{A_x} \quad \text{and} \quad \Delta^{n+1}_{x_1,x_2, \dots, x_{n+1}}=(\Delta^n_{x_1, \dots, x_n} \otimes\id_{A_{x_{n+1}}}) \Delta_{x_1\cdots  x_n, x_{n+1}}.
$$
The coassociativity and counitality of the coproduct imply that 
$$
\Delta^{n_1+\cdots +n_k}_{\bar{x}_1, \dots, \bar{x}_k}=(\Delta^{n_1}_{\bar{x}_1} \otimes \cdots \otimes \Delta^{n_k}_{\bar{x}_k})\Delta^k_{|\bar{x}_1|, \dots ,|\bar{x}_k|},
$$
where each $\bar{x}_i$ is an $n_i$-tuple of elements of $H$ and $|\bar{x}_i|$ denotes the product of the elements of $\bar{x}_i$.
We will depict the morphisms $\mu^n_x$ and $\Delta^n_{x_1, \dots, x_n}$ as
$$
\psfrag{a}[Bl][Bl]{\scalebox{.8}{$x_1$}}
\psfrag{c}[Bl][Bl]{\scalebox{.8}{$x_2$}}
\psfrag{n}[Bl][Bl]{\scalebox{.8}{$x_n$}}
\psfrag{v}[Bl][Bl]{\scalebox{.8}{$x_1 \cdots x_n$}}
\psfrag{x}[Bl][Bl]{\scalebox{.8}{$x$}}
\mu^n_x=\rsdraw{.4}{.9}{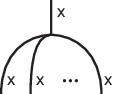}
\qquad \text{and} \qquad
\Delta^n_{x_1, \dots, x_n}=\rsdraw{.4}{.9}{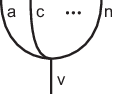} \,\;.
$$

\subsection{Crossed module-actions}\label{sect-crossed-module-actions}
Let $\chi \co E \to H$ be a crossed module. A \emph{$\chi$-action} on a Hopf $H$-coalgebra $A = \{ A_x \}_{x \in H}$ is a family of \kt algebra homomorphisms 
$$
\phi=\{\phi_{x,e} \co A_x \to A_{\CM(e)x} \}_{(x,e) \in H \times E}
$$
such that for all $x,y \in H$ and $e,f \in E$,
\begin{itemize}
\item $\phi_{x,1}=\id_{A_x}$, 
\item $\phi_{\CM(e)x,f}\, \phi_{x,e}=\phi_{x,fe}$,  
\item $(\phi_{x,e} \otimes \phi_{y,f}) \Delta_{x,y}=\Delta_{\CM(e)x,\CM(f)y}\, \phi_{xy,e\lact{x}{\!f}}$.
\end{itemize}
Note that the indices in Axiom~(iii) are coherent since the equivariance of $\CM$ (see Section~\ref{sect-crossed-modules-def}) implies that
$\CM\bigl(e\lact{x}{\!f}\bigr)xy=\CM(e)x\CM(f)y$. Also, each $\phi_{x,e}$ is an isomorphism, with inverse
$$
\phi_{x,e}^{-1}=\phi_{\CM(e)x, e^{-1}},
$$
and commutes with the antipode $S$ of $A$ in the following sense:  
$$
\phi_{x,e} \, S_x = S_{\CM(e)x} \, \phi_{x^{-1},\lact{{x^{-1}}}{\!(}e^{-1})}.
$$

We depict the $\CM$-action $\phi_{x,e} \co A_x \to A_{\CM(e)x}$ by a strand with a dot labeled with $e$ (on the left or on the right) as follows:
$$
\psfrag{x}[Bl][Bl]{\scalebox{.8}{$x$}}
\psfrag{y}[Bl][Bl]{\scalebox{.8}{$\CM(e)x$}}
\psfrag{e}[Br][Br]{\scalebox{.9}{$e$}}
\phi_{x,e}= \, \rsdraw{.45}{.9}{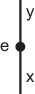} \quad \text{or} \quad
\psfrag{e}[Bl][Bl]{\scalebox{.9}{$e$}}
 \phi_{x,e}= \,\rsdraw{.45}{.9}{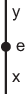} \;.
$$
The three axioms of the $\chi$-action are  depicted as
$$
\psfrag{x}[Bl][Bl]{\scalebox{.8}{$x$}}
\psfrag{e}[Br][Br]{\scalebox{.9}{$1$}}
\rsdraw{.45}{.9}{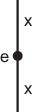}= \,
\psfrag{x}[Bl][Bl]{\scalebox{.8}{$x$}}
\rsdraw{.45}{.9}{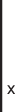} \;, \quad \quad \qquad
\psfrag{x}[Br][Br]{\scalebox{.8}{$x$}}
\psfrag{y}[Br][Br]{\scalebox{.8}{$\CM(e)x$}}
\psfrag{z}[Br][Br]{\scalebox{.8}{$\CM(f)\CM(e)x$}}
\psfrag{u}[Bl][Bl]{\scalebox{.9}{$f$}}
\psfrag{e}[Bl][Bl]{\scalebox{.9}{$e$}}
\rsdraw{.45}{.9}{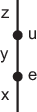}= \;
\psfrag{x}[Bl][Bl]{\scalebox{.8}{$x$}}
\psfrag{a}[Bl][Bl]{\scalebox{.8}{$\CM(fe)x$}}
\psfrag{v}[Bl][Bl]{\scalebox{.9}{$fe$}}
\rsdraw{.45}{.9}{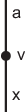} \quad, \qquad \qquad
\psfrag{z}[Bl][Bl]{\scalebox{.8}{$xy$}}
\psfrag{n}[Bl][Bl]{\scalebox{.8}{$\CM\bigl(e\lact{x}{\!f}\bigr)xy$}}
\psfrag{a}[Bl][Bl]{\scalebox{.8}{$\CM(f)y$}}
\psfrag{c}[Br][Br]{\scalebox{.8}{$\CM(e)x$}}
\psfrag{x}[Br][Br]{\scalebox{.8}{$x$}}
\psfrag{y}[Bl][Bl]{\scalebox{.8}{$y$}}
\psfrag{u}[Bl][Bl]{\scalebox{.9}{$f$}}
\psfrag{e}[Br][Br]{\scalebox{.9}{$e$}}
\psfrag{v}[Br][Br]{\scalebox{.9}{$e\lact{x}{\!f}$}}
\rsdraw{.45}{.9}{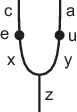} \quad \; = \; \quad \rsdraw{.45}{.9}{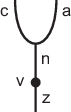} \quad \quad.
$$
The fact that $\phi_{x,e}$ is an algebra homomorphism is depicted as
$$
\psfrag{x}[Bl][Bl]{\scalebox{.8}{$x$}}
\psfrag{u}[Bl][Bl]{\scalebox{.9}{$e$}}
\psfrag{z}[Bl][Bl]{\scalebox{.8}{$\CM(e)x$}}
\psfrag{c}[Br][Br]{\scalebox{.8}{$x$}}
\psfrag{a}[Br][Br]{\scalebox{.8}{$\CM(e)x$}}
\psfrag{e}[Br][Br]{\scalebox{.9}{$e$}}
\rsdraw{.45}{.9}{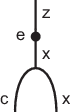}\;\, = \;\,\quad \rsdraw{.45}{.9}{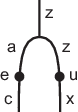}\quad \qquad \text{and} \qquad \quad
\psfrag{x}[Br][Br]{\scalebox{.8}{$x$}}
\psfrag{y}[Br][Br]{\scalebox{.8}{$\CM(e)x$}}
\psfrag{z}[Bl][Bl]{\scalebox{.8}{$\CM(e)x$}}
\psfrag{e}[Bl][Bl]{\scalebox{.9}{$e$}}
\rsdraw{.45}{.9}{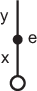}\; = \;\, \rsdraw{.45}{.9}{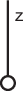} \quad \;.
$$
The commutativity of a $\CM$-action with the antipode is depicted as
$$
\psfrag{e}[Br][Br]{\scalebox{.9}{$e$}}
\psfrag{u}[Br][Br]{\scalebox{.9}{$\lact{{x^{-1}}}{\!(}e^{-1})$}}
\psfrag{z}[Bl][Bl]{\scalebox{.8}{$\CM(e)x$}}
\psfrag{x}[Bl][Bl]{\scalebox{.8}{$x$}}
\psfrag{c}[Bl][Bl]{\scalebox{.8}{$x^{-1}$}}
\psfrag{a}[Bl][Bl]{\scalebox{.8}{$x^{-1}\CM(e^{-1})$}}
\rsdraw{.45}{.9}{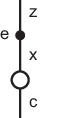}\;= \;\,
\rsdraw{.45}{.9}{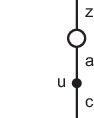}\; \quad \quad  \quad \;.
$$

\subsection{Hopf crossed module-coalgebras}\label{sect-Hopf-crossed-module-coalgebras-def}
A \emph{Hopf $\CM$-coalgebra} (over $\kk$)  is a Hopf $H$-coalgebra  (over $\kk$) endowed with a $\CM$-action.  

A Hopf $\CM$-coalgebra $A = \{ A_x \}_{x \in H}$ is \emph{of finite type} if each $A_x$ is finite-dimensional. It is \emph{involutory} if its antipode $S=\{S_x \}_{x \in H}$ satisfies $S_x  S_{x^{-1}} = \id_{A_x}$ for all $x \in H$, which is depicted as
$$
\psfrag{x}[Bl][Bl]{\scalebox{.8}{$x$}}
\psfrag{c}[Bl][Bl]{\scalebox{.8}{$x^{-1}$}}
\rsdraw{.45}{.9}{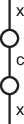}\quad  = \;\; \rsdraw{.45}{.9}{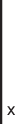} \;.
$$

\subsection{Particular cases}\label{sect-Hopf-Xi-coalg-particular}
1. Given a group $H$, the trivial map $1 \to H$ is a crossed module and the notion of a Hopf $(1\to H)$-coalgebra  agrees with that of a Hopf $H$-coalgebra.  

2. Let $E$ be an abelian group, so that the trivial map $E \to 1$ is a crossed module. Then there is a bijective correspondence between:
\begin{itemize}
\item   Hopf $(E\to 1)$-coalgebras of finite type,
\item   finite-dimensional Hopf algebras $A$ endowed with a group homomorphism from $E$ to the center of the group $G(A^*)=\Hom_{\text{alg}}(A,\kk)$ of grouplike elements of the Hopf algebra $A^*$ dual to $A$.
\end{itemize}
In this correspondence, the $(E\to 1)$-action $\phi=\{\phi_{e} \co A \to A \}_{e \in E}$ associated with a group homomorphism $\rho \co E \to Z(G(A^*))$ is given by $\phi_{e}=(\rho_e \otimes \id_A)\Delta$, where $\Delta$ is the coproduct of $A$.

\subsection{(Co)opposite Hopf crossed module-coalgebras}\label{sect-opposite-Xi-HA}
Let $A=\{A_x\}_{x \in H}$ be a  Hopf $\CM$-coalgebra with coproduct $\Delta$, antipode $S$, and $\CM$-action $\phi$.

The Hopf $\CM$-coalgebra \emph{opposite to $A$} is the family  $A^\opp=\{A^\opp_x\}_{x \in H}$, where $A^\opp_x$ denotes the opposite algebra to $A_x$, endowed with the coproduct, counit, and $\CM$\ti action of $A$ and with  the antipode $S^\opp=\{ S^\opp_x=S^{-1}_{x^{-1}} \}_{x \in H}$.

The Hopf $\CM$-coalgebra \emph{coopposite to $A$} is the family  $A^\cop=\{A^\cop_x\}_{x \in H}$, where $A^\cop_x=A_{x^{-1}}$ as an algebra, endowed with the counit of $A$ and the coproduct $\Delta^\cop$, antipode $S^\cop$, and $\CM$-action $\phi^\cop$ defined by 
$$
\Delta^\cop_{x,y}=\sigma_{A_{y^{-1}},A_{x^{-1}}} \Delta_{y^{-1},x^{-1}}, \quad S^\cop_x=S^{-1}_x , \quad \phi^\cop_{x,e}=\phi_{x^{-1},\lact{x^{-1}}{\!(}e^{-1})}
$$
for all $x,y \in H$ and $e \in E$.

\subsection{Integrals}\label{sect-integrals} 
Let $A=\{A_x\}_{x \in H}$ be an involutory Hopf $\chi$-coalgebra of finite type. 
Recall that a \emph{two-sided integral element} of the Hopf algebra $A_1$ is an element $\Lambda \in A_1$ such that for all $a \in A_1$,
$$
\Lambda a =\varepsilon(a)  \Lambda = a \Lambda.
$$
A \emph{two-sided $\chi$-integral on $A$} is a family of $\kk$-linear forms $\lambda = \{\lambda_x \colon A_x \to \kk\}_{x \in H}$ such that for all $x, y \in H$ and $e \in E$,
$$
(\id_{A_x} \otimes \lambda_y) \Delta_{x,y}=\eta_x \lambda_{xy}, \quad (\lambda_x \otimes \id_{A_y})\Delta_{x,y} = \eta_y \lambda_{xy}, \quad 
\lambda_{\chi(e)x} \phi_{x,e} = \lambda_x.
$$

\begin{lem}\label{lem-existence-Xi-integrals}
Assume that the characteristic of $\kk$ does not divide $\dim_\kk(A_1)$.  Then there is a unique two-sided integral element $\Lambda$ of $A_1$ and a unique two-sided $\chi$\ti integral  $\lambda = \{\lambda_x\}_{x \in H}$ on $A$ such that for all $x \in H$,
$$
\lambda_x(1_x)=\lambda_1(\Lambda)= \varepsilon(\Lambda)=\dim_\kk(A_1) 1_\kk.
$$
Moreover, the integral $\Lambda$ is cosymmetric, the $\chi$-integral~$\lambda$ is symmetric, and both are $S$-invariant:  
$$
\Delta^{\mathrm{cop}}_{x,x^{-1}}(\Lambda)=\Delta_{x^{-1},x}(\Lambda), \quad   \lambda_x(ba)=\lambda_x(ab), \quad   S_1(\Lambda)=\Lambda, \quad  \lambda_x S_x=\lambda_{x^{-1}}
$$
for all $x \in H$ and $a,b \in A_x$, where $\Delta^{\mathrm{cop}}_{x,x^{-1}}$ is defined in Section~\ref{sect-opposite-Xi-HA}. 
\end{lem}

\begin{proof}
Since $A_1$ is a finite-dimensional involutory Hopf algebra and $\dim_\kk(A_1) 1_\kk \neq 0$, \cite[Corollary~2.6]{LR} gives that $A_1$ is semisimple (and so unimodular) and cosemisimple. In particular, by Maschke's theorem for Hopf algebras,  there is a unique nonzero two-sided integral $\Lambda$ of $A_1$ such that $\varepsilon(\Lambda)=\dim_\kk(A_1) 1_\kk$.  Since $S_1(\Lambda)$ is also a two-sided integral of $A_1$ verifying $\varepsilon(S_1(\Lambda))=\varepsilon(\Lambda)=\dim_\kk(A_1) 1_\kk$, the uniqueness of $\Lambda$ implies that $S_1(\Lambda)=\Lambda$.

Since $A$ is of finite type and $A_1$ is cosemisimple, the Hopf $H$-coalgebra $A$ is cosemisimple (by \cite[Corollary~5.5]{Vi2}) and so the $\chi$-integrals on $A$ are two-sided (by \cite[Corollary~5.7]{Vi2}). Then it follows from \cite[Theorem~9.2]{SV2} that the space of two-sided $\chi$-integrals on $A$ is one-dimensional. Consequently, using \cite[Theorem~5.4]{Vi2}, there is a unique two-sided $\chi$-integral $\kappa = \{\kappa_x\}_{x \in H}$ on $A$ such that $\kappa_x(1_x)=1_\kk$ for all $x \in H$. Then the 
two-sided $\chi$-integral $\lambda = \dim_\kk(A_1) \kappa$   on $A$ verifies that $\lambda_x(1_x)=\dim_\kk(A_1) \kappa_x(1_x) =\dim_\kk(A_1) 1_\kk$ for all $x \in H$. Since $A_1$ is unimodular, $A$ is involutory, and the distinguished $H$-grouplike element of $A$ is trivial (by \cite[Corollary~5.7]{Vi2}), it follows from \cite[Theorem~4.2]{Vi2} that $\lambda$ is symmetric and $S$\ti invariant, and it follows from \cite[Corollary~4.4]{Vi2} that $\Lambda$ is cosymmetric. 

Finally, $\lambda_1(\Lambda) \neq 0$ (by \cite[Proposition~1.1]{LR}) and so $\Lambda'=\lambda_1(\Lambda)^{-1} \Lambda$ is a two-sided integral of $A_1$ such that  $\lambda_1(\Lambda')=1_\kk$. Using \cite[Theorem~2.5]{LR}, we obtain
$$\dim_\kk(A_1) 1_\kk=\mathrm{Tr}(\id_{A_1})=\mathrm{Tr}(S_1^2)=\lambda_1(1_1)\varepsilon(\Lambda')=\dim_\kk(A_1)\lambda_1(\Lambda)^{-1} \varepsilon(\Lambda).$$
Since $\varepsilon(\Lambda)=\dim_\kk(A_1) 1_\kk\neq 0$, we conclude that $\lambda_1(\Lambda)=\dim_\kk(A_1)1_\kk$.
\end{proof}

We depict $\Lambda \in A_1$ and $\lambda_x \co A_x \to \kk$ as in Lemma~\ref{lem-existence-Xi-integrals} by
$$
\psfrag{x}[Bl][Bl]{\scalebox{.8}{$x$}}
\psfrag{1}[Bl][Bl]{\scalebox{.8}{$1$}}
\Lambda=\rsdraw{.45}{.9}{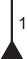} \quad \text{and} \quad \lambda_x=\rsdraw{.45}{.9}{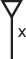} \;.
$$
Their properties are then depicted as follows:
\begin{gather*}
\psfrag{a}[Br][Br]{\scalebox{.8}{$x$}}
\psfrag{x}[Bl][Bl]{\scalebox{.8}{$x$}}
\psfrag{y}[Bl][Bl]{\scalebox{.8}{$y$}}
\psfrag{z}[Bl][Bl]{\scalebox{.8}{$xy$}}
\rsdraw{.45}{.9}{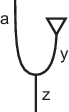}\;\, = \; \rsdraw{.45}{.9}{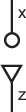} \;, \qquad
\psfrag{a}[Br][Br]{\scalebox{.8}{$x$}}
\psfrag{x}[Bl][Bl]{\scalebox{.8}{$y$}}
\psfrag{y}[Bl][Bl]{\scalebox{.8}{$y$}}
\psfrag{z}[Bl][Bl]{\scalebox{.8}{$xy$}}
\rsdraw{.45}{.9}{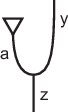} = \; \rsdraw{.45}{.9}{Cointeg-ppt2} \;, \qquad
\psfrag{a}[Br][Br]{\scalebox{.8}{$x$}}
\psfrag{x}[Bl][Bl]{\scalebox{.8}{$x$}}
\rsdraw{.45}{.9}{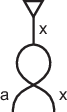} = \; \rsdraw{.45}{.9}{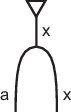} \;, \\[.8em]
\psfrag{u}[Bl][Bl]{\scalebox{.8}{$1$}}
\psfrag{v}[Br][Br]{\scalebox{.8}{$1$}}
\rsdraw{.45}{.9}{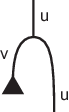}  \!= \;\, \rsdraw{.45}{.9}{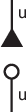}\, \; =  \rsdraw{.45}{.9}{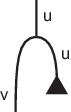} \;\;, \qquad  
\psfrag{u}[Bl][Bl]{\scalebox{.8}{$1$}}
\psfrag{x}[Bl][Bl]{\scalebox{.8}{$x$}}
\psfrag{c}[Br][Br]{\scalebox{.8}{$x^{-1}$}}
\psfrag{a}[Br][Br]{\scalebox{.8}{$x$}}
\psfrag{z}[Bl][Bl]{\scalebox{.8}{$x^{-1}$}}
\rsdraw{.45}{.9}{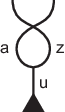} \;\;\; = \; \rsdraw{.45}{.9}{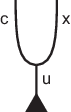} \;, \\[.8em]
\psfrag{x}[Bl][Bl]{\scalebox{.8}{$x$}}
\psfrag{a}[Br][Br]{\scalebox{.8}{$x$}}
\psfrag{e}[Bl][Bl]{\scalebox{.9}{$e$}}
\psfrag{z}[Br][Br]{\scalebox{.8}{$\CM(e)x$}}
\rsdraw{.45}{.9}{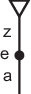}  = \rsdraw{.45}{.9}{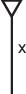} \;\,, \qquad 
\psfrag{c}[Br][Br]{\scalebox{.8}{$x^{-1}$}}
\psfrag{a}[Br][Br]{\scalebox{.8}{$x$}}
\psfrag{z}[Bl][Bl]{\scalebox{.8}{$x^{-1}$}}
\rsdraw{.45}{.9}{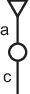}  \, =  \rsdraw{.45}{.9}{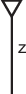} \; \;\;\,, \qquad  
\psfrag{u}[Bl][Bl]{\scalebox{.8}{$1$}}
\psfrag{v}[Br][Br]{\scalebox{.8}{$1$}}
\rsdraw{.45}{.9}{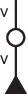} \, =  \rsdraw{.45}{.9}{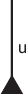} \;\,, \\[.8em]
\psfrag{v}[Br][Br]{\scalebox{.8}{$1$}}
\psfrag{a}[Br][Br]{\scalebox{.8}{$x$}}
\rsdraw{.45}{.9}{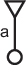} = \; \rsdraw{.45}{.9}{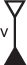} = \; \rsdraw{.45}{.9}{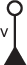}  = \dim_\kk(A_1)1_\kk\,.
\end{gather*}

\subsection{Example}\label{ex-Hopf-xi-from-bicharacter}
Let  $E$ be an abelian group and  $G$ be a finite group. Recall that the group algebra $\kk[G]$ is a finite-dimensional Hopf algebra with copro\-duct~$\Delta$, counit~$\varepsilon$, and antipode $S$ defined by $\Delta(g)=g \otimes g$, $\varepsilon(g)=1_\kk$, and~$S(g)=g^{-1}$ for all $g \in G$. Note that group homomorphisms 
$$
E \to Z\bigl(\Hom_{\text{alg}}(\kk[G],\kk)\bigr)=\Hom_{\text{group}}(G,\kk^*)
$$ 
are in bijective correspondence with bicharacters $G \times E \to \kk^*$. 
Then,  by Section~\ref{sect-Hopf-Xi-coalg-particular}, 
such a bicharacter $\omega \co G \times E \to \kk^*$ gives rise to a Hopf $(E\to 1)$-coalgebra of finite type denoted by $\kk^\omega[G]$. It is involutory and its $\CM$-action  $\phi \co E \to \Aut_{\kk}(\kk[G])$ is computed by $\phi_e(g)= \omega(g,e) g$ for all $e \in E$ and $g \in G$. When the characteristic of~$\kk$ does not divide $\dim_\kk(\kk[G])=|G|$, the integral element $\Lambda$ of $(\kk^\omega[G])_1=\kk[G]$ and the 
$(E\to 1)$-integral  $\lambda = \{\lambda_1\}$ given by Lemma~\ref{lem-existence-Xi-integrals} are computed by 
$$
\Lambda=\sum_{g \in G}g \quad \text{and} \quad \lambda_1(g)=\delta_{g,1} |G|\, 1_\kk.
$$

\subsection{Example}\label{ex-Hopf-Z4-to-Z2} 
Let $\CM\co \Z/4\Z \to \Z/2\Z$ be the crossed module from Section~\ref{ex-Lens-labelings-Z4Z-22Z}. Recall that $\Z/2\Z=\{0,1\}$ acts on $\Z/4\Z=\{\bar{0},\bar{1},\bar{2},\bar{3}\}$ by $\lact{0}{\bar{n}}=\bar{n}$ and $\lact{1}{\bar{n}}=-\bar{n}$, and that
$\CM(\bar{n})=0$. Assume that the characteristic of the field $\kk$ is not 2. Consider the \kt algebras~$A_0$ generated by $a$ and~$A_1$ generated by $u,v$ subject to the relations
$$
a^4=1, \quad u^2=1=v^2, \quad vu=-uv.
$$
Both $A_0$ and $A_1$ have  dimension 4 with respective basis $\{1,a,a^2,a^3\}$ and $\{1,u,v,uv\}$. (Note that if $\kk$ has a primitive fourth root of unity, then $A_0 \cong \kk^4$ and $A_1 \cong M_2(\kk)$ as \kt algebras, and $A_0 \oplus A_1$ is the Kac-Paljutkin Hopf algebra.)
Define algebra morphisms $\Delta_{x,y}\co A_{x+y} \to A_x \otimes A_y$ for $x,y \in\Z/2\Z$ and $\varepsilon \co A_0 \to \kk$ by setting $\varepsilon(a)=1$ and
\begin{align*}
&\Delta_{0,0}(a)=(a \otimes a)\Omega(a^2,a^2),  && \Delta_{0,1}(u)=(a \otimes u)\Omega(a^2,v), && \Delta_{0,1}(v)=a^2 \otimes v, \\
& \Delta_{1,1}(a)=(u \otimes u)\Omega(-v,v), && \Delta_{1,0}(u)=(u \otimes a)\Omega(-v,a^2), && \Delta_{1,0}(v)=v \otimes a^2,
\end{align*}
where $$\Omega(s,t)=\frac{1}{2}\bigl(1 \otimes 1 + s \otimes 1 + 1 \otimes t - s \otimes t\bigr).$$  
Define anti-multiplicative $\kk$-linear isomorphisms $S_0\co A_0 \to A_0$ and $S_1\co A_1 \to A_1$ by setting
$$
S_0(a)=a, \quad S_1(u)=u, \quad S_1(v)=v.
$$
For $\bar{n} \in \Z/4\Z$, consider the algebra morphisms $\phi_{0, \bar{n}} \co A_0 \to A_0$ and $\phi_{1, \bar{n}} \co A_1 \to A_1$ defined by 
$$
\phi_{0, \bar{n}}(a)=(-1)^{n}a, \quad \phi_{1, \bar{n}}(u)=(-1)^n u, \quad \phi_{1, \bar{n}}(v)=v.
$$
Then $A=\{A_0,A_1\}$ is an involutory Hopf $\CM$-coalgebra of finite type with  $\CM$-action $\phi=\{\phi_{x,\bar{n}}\}_{(x, \bar{n}) \in \Z/2\Z \times \Z/4\Z}$. Note that the characteristic of $\kk$ does not divide $\dim_\kk(A_0)=4$. Then Lemma~\ref{lem-existence-Xi-integrals} applies and the two-sided integral element $\Lambda$ of $A_0$ and the two-sided $\chi$-integral  $\lambda = \{\lambda_0, \lambda_1\}$ on $A$ given by this lemma are computed by
\begin{gather*}
\Lambda=1+a+a^2+a^3, \qquad \lambda_0(1)=\lambda_1(1)=4, \\
\lambda_0(a)=\lambda_0(a^2)=\lambda_0(a^3)=\lambda_1(u)=\lambda_1(v)=\lambda_1(uv)=0.
\end{gather*}

\section{Invariants of crossed manifolds}\label{sect-invariant-Xi-manifolds}

Throughout this section, $\CM \co E \to H$ is a crossed module and $A=\{A_x\}_{x \in H}$ is an involutory Hopf $\chi$-coalgebra of finite type such that the characteristic of $\kk$ does not divide $\dim_\kk(A_1)$. We construct from this data  a topological invariant of $\CM$\ti manifolds. Its definition goes by associating tensors (in the spirit of Kuperberg~\cite{Ku}) with $\CM$-Heegaard diagrams
of $\CM$-manifolds. We prove by examples that this invariant is nontrivial.

\subsection{An invariant of $\CM$-manifolds}\label{sect-def-invariant}
Let $(M,g\in[M,B\CM])$ be a $\CM$-manifold. Pick a $\CM$-Hee\-gaard diagram $(D,(\alpha,\beta))$ of $(M,g)$, where $D=(\Sigma,\up,\low)$. For an intersection point~$s$ of~$D$, we denote by $\alpha(s)\in H$ the label of the upper circle containing $s$ and set $\nu_s=1$ if $s$ is positive and $\nu_s=-1$ otherwise (see Section~\ref{sect-Heegaard-diag}). Consider the integral element $\Lambda \in A_1$ and the $\chi$-integral  $\lambda = \{\lambda_x\}_{x \in H}$ associated to $A$ by Lemma~\ref{lem-existence-Xi-integrals} and their graphical representation given in Section~\ref{sect-integrals}. 

For a lower circle $l \in \low$, the set $\I_l$ of intersection points between $l$ and the upper circles is totally ordered by enumerating its elements  starting from the basepoint of~$l$ and following the orientation of $l$. We associate to $l$ the vector
$$
\psfrag{c}[Br][Br]{\scalebox{.8}{$\alpha(s_1)^{\nu_{s_1}}$}}
\psfrag{u}[Bc][Bc]{\scalebox{.8}{$\alpha(s_2)^{\nu_{s_2}}$}}
\psfrag{x}[Bl][Bl]{\scalebox{.8}{$\alpha(s_n)^{\nu_{s_n}}$}}
\psfrag{a}[Bl][Bl]{\scalebox{.8}{$1$}}
\psfrag{e}[Br][Br]{\scalebox{.9}{$\beta(l)$}}
\psfrag{z}[Bl][Bl]{\scalebox{.8}{$\CM(\beta(l))$}}
\Lambda_l= \rsdraw{.3}{.9}{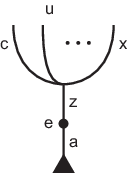}  \; = \Delta^{n}_{\alpha(s_1)^{\nu_{s_1}}, \dots , \alpha(s_n)^{\nu_{s_n}}} \phi_{1,\beta(l)}(\Lambda) \in \bigotimes_{s \in \I_l} A_{\alpha(s)^{\nu_s}}
$$
where $\I_l=\{s_1< \cdots <s_n\}$. 
Here (and in what follows) a tensor product indexed by a totally ordered set is taken following the order of this set. Note that the vector~$\Lambda_l$ is well-defined since $\alpha(s_1)^{\nu_{s_1}} \cdots  \alpha(s_n)^{\nu_{s_n}} =\CM(\beta(l))$ by the first condition of a $\CM$\ti labeling. 
Set also
$$
\ant_l=\bigotimes_{s \in \I_l}\ant_s \co \bigotimes_{s \in \I_l} A_{\alpha(s)^{\nu_s}} \to \bigotimes_{s \in \I_l} A_{\alpha(s)}
$$
where $\ant_s \co A_{\alpha(s)^{\nu_s}} \to A_{\alpha(s)}$ is defined by
$\ant_s=\id_{A_{\alpha(s)}}$ if $s$ is positive and $\ant_s=S_{\alpha(s)}$ if $s$ is negative  (with  $S$ being the antipode of $A$).

For an upper circle $u \in \up$, pick a basepoint for $u$ distinct from the intersection points. Then the set $\I_u$ of intersection points between $u$ and the lower circles is totally ordered by enumerating its elements  starting from the basepoint of $u$ and following the orientation of $u$. We associate to $u$ the linear form
$$
\psfrag{u}[Br][Br]{\scalebox{.8}{$\alpha(u)$}}
\psfrag{a}[Bc][Bc]{\scalebox{.8}{$\alpha(u)$}}
\psfrag{x}[Bl][Bl]{\scalebox{.8}{$\alpha(u)$}}
\lambda_u= \rsdraw{.45}{.9}{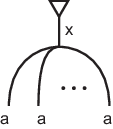}  =\lambda_{\alpha(u)} \mu_{\alpha(u)}^{|u|} \co  \bigotimes_{s \in \I_u} A_{\alpha(s)}=\bigl(A_{\alpha(u)}\bigr)^{\otimes |u|} \to \kk.
$$

Pick total orders for the sets $\up$ and $\low$.  Let $\I$ be the set of intersection points of~$D$. The order of $\up$ induces a total order on $\I$ given by the lexicographic order on $\I=\amalg_{u \in \up} \I_u$ (meaning that $s \in \I_u$ is smaller than $s' \in \I_{u'}$ if $u < u'$ in $\up$ or if $u=u'$ and $s <s'$ in~$\I_u$).
The order of $\low$ induces similarly a total order on~$\I$.
Denote by $\I_\up$ (respectively $\I_\low$) the set $\I$ endowed with the order induced by $\up$ (respectively~$\low$). 
Set 
$$
\Omega_\low=\bigotimes_{l\in \low} \ant_l(\Lambda_l)\in \bigotimes_{s \in \I_\low}A_{\alpha(s)} \quad \text{and} \quad
\lambda_\up=\bigotimes_{u \in \up} \lambda_u \co \bigotimes_{s \in \I_\up} A_{\alpha(s)} \to \kk.
$$
Recall that a symmetric monoidal category provides representations of the symmetric groups. Thus the flip maps of $\kk$-vector spaces associate to the unique increasing map $\I_\low \to \I_\up$  (which is a permutation of $\I$) a \kt linear isomorphism
$$
P_{\low,\up}\co \bigotimes_{s \in \I_\low}A_{\alpha(s)} \to \bigotimes_{s \in \I_\up} A_{\alpha(s)}.
$$
Finally, set
\begin{equation}\label{eq-def-KA}
K_A(M,g)=\dim_\kk(A_1)^{g(\Sigma)-|\up|-|\low|} \; \lambda_\up P_{\low,\up} (\Omega_\low) \in \kk,
\end{equation}
where $g(\Sigma)$ denotes the genus of $\Sigma$.

\begin{thm}\label{thm-colored-Kuperberg}
$K_A(M,g)$ is a topological invariant of $\CM$-manifolds.
\end{thm}
In view of Section~\ref{sect-flat-2-bundles}, Theorem~\ref{thm-colored-Kuperberg} provides a topological invariant of $\CM$-bundles, that is, of 
(flat) principal $\tg_\CM$-bundles over closed connected oriented 3-manifolds (where $\tg_\CM$ is the 2-group associated with $\CM$).

We prove Theorem~\ref{thm-colored-Kuperberg} in Section~\ref{sect-proof-xi-kup}. The proof consists of showing that the right-hand side of~\eqref{eq-def-KA} is independent of the choice of the $\CM$-Heegaard diagram representing~$(M,g)$, the choice of basepoints for the upper circles,  and the choice of the total orders for the sets of upper and lower circles. For that, we use in particular Theorem~\ref{thm-colored-Reidemeister} which relates two $\CM$-Heegaard diagrams of $(M,g)$ via  $\CM$-moves.

We illustrate the computation of the invariant of Theorem~\ref{thm-colored-Kuperberg} in Sections~\ref{ex-compute-Lens}\ti \ref{ex-computation-E-G-pairing} below. In particular, this invariant is nontrivial  and may even distinguish homotopy classes of phantom maps (i.e., of maps inducing trivial homomorphisms on homotopy groups), see for instance Section~\ref{ex-compute-Lens}.

\subsection{Properties}\label{sect-properties-KuA} 
The  invariant $K_A$ of $\CM$-manifolds from Theorem~\ref*{thm-colored-Kuperberg} satisfies the following properties:\\[-1em]

\begin{mylist}{2}
  \item[1] If $g$ is the homotopy class of a constant map, then 
$$K_A(M,g)= \mathrm{Ku}_{A_1}(M),$$ 
where the right-hand side is the Kuperberg invariant \cite{Ku} of $M$ derived from the involutory Hopf algebra $A_1$.
  \item[2] The invariant $K_A$ generalizes the invariant defined in \cite{Vi3}. More precisely, if $E=1$, then $B\CM=B(1\to H)=BH$ is a $K(H,1)$ space  and the invariant $K_A$ of $(1\to H)$-manifolds (or equivalently of flat $H$-bundles over 3-manifolds) is equal to the invariant defined in \cite[Theorem 9]{Vi3} using the involutory Hopf $H$-coalgebra~$A$. As a consequence, the invariant of \cite[Theorem 9]{Vi3} does not depend on the choice of the basepoints (of the flat bundles). 
  \item[3] The invariant $K_A$ of Theorem~\ref{thm-colored-Kuperberg} is multiplicative with respect to the connected sum of $\CM$-manifolds:
    $$
    K_A(M\# M',g\#g')=K_A(M,g)K_A(M',g').
    $$
    This follows directly from the fact that a $\CM$-Heegaard diagram of a connected sum of two $\CM$-manifolds is given by the connected sum of the  $\CM$-Heegaard diagrams of the $\CM$-manifolds. 
  \item[4] The opposite of a $\CM$-manifold  $(M,g)$ is the $\CM$-manifold $(-M,g)$, where $-M$ is~$M$ with the opposite orientation. Then
    $$
    K_A(-M,g)=K_{A^\opp}(M,g)=K_{A^\cop}(M,g)
    $$
    where $A^\opp$ and $A^\cop$ are the Hopf $\CM$-coalgebras opposite and coopposite to $A$ (see Section~\ref{sect-opposite-Xi-HA}).
    Indeed, if $((\Sigma,\up,\low),(\alpha,\beta))$ is a $\CM$-Heegaard diagram for $(M,g)$, then  $((-\Sigma,\up,\low),(\alpha^{-1},\beta))$ is a $\CM$-Heegaard diagram for $(-M,g)$. 
    The first equality follows from the anti-multiplicativity of the antipode $S$ of $A$ and from the $S$-invariance of the $\chi$\ti integral $\lambda$. The second equality follows from the anti-comultiplicativity of~$S$ and from the $S$-invariance of the integral element $\Lambda$ of $A_1$. 
  \item[5] Recall from \cite{SV} that any $\CM$-fusion category $\cc$ with $\dim(\cc^1_1)$ invertible in~$\kk$ defines a state sum invariant $\vert M,g \vert_\cc$ of $\CM$-manifolds.
    When $\kk$ is an algebraically closed field, the category $\mmod_\CM(A)$ of finite-dimensional modules over $A$ is $\CM$\ti fusion (by \cite[Theorem 8.3]{SV2}) and $\dim\bigl(\mmod_\CM(A)^1_1\bigr)=\dim_\kk(A_1)1_\kk$ is invertible. We will prove in a subsequent article that 
    $$
    K_A(M,g)=\dim_\kk(A_1)\, \vert M,g \vert_{\mmod_\CM(A)}.
    $$
    In particular, $\vert M,g \vert_{\mmod_\CM(A)}=\dim_\kk(A_1)^{-1} K_A(M,g)$ can be computed without determining the simple $A$-modules. This equality generalizes the result of \cite{BW2} which relates the Kuperberg invariant of 3-manifolds with the Turaev-Viro-Barrett-Westbury invariant of 3-manifolds.
\end{mylist}

\subsection{Remark}\label{rk-Hopf-xi-algebras}
The notion of a Hopf $\CM$-coalgebra is not self-dual. The dual notion is that of a Hopf $\CM$-algebra (see \cite[Section 7.8]{SV2}).
An invariant of $\CM$-manifolds can be defined from an involutory Hopf $\CM$-algebra of finite type $B=\{B_x\}_{x \in H}$ in a manner similar to that in Section~\ref{sect-def-invariant} (by using the $\CM$-action for upper circles and the graded integral elements for lower circles). We do not elaborate further on this, since this invariant actually corresponds to the invariant $K_{B^*}$, where $B^*=\{B_x^*\}_{x \in H}$ is the Hopf $\CM$-coalgebra dual to $B$.

\subsection{Example}\label{ex-compute-Lens}
Consider the lens space $L(p,q)$ and its Heegaard diagram $D$ given in Section~\ref{ex-Lens}. Recall from  Section~\ref{ex-Lens-labelings} that the set of homotopy classes of maps  $[L(p,q), B\CM]$ is in bijection with the set of orbits $\mathfrak{L}_D/\tg_D$, where
$$
\mathfrak{L}_D=\bigl\{ (x,e) \in H \times E\; \big | \; \CM(e)=x^p \;\, \text{and} \;\, \lact{x}{e}=e \bigr\} \quad \text{and} \quad \tg_D=H \ltimes E,
$$
and that $g_{[x,e]} \in [L(p,q), B\CM]$ denotes the homotopy class of maps associated with an orbit $[x,e]\in \mathfrak{L}_D/\tg_D$. Then
$$
\psfrag{a}[Bl][Bl]{\scalebox{.8}{$1$}}
\psfrag{e}[Br][Br]{\scalebox{.9}{$e$}}
\psfrag{z}[Bl][Bl]{\scalebox{.8}{$\CM(e)=x^p$}}
\psfrag{x}[Bl][Bl]{\scalebox{.8}{$x$}}
\psfrag{I}[Bc][Bc]{\scalebox{.9}{$\mathfrak{S}_x^q$}} 
K_A\bigl(L(p,q),g_{[x,e]}\bigr)=d^{-1}\;  \rsdraw{.55}{.9}{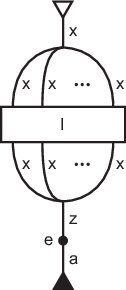}\,  =d^{-1}  \lambda_x \mu_x^p \mathfrak{S}_x^q \Delta^{p}_{x, \dots , x} \phi_{1,e}(\Lambda) \in \kk
$$
where $d=\dim_\kk(A_1)$ and $\mathfrak{S}_x$ is the flip map
$$
\psfrag{x}[Bl][Bl]{\scalebox{.8}{$x$}}
\mathfrak{S}_x=\;  \rsdraw{.45}{.9}{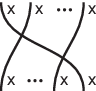}\,  =\sigma_{A_x^{\otimes (p-1)},A_x} \co A_x^{\otimes p} \to A_x^{\otimes p}.
$$
For instance, using the symmetry of the $\chi$-integral~$\lambda$ of $A$ (see Lemma~\ref{lem-existence-Xi-integrals}), the invariant for the 
real projective space $\mathbb{RP}^3=L(2,1)$ is computed by
$$
K_A\bigl(\mathbb{RP}^3,g_{[x,e]}\bigr)=d^{-1}\; \lambda_x \mu_x \Delta_{x,x} \phi_{1,e}(\Lambda).
$$ 
As an example, consider the crossed module $\CM\co \Z/4\Z=\{\bar{0},\bar{1},\bar{2},\bar{3}\} \to \Z/2\Z=\{0,1\}$ and the involutory Hopf $\CM$-coalgebra $A=\{A_0,A_1\}$ of finite type from Section~\ref{ex-Hopf-Z4-to-Z2} and assume that $\kk$ is not of characteristic 2 (so that the characteristic of $\kk$ does not divide $\dim_\kk(A_0)=4$). Recall from Section~\ref{ex-Lens-labelings-Z4Z-22Z} that
$$
\big |[\mathbb{RP}^3, B\CM] \big |=4 \quad \text{with} \quad [\mathbb{RP}^3, B\CM]=\bigl\{ g_{[0,\bar{0}]}, g_{[0,\bar{1}]}, g_{[1,\bar{0}]}, g_{[1,\bar{2}]} \bigr \}.
$$
Using the explicit description of the product, coproduct, $\CM$-action, and integrals of~$A$, we obtain:
$$
K_A\bigl(\mathbb{RP}^3,g_{[0,\bar{0}]}\bigr)=1, \; K_A\bigl(\mathbb{RP}^3,g_{[0,\bar{1}]}\bigr)=0, \; K_A\bigl(\mathbb{RP}^3,g_{[1,\bar{0}]}\bigr)=K_A\bigl(\mathbb{RP}^3,g_{[1,\bar{2}]}\bigr)=\frac{3}{4}.
$$
In particular, $K_A$ distinguishes $g_{[0,\bar{0}]}$ and $g_{[0,\bar{1}]}$ which are phantom maps. Indeed both  induce trivial homomorphisms on homotopy groups: they induce the trivial homomorphism on the fundamental groups (by definition of their labeling) and for higher homotopy groups,  this follows from the facts that $\pi_2(\mathbb{RP}^3)$ and $\pi_n(B\CM)$ are  trivial for $n \geq 3$.

\subsection{Example}\label{ex-compute-Poincare}
Consider the Poincar\'e homology sphere $\mathbb{P}$ and its Heegaard diagram~$D$ given in Section~\ref{ex-Poincare}.  Recall from  Section~\ref{ex-Poincare-labelings} that the set of homotopy classes of maps  $[\mathbb{P}, B\CM]$ is in bijection with the set of orbits $\mathfrak{L}_D/\tg_D$.  Denote by $[x,y,e,f]\in \mathfrak{L}_D/\tg_D$ the class of $(x,y,e,f)\in \mathfrak{L}_D$ and by $g_{[x,y,e,f]} \in [\mathbb{P}, B\CM]$ the corresponding homotopy class. Then
$$
\psfrag{a}[Bl][Bl]{\scalebox{.8}{$1$}}
\psfrag{e}[Br][Br]{\scalebox{.9}{$e$}}
\psfrag{z}[Bl][Bl]{\scalebox{.8}{$\CM(e)$}} 
\psfrag{d}[Br][Br]{\scalebox{.9}{$f$}}
\psfrag{c}[Bl][Bl]{\scalebox{.8}{$\CM(f)$}} 
\psfrag{x}[Bl][Bl]{\scalebox{.8}{$x$}}
\psfrag{s}[Br][Br]{\scalebox{.8}{$x$}}
\psfrag{y}[Bl][Bl]{\scalebox{.8}{$y$}}
\psfrag{b}[Br][Br]{\scalebox{.8}{$y$}}
\psfrag{o}[Bc][Bc]{\scalebox{.8}{$x^{-1}$}}
\psfrag{u}[Bc][Bc]{\scalebox{.8}{$y^{-1}$}}
K_A\bigl(\mathbb{P},g_{[x,y,e,f]}\bigr)=\dim_\kk(A_1)^{-2}\;  \rsdraw{.55}{.9}{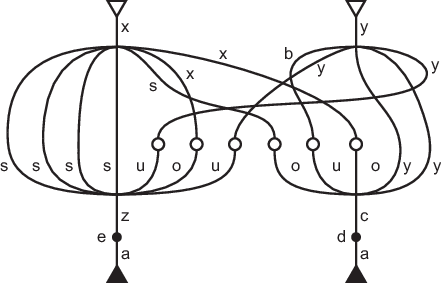} 
$$
where the disks denote the antipode (see Section~\ref{sect-graphical-conventions}).

\subsection{Example}\label{ex-computation-E-G-pairing}
Let $\omega \co G \times E \to \kk^*$ be a bicharacter, where $G$ is a finite group and~$E$ is an abelian group. Consider the involutory Hopf $(E \to 1)$-coalgebra of finite type $\kk^\omega[G]$ from Section~\ref{ex-Hopf-xi-from-bicharacter}. Assume that the characteristic of $\kk$ does not divide
$\dim_\kk((\kk^\omega[G])_1)=|G|$. Then, for any $(E \to 1)$-manifold $(M,g)$,  we have
$$
K_A(M,g)= \sum_{f \in [M,BG]} \!\! I_\omega(M,f,g)
$$
where the scalar $I_\omega(M,f,g)$ is defined as follows. Since $E$ and $\kk^*$ are abelian, the bicharacter $\omega$ induces a group homomorphism $\widetilde{\omega} \co G^\ab \otimes E \to \kk^*$, where $G^\ab$ is the abelianization of the group $G$. Denote by  $f_\sharp \co \pi_1(M)^\ab \to G^\ab$ the group homomorphism induced by $f$. Let $\theta_g$ be the image of $g$ under the isomorphism
$$
[M,B(E \to 1)] \cong [M,K(E,2)] \cong H^2(M,E) \cong H_1(M,E) \cong \pi_1(M)^\ab \otimes E
$$
induced by the Poincar\'e duality and the Hurewicz theorem. Then 
$$
I_\omega(M,f,g)=\widetilde{\omega}\bigl( (f_\sharp \otimes \id_E)(\theta_g) \bigr).
$$

\subsection{Proof of Theorem~\ref*{thm-colored-Kuperberg}}\label{sect-proof-xi-kup}
Given a $\CM$-Heegaard diagram $(D=(\Sigma,\up,\low),(\alpha,\beta))$, any choice of basepoints for the upper circles and of total orders for $\up$ and $\low$ allows us to define the right-hand side of~\eqref{eq-def-KA}, that is, the scalar
\begin{equation}\label{eq-pf-def1}
K_A\bigl(D,(\alpha,\beta)\bigr)=\dim_\kk(A_1)^{g(\Sigma)-|\up|-|\low|} \; \lambda_\up P_{\low,\up} (\Omega_\low) \in \kk.
\end{equation}

First note that $K_A(D,(\alpha,\beta))$ does not depend on these choices. Indeed, independence in the choice of basepoints for the upper circles follows from the fact that the linear form 
$\lambda_u=\lambda_{\alpha(u)} \mu_{\alpha(u)}^{|u|}$  associated with any upper circle $u$ is cyclically symmetric (by the symmetry of the $\chi$-integral~$\lambda$). Independence in the choice of the total orders for $\up$ and $\low$ follows from the symmetry of the category of $\kk$-vector spaces (and in particular from the fact that under the canonical identification $\kk \otimes \kk \cong \kk$, the flip map $\sigma_{\kk,\kk}$ is the identity). 

Next, by Theorem~\ref{thm-colored-Reidemeister}, we only need to verify that $K_A(D,(\alpha,\beta))$ remains unchanged when applying to $(D,(\alpha,\beta))$ one of the $\CM$-moves (i)-(viii) described in Section~\ref{sect-Xi-moves}.
Before proving the invariance under these moves below, we give another useful formulation of $K_A(D,(\alpha,\beta))$. Set 
$$
\Lambda_\low=\bigotimes_{l\in \low}  \Lambda_l, \quad  \ant_\low=\bigotimes_{l\in \low} \ant_l,  \quad \ant_\up=\bigotimes_{u\in \up} \ant_u  \quad \text{where} \quad \ant_u=\bigotimes_{s \in \I_u}\ant_s,
$$
and consider  \kt linear isomorphism 
$$
Q_{\low,\up}\co \bigotimes_{s \in \I_\low} A_{\alpha(s)^{\nu_s}} \to \bigotimes_{s \in \I_\up} A_{\alpha(s)^{\nu_s}}
$$
induced by the flip maps of $\kk$-vector spaces and associated with the unique increasing map $\I_\low \to \I_\up$  (which is a permutation of $\I$).
Then 
$$
\ant_\low(\Lambda_\low)=\bigotimes_{l\in \low} \ant_l(\Lambda_l)=\Omega_\low, \quad \lambda_\up \ant_\up=\bigotimes_{u\in \up} \lambda_u\ant_u, \quad \text{and} \quad  P_{\low,\up}\ant_\low=\ant_\up Q_{\low,\up}.
$$
In particular, we obtain 
\begin{equation}\label{eq-pf-def2}
K_A\bigl(D,(\alpha,\beta)\bigr)=\dim_\kk(A_1)^{g(\Sigma)-|\up|-|\low|} \; \lambda_\up \ant_\up Q_{\low,\up}(\Lambda_\low).
\end{equation}
\vspace{-.4em}

\pfpart{Invariance under orientation preserving diffeomorphisms of $\Sigma$} 
Such a diffeomorphism does not modify either the intersection pattern or the $\chi$-labelings. By transferring (via this diffeomorphism) the basepoints for upper circles and the orders of the sets of upper and lower circles, we deduce that $K_A(D,(\alpha,\beta))$ remains literally the same when applying this $\CM$-move.\\ 

\pfpart{Invariance under moving basepoints}
This move consists of moving the basepoint of a lower circle $l$ across an intersection point $s$ by following its orientation. Denote $\I_l=\{s=s_1< \cdots <s_n\}$ and set $x=\alpha(s_1)^{\nu_1}$, $y=\alpha(s_2)^{\nu_2} \cdots \alpha(s_n)^{\nu_n}$, $e=\beta(l)$, where $\nu_i=\nu_{s_i}$ for $1 \leq i \leq n$. Recall that $\CM(e)=x y$. Before the move, the vector~$\Lambda_l$ associated with  $l$ is  
$$
\Lambda_l= (\id_{A_x} \otimes \delta)\Delta_{x,y}\bigl(\phi_{1,e}(\Lambda)\bigr) \quad \text{where} \quad \delta= \Delta^{(n-1)}_{\alpha(s_2)^{\nu_2}, \dots , \alpha(s_n)^{\nu_n}}.
$$
After the move, the new label of~$l$~is $\lact{(x^{-1})}{e}$ and the vector associated with $l$ is  
$$
\Lambda'_l= (\delta \otimes \id_{A_x})\Delta_{y,x}\bigl(\phi_{1,\lact{x^{-1}}{\!e}}(\Lambda)\bigr).
$$
Now 
\begin{gather*}
\psfrag{z}[Bl][Bl]{\scalebox{.8}{$1$}}
\psfrag{a}[Bl][Bl]{\scalebox{.8}{$x$}}
\psfrag{x}[Bl][Bl]{\scalebox{.8}{$y^{-1}$}}
\psfrag{t}[Bl][Bl]{\scalebox{.8}{$y$}}
\psfrag{e}[Bl][Bl]{\scalebox{.9}{$e$}}
\psfrag{v}[Br][Br]{\scalebox{.9}{$\lact{(x^{-1})}{e}$}}
\psfrag{c}[Br][Br]{\scalebox{.8}{$y$}}
\psfrag{b}[Br][Br]{\scalebox{.8}{$x$}}
\psfrag{u}[Br][Br]{\scalebox{.8}{$y^{-1}$}}
\psfrag{s}[Br][Br]{\scalebox{.8}{$e$}}
\rsdraw{.45}{.9}{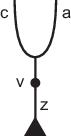} \; \overset{(i)}{=} \;\, \rsdraw{.45}{.9}{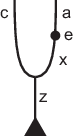} \quad  \overset{(ii)}{=} \quad \, \rsdraw{.45}{.9}{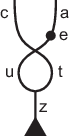} \;\;\, \overset{(iii)}{=} \;\;\, \rsdraw{.45}{.9}{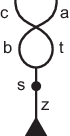}  \;.
\end{gather*}
Here $(i)$ and $(iii)$ follow from the compatibility of the $\CM$-action with the coproduct, and $(ii)$ from the cosymmetry of $\Lambda$. Consequently, 
$$
\Lambda'_l=\sigma_{\bigl(A_{\alpha(s_1)^{\nu_1}}\bigr),\bigl(A_{\alpha(s_2)^{\nu_2}} \otimes \cdots \otimes A_{\alpha(s_n)^{\nu_n}} \bigr)} (\Lambda_l).
$$
This implies that the morphism $Q_{\low,\up} (\Lambda_\low)$, and thus $K_A(D,(\alpha,\beta))$,  remains the same when applying this $\CM$-move.\\

\pfpart{Invariance under orientation reversals} 
Suppose first that the orientation of an upper circle $u$ is reversed (while keeping the basepoints of the upper circles and the total orders for $\up$ and $\low$ the same). Denote by $\bar{u}$ the circle $u$ with the opposite orientation.  The flip maps of $\kk$-vector spaces associate to the unique increasing map $\I_{\bar{u}}\to \I_u$ the \kt linear isomorphisms
$$
\vartheta_u \co \bigotimes_{s \in \I_{\bar{u}}}A_{\alpha(s)} \to \bigotimes_{s \in \I_u} A_{\alpha(s)} \quad \text{and} \quad 
\theta_u \co \bigotimes_{s \in \I_{\bar{u}}}A_{\alpha(s)^{\nu_s}} \to \bigotimes_{s \in \I_u} A_{\alpha(s)^{\nu_s}}
$$ 
Since the sign of each intersection point lying in $u$ is changed into its opposite under this $\CM$-move and the antipode $S$ is involutory, we have: $$\ant_u \theta_u=\vartheta_u S_x^{\otimes n} \ant_{\bar{u}},$$ where $n=|u|$ and $x=\alpha(u)$. Then, using the $S$-invariance of $\lambda$ and the anti-multiplicativity of  $S$, we deduce that 
$$
\lambda_{\bar{u}}\ant_{\bar{u}} = \lambda_{x^{-1}} \mu_{x^{-1}}^{n}\ant_{\bar{u}} = \lambda_{x} S_{x} \mu_{x^{-1}}^{n}\ant_{\bar{u}}  = \lambda_{x} \mu_x^n \vartheta_u  S_x^{\otimes n} \ant_{\bar{u}} = \lambda_u \ant_u \theta_u.
$$
This implies that the morphism $\lambda_\up \ant_\up Q_{\low,\up}$, and thus $K_A(D,(\alpha,\beta))$,  remains the same when reversing the orientation of $u$. 

Suppose next that the orientation of a lower circle $l$ is reversed (while keeping the basepoints of the upper circles and the total orders for $\up$ and $\low$ the same). Denote by~$\bar{l}$ the circle $l$ with the opposite orientation.  Let $\I_l=\{s_1< \cdots <s_n\}$ and $e=\beta(l)$. Set $x_i=\alpha(s_i)$ and $\nu_i=\nu_{s_i}$ for $1 \leq i \leq n$. Recall that $\CM(e)= x_1^{\nu_1} \cdots x_n^{\nu_n}$. 
The flip maps of $\kk$-vector spaces associate to the permutation $k \mapsto (n-k+1)$ of~$\{1, \dots,n\}$ the \kt linear homomorphisms 
$$
A_{x_1} \otimes \cdots \otimes A_{x_n} \xrightarrow{\theta_l} A_{x_n} \otimes \cdots \otimes A_{x_1}, \quad 
A_{x_1^{\nu_1}} \otimes \cdots \otimes A_{x_n^{\nu_n}} \xrightarrow{\vartheta_l} A_{x_n^{\nu_n}} \otimes \cdots \otimes A_{x_1^{\nu_1}}.
$$
Since the sign of each intersection point lying in $l$ is changed into its opposite under this $\CM$-move and the antipode $S$ is involutory, we have: 
$$ 
\theta_l\ant_l= \ant_{\bar{l}}\ps  S_{\bar{l}} \ps \vartheta_l \quad \text{where} \quad S_{\bar{l}}=S_{x_n^{-\nu_n}} \otimes \cdots \otimes S_{x_1^{-\nu_1}}.
$$ 
Now, by the $S$-invariance of $\Lambda$ and the commutativity of the $\CM$-action $\phi$ with $S$,  
$$
\phi_{1,e^{-1}} (\Lambda)=\phi_{1,e^{-1}} S_1(\Lambda)=S_{\CM(e)^{-1}}\phi_{1,e}(\Lambda).
$$
Then, together with the anti-comultiplicativity of  $S$, we deduce that 
\begin{align*}
\ant_{\bar{l}} \ps \Lambda_{\bar{l}}
& = \ant_{\bar{l}} \ps \Delta^n_{x_n^{-\nu_n}, \dots , x_1^{-\nu_1}}  \phi_{1,e^{-1}} (\Lambda) \\
& = \ant_{\bar{l}} \ps \Delta^n_{x_n^{-\nu_n}, \dots , x_1^{-\nu_1}} S_{\CM(e)^{-1}}\phi_{1,e}(\Lambda)\\
& = \ant_{\bar{l}} \ps S_{\bar{l}} \ps  \vartheta_l \Delta^n_{x_1^{\nu_1}, \dots , x_n^{\nu_n}} \phi_{1,e}(\Lambda)\\
& =\theta_l\ant_l \Lambda_l.
\end{align*}
This implies that the morphism $P_{\low,\up}\ant_\low (\Lambda_\low)$, and thus $K_A(D,(\alpha,\beta))$,  remains the same when reversing the orientation of $l$.\\ 
    
\pfpart{Invariance under two-point moves}  
Consider a two-point move between  an upper circle $u \in \up$ and a lower circle $l \in \low$. Up to reversing the orientation of these circles, moving the basepoint of $l$, and appropriately choosing the basepoint of $u$, we can suppose that the two added intersections points are the first two with respect to the orders of $\I_u$ and $\I_l$ and that the first one is positive (and so the second one is negative). The invariance of $K_A(D,(\alpha,\beta))$ under this $\CM$-move is then a consequence of the following equalities, where $x=\alpha(u)$ and $e=\beta(l)$:
\begin{gather*}
\psfrag{z}[Bl][Bl]{\scalebox{.8}{$x^{-1}$}}
\psfrag{c}[Br][Br]{\scalebox{.8}{$x$}}
\psfrag{u}[Br][Br]{\scalebox{.8}{$1$}}
\psfrag{x}[Bl][Bl]{\scalebox{.8}{$x$}}
\psfrag{e}[Bl][Bl]{\scalebox{.9}{$e$}}
\rsdraw{.5}{.9}{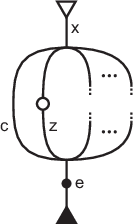} \;\;\; \overset{(i)}{=} \;\; \rsdraw{.5}{.9}{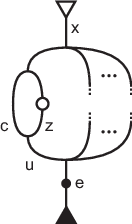}  \;\;\; \overset{(ii)}{=} \;\;\,  \rsdraw{.5}{.9}{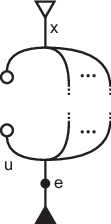} \;\;\, \overset{(iii)}{=} \;\;\,  \rsdraw{.5}{.9}{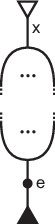} \;\; .
\end{gather*}
Here $(i)$ follows from the (co)associativity of the (co)product, $(ii)$ from the axiom of the antipode, and $(iii)$ from the (co)unitality of the (co)product.\\ 

\pfpart{Invariance under stabilizations}
This move adds a lower circle  labeled by some $e \in E$ and an upper circle  labeled by $\CM(e) \in H$. These added circles have a unique intersection point (which is positive) and so contribute to  $\lambda_\up P_{\low,\up} (\Omega_\low)$ by a factor of
$\lambda_{\CM(e)} \phi_{1,e}(\Lambda)=\lambda_1(\Lambda)= \dim_\kk(A_1)$.
Since this  contribution is  cancelled by the normalization term $\dim_\kk(A_1)^{g(\Sigma)-|\up|-|\low|}$, we deduce that $K_A(D,(\alpha,\beta))$  is invariant under this $\CM$-move.\\ 
    
\pfpart{Invariance under handle slides of upper circles}  
This move slides an upper circle~$u_1$ over another upper circle $u_2$ and creates new upper circles: $u_1' = u_1 \#_b u_2$ and a copy~$u'_2$ of $u_2$. Let $(D',(\alpha',\beta'))$ be the resulting $\CM$-Heegaard diagram. We choose basepoints for  $u_1$ and $u_2$ 
in such a way that the band connects the circles $u_1$ and~$u_2$  just before (with respect to their orientation) these basepoints.
We endow $u'_1$ and~$u'_2$ with the basepoints inherited from $u_1$ and $u_2$. Set $x=\alpha(u_1)$ and $y=\alpha(u_2)$, so that $\alpha(u_1')=x$ and $\alpha(u_2')=x^{-1}y$. Denote 
$\I_{u_2}=\{s_1<\dots<s_n\}$.
It follows from the definitions and the coassociativity of the coproduct that  $K_A(D',(\alpha',\beta'))$ is obtained from the formulation \eqref{eq-pf-def2} of~$K_A(D,(\alpha,\beta))$ by replacing $\lambda_{u_1}\ant_{u_1} \otimes \lambda_{u_2}\ant_{u_2}$ with
$$
\psfrag{x}[Bl][Bl]{\scalebox{.8}{$x$}}
\psfrag{z}[Bl][Bl]{\scalebox{.8}{$x^{-1}y$}}
\psfrag{c}[Br][Br]{\scalebox{.8}{$x$}}
\psfrag{A}[Bc][Bc]{\scalebox{.9}{$S_x^{\kappa_1}$}}
\psfrag{B}[Bc][Bc]{\scalebox{.9}{$S_x^{\kappa_n}$}}
\psfrag{D}[Bc][Bc]{\scalebox{.9}{$S_{x^{-1}y}^{\kappa_1}$}}
\psfrag{E}[Bc][Bc]{\scalebox{.9}{$S_{x^{-1}y}^{\kappa_n}$}}
\psfrag{C}[Bc][Bc]{\scalebox{.9}{$\ant_{u_1}$}}
\psfrag{m}[Bc][Bc]{\scalebox{.9}{$\delta_1$}}
\psfrag{w}[Bc][Bc]{\scalebox{.9}{$\delta_n$}}
\rsdraw{.45}{.9}{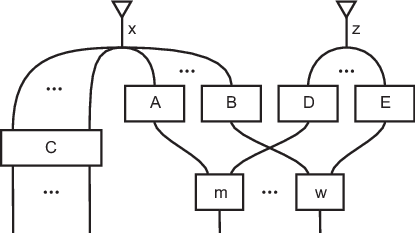} \;\,\;,
$$
where $\kappa_i=0$ and $\delta_i=\Delta_{x,x^{-1}y}$ if $s_i$ is positive, and  $\kappa_i=1$ and $\delta_i=\Delta_{x,x^{-1}y}^\cop$ if $s_i$ is negative. Here, by convention,   $S_z^{0}=\id_{A_z}$ for all $z \in H$.
Now
$$
\psfrag{x}[Bl][Bl]{\scalebox{.8}{$x$}}
\psfrag{v}[Bl][Bl]{\scalebox{.8}{$y$}}
\psfrag{z}[Bl][Bl]{\scalebox{.8}{$x^{-1}y$}}
\psfrag{c}[Br][Br]{\scalebox{.8}{$x$}}
\psfrag{A}[Bc][Bc]{\scalebox{.9}{$S_x^{\kappa_1}$}}
\psfrag{B}[Bc][Bc]{\scalebox{.9}{$S_x^{\kappa_n}$}}
\psfrag{D}[Bc][Bc]{\scalebox{.9}{$S_{x^{-1}y}^{\kappa_1}$}}
\psfrag{E}[Bc][Bc]{\scalebox{.9}{$S_{x^{-1}y}^{\kappa_n}$}}
\psfrag{R}[Bc][Bc]{\scalebox{.9}{$S_y^{\kappa_1}$}}
\psfrag{S}[Bc][Bc]{\scalebox{.9}{$S_y^{\kappa_n}$}}
\psfrag{C}[Bc][Bc]{\scalebox{.9}{$\ant_{u_1}$}}
\psfrag{U}[Bc][Bc]{\scalebox{.9}{$\ant_{u_2}$}}
\psfrag{m}[Bc][Bc]{\scalebox{.9}{$\delta_1$}}
\psfrag{w}[Bc][Bc]{\scalebox{.9}{$\delta_n$}}
\rsdraw{.45}{.9}{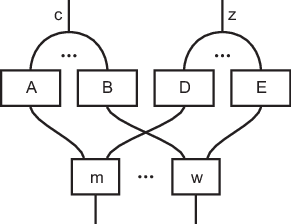} \;\;\,\overset{(i)}{=}\;\;
\rsdraw{.45}{.9}{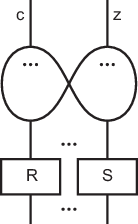} \;\;\overset{(ii)}{=}\;\;
\rsdraw{.45}{.9}{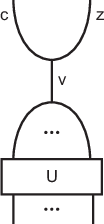} \;.
$$
Here $(i)$ follows from the equality $(S_x\otimes S_{x^{-1}y})\Delta_{x,x^{-1}y}^\cop=\Delta_{x,x^{-1}y} S_y $ and $(ii)$ from the multiplicativity of the coproduct. Consequently,  $K_A(D',(\alpha',\beta'))$ is obtained from~$K_A(D,(\alpha,\beta))$ by replacing $\lambda_{u_1}\ant_{u_1} \otimes \lambda_{u_2}\ant_{u_2}$ with
$$
\psfrag{x}[Bl][Bl]{\scalebox{.8}{$x$}}
\psfrag{v}[Bl][Bl]{\scalebox{.8}{$y$}}
\psfrag{z}[Bl][Bl]{\scalebox{.8}{$x^{-1}y$}}
\psfrag{C}[Bc][Bc]{\scalebox{.9}{$\ant_{u_1}$}}
\psfrag{U}[Bc][Bc]{\scalebox{.9}{$\ant_{u_2}$}}
\rsdraw{.4}{.9}{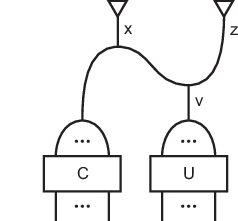}\;\;\overset{(i)}{=}\;\; \rsdraw{.4}{.9}{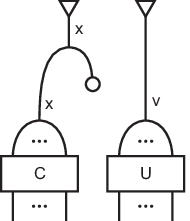} \;\,\overset{(ii)}{=}\; \lambda_{u_1}\ant_{u_1} \otimes \lambda_{u_2}\ant_{u_2}.
$$
Here $(i)$ follows from the fact that $\lambda$ is a two-sided $\chi$-integral and $(ii)$ from the unitality of the product.
Thus $K_A(D',(\alpha',\beta'))=K_A(D,(\alpha,\beta))$.\\ 

\pfpart{Invariance under handle slides  of lower circles} 
This move slides a lower circle $l_1$ over another lower circle $l_2$ and creates new lower circles: $l_1' = l_1 \#_b l_2$ and a copy~$l'_2$ of $l_2$. Let $(D',(\alpha',\beta'))$ be the resulting $\CM$-Heegaard diagram. Set $e=\beta(l_1)$ and $f=\beta(l_2)$, so that $\beta(l_1')=ef$ and $\beta(l_2')=f$. Denote  $\I_{l_2}=\{s_1<\dots<s_n\}$ and set $x_i=\alpha(s_i)$ for $1 \leq i \leq n$. 
It follows from the definitions and the associativity of the product that  $K_A(D',(\alpha',\beta'))$ is obtained from the formulation~ \eqref{eq-pf-def1} of~$K_A(D,(\alpha,\beta))$ by replacing $\ant_{l_1}(\Lambda_{l_1}) \otimes \ant_{l_2}(\Lambda_{l_2})$ with 
$$
\psfrag{x}[Br][Br]{\scalebox{.8}{$x_1$}}
\psfrag{z}[Bl][Bl]{\scalebox{.8}{$x_n$}}
\psfrag{A}[Bc][Bc]{\scalebox{.9}{$S_{x_1}^{\kappa_1}$}}
\psfrag{B}[Bc][Bc]{\scalebox{.9}{$S_{x_n}^{\kappa_n}$}}
\psfrag{C}[Bc][Bc]{\scalebox{.9}{$\ant_{l_1}$}}
\psfrag{m}[Bc][Bc]{\scalebox{.9}{$\mathrm{m}_1$}}
\psfrag{w}[Bc][Bc]{\scalebox{.9}{$\mathrm{m}_n$}}
\psfrag{e}[Bl][Bl]{\scalebox{.9}{$ef$}}
\psfrag{f}[Bl][Bl]{\scalebox{.9}{$f$}}
\rsdraw{.45}{.9}{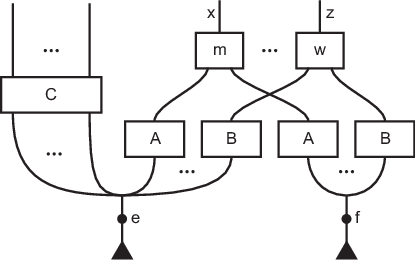} \;\,\;,
$$
where $\kappa_i=0$ and $\mathrm{m}_i=\mu_{x_i}$ if $s_i$ is positive, and  $\kappa_i=1$ and $\mathrm{m}_i=\mu_{x_i}^\opp=\mu_{x_i} \sigma_{A_{x_i},A_{x_i}}$ if $s_i$ is negative. Here, by convention,   $S_z^{0}=\id_{A_z}$ for all $z \in H$.
Now 
$$
\psfrag{E}[Br][Br]{\scalebox{.8}{$\CM(f)$}}
\psfrag{F}[Bl][Bl]{\scalebox{.8}{$\CM(f)$}}
\psfrag{x}[Br][Br]{\scalebox{.8}{$x_1$}}
\psfrag{z}[Bl][Bl]{\scalebox{.8}{$x_n$}}
\psfrag{A}[Bc][Bc]{\scalebox{.9}{$S_{x_1}^{\kappa_1}$}}
\psfrag{B}[Bc][Bc]{\scalebox{.9}{$S_{x_n}^{\kappa_n}$}}
\psfrag{m}[Bc][Bc]{\scalebox{.9}{$\mathrm{m}_1$}}
\psfrag{w}[Bc][Bc]{\scalebox{.9}{$\mathrm{m}_n$}}
\psfrag{U}[Bc][Bc]{\scalebox{.9}{$\ant_{l_2}$}}
\rsdraw{.45}{.9}{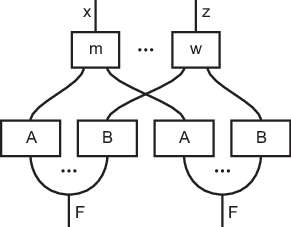} \;\;\,\overset{(i)}{=}\;\;
\rsdraw{.45}{.9}{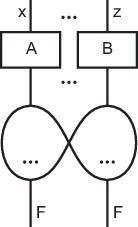} \;\;\overset{(ii)}{=}\;\;
\rsdraw{.45}{.9}{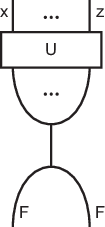} \;.
$$
Here $(i)$ follows from the equality $\mu_x^\opp(S_x\otimes S_x)=S_x\mu_{x^{-1}}$ and $(ii)$ from the multiplicativity of the coproduct. Also,
$$
\psfrag{x}[Bl][Bl]{\scalebox{.8}{$\CM(e)$}}
\psfrag{z}[Bl][Bl]{\scalebox{.8}{$\CM(f)$}}
\psfrag{e}[Bl][Bl]{\scalebox{.9}{$ef$}}
\psfrag{b}[Bl][Bl]{\scalebox{.9}{$e$}}
\psfrag{f}[Bl][Bl]{\scalebox{.9}{$f$}}
\psfrag{a}[Br][Br]{\scalebox{.9}{$e$}}
\rsdraw{.45}{.9}{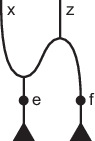} \;\;\overset{(i)}{=}\;\;\,
\rsdraw{.45}{.9}{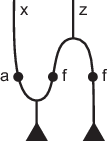} \;\;\,\overset{(ii)}{=}\;\;
\rsdraw{.45}{.9}{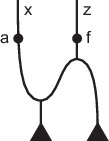} \;\overset{(iii)}{=}\;
\rsdraw{.45}{.9}{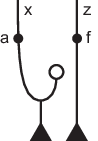} \;\overset{(iv)}{=}\;
\rsdraw{.45}{.9}{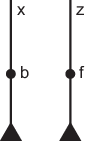} .
$$
Here $(i)$ and $(ii)$ follow from the axioms of a $\CM$-action, $(iii)$ from the fact that $\Lambda$ is a two-sided integral element  of $A_1$, and $(iv)$ from the counitality of the coproduct. Consequently,
$K_A(D',(\alpha',\beta'))$ is obtained from~$K_A(D,(\alpha,\beta))$ by replacing $\ant_{l_1}(\Lambda_{l_1}) \otimes \ant_{l_2}(\Lambda_{l_2})$ with
$$
\psfrag{x}[Bl][Bl]{\scalebox{.8}{$\CM(e)$}}
\psfrag{z}[Bl][Bl]{\scalebox{.8}{$\CM(f)$}}
\psfrag{e}[Bl][Bl]{\scalebox{.9}{$ef$}}
\psfrag{f}[Bl][Bl]{\scalebox{.9}{$f$}}
\psfrag{b}[Bl][Bl]{\scalebox{.9}{$e$}}
\psfrag{a}[Br][Br]{\scalebox{.9}{$e$}}
\psfrag{C}[Bc][Bc]{\scalebox{.9}{$\ant_{l_1}$}}
\psfrag{U}[Bc][Bc]{\scalebox{.9}{$\ant_{l_2}$}}
\rsdraw{.4}{.9}{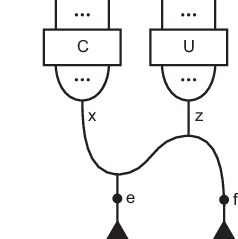} \;\, = \; 
\rsdraw{.4}{.9}{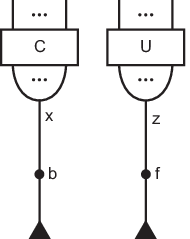} \; = \; 
\ant_{l_1}(\Lambda_{l_1}) \otimes \ant_{l_2}(\Lambda_{l_2}).
$$
Thus $K_A(D',(\alpha',\beta'))=K_A(D,(\alpha,\beta))$.\\

\pfpart{Invariance under adding trivial circles} 
Adding a trivial upper circle labeled by some $x \in H$ contributes a factor of $\lambda_x(1_x)=\dim_{\kk}(A_1)$ to  $\lambda_\up P_{\low,\up} (\Omega_\low)$. Adding a trivial lower circle labeled by $1 \in E$ contributes a  factor of $\varepsilon(\Lambda)=\dim_{\kk}(A_1)$ to  $\lambda_\up P_{\low,\up} (\Omega_\low)$. 
Since these contributions are  cancelled by the normalization term $\dim_\kk(A_1)^{g(\Sigma)-|\up|-|\low|}$, we deduce that $K_A(D,(\alpha,\beta))$  is invariant under this $\CM$-move.

\appendix

\section{Principal 2-bundles over topological spaces}\label{sect-2-bundles-appendix}

In this appendix, we review the basics of 2-bundles over topological spaces (in the spirit of~\cite{Wo}) and their relationship with \v{C}ech cohomology.

\subsection{Internal groupoids}\label{sect-internal-groupoids} 
A \emph{groupoid internal to} a category $\cc$ with pullbacks consists of an object $P_0$ of $\cc$ (the object of \emph{objects}), an object $P_1$ of $\cc$ (the object of \emph{morphisms}), and the following structure morphisms in $\cc$: the \emph{source} $s \co P_1 \to P_0$, the \emph{target} $t \co P_1 \to P_0$, the \emph{product} (or \emph{composition}\footnote{As for usual groupoids (see Appendix~\ref{sect-groupoids}),  our notation $\ast$ for the product of internal groupoids is opposite to the standard notation $\circ$ for composition: $f \ast g=g \circ f$ whenever $tf=sg$.}) $\ast \co P_1 \times_{P_0} P_1 \to P_1$, where $P_1 \times_{P_0} P_1$ is the pullback of the pair $(t,s)$, the \emph{unit} morphism $P_0 \to P_1$, and the \emph{inversion} $P_1 \to P_1$, satisfying the usual axioms of a groupoid. We denote such a groupoid by $P=(P_1 \rightrightarrows P_0)$ where the arrows stand for the source and the target morphisms. 

A \emph{functor} $f \co P\to Q$  between groupoids $P$ and $Q$ internal to $\cc$ is a pair $f=(f_0 \co P_0 \to Q_0, f_1 \co P_1 \to Q_1)$ of morphisms in $\cc$ satisfying the usual axioms of a functor (preservation of source, target, product, and units). Such a functor automatically preserves the inversion. For example, $\id_{P}=(\id_{P_0}, \id_{P_1})$ is a functor from $P$ to itself called the \emph{identity functor}. 

A \emph{natural transformation} $\alpha \co f \Rightarrow g$ between two functors $f,g \co P\to Q$ is a morphism $\alpha \co P_0 \to Q_1$   in $\cc$ satisfying the usual axioms of naturality. Note that since~$Q$ is a groupoid, any such natural transformation is in fact a natural isomorphism.
For example, given a  functor $f=(f_0,f_1) \co P\to Q$, the composition of $f_0$ with the unit morphism of  $Q$ is a natural transformation $f \Rightarrow f$ denoted  by $1_f$.

For a detailed discussion of the definitions above, we refer for example to~\cite{MF} and~\cite{Mi}.

\subsection{Topological 2-spaces}\label{sect-2-spaces}
A \textit{topological 2-space} is a groupoid 
internal to the category of topological spaces, that is, a groupoid  $P=(P_1 \rightrightarrows P_0)$ where the set~$P_0$ of objects and the set $P_1$ of morphisms are topological spaces and the groupoid structure morphisms are continuous maps. For example, any topological space $\B$ gives rise to the topological $2$-space  $\widebar{\B}=(\B \rightrightarrows \B)$ where both the source and the target maps are the identity. 

A \emph{continuous functor} between topological 2-spaces is a functor internal to the category of topological spaces, that is, a functor whose maps between objects and morphisms are continuous.

\subsection{2-groups}\label{sect-2-groups}
A (strict) \emph{2-group} is a groupoid internal to the category of groups. In a 2-group $\tg=(\tg_1 \rightrightarrows \tg_0)$, the product and the inversion  are computed from the group structures of $\tg_1$ and $\tg_0$ as follows: for all $g,h \in \tg_1$ with $t(g)=s(h)$,
$$
g \ast h=h 1_{t(g)^{-1}} g \quad \text{and} \quad \bar{g}=1_{s(g)}g^{-1}1_{t(g)},
$$
where $1_x$ denotes the unit associated with $x \in \tg_0$.

There is a bijective correspondence between strict 2-groups and crossed modules. In this correspondence, the 2-group associated with a crossed module $\CM\co E \to H$ (see Section~\ref{sect-crossed-modules-def}) is the groupoid
$$
\tg_\CM=( H \ltimes E   \rightrightarrows H )
$$
with source, target, and product given by
$$
s(x,e)=x, \quad t(x,e)=\CM(e)x, \quad (x,e) \ast (\CM(e)x,f)=(x,fe).
$$
Conversely, the crossed module associated with a 2-group $\tg=(\tg_1 \rightrightarrows \tg_0)$ is the restriction $t\co \te \to \tg_0$  of the target map of $\tg$ to the kernel $\te =\Ker(s)\subset \tg_1$ of the source map of $\tg$. Here the action of the group $\tg_0$ on $\te$ is given by  $\lact{x}{e}=1_x   e 1_{x^{-1}}$ for all $(x,e) \in \tg_0 \times \te$, where $1_x$ denotes the unit associated with~$x$.

\subsection{Topological 2-groups}\label{sect-top-2-groups}
A (strict) \emph{topological $2$-group} is a groupoid internal to the category of topological groups, that is, a 2-group where the sets of objects and morphisms are topological groups and the groupoid structure homomorphisms are continuous. 

A topological $2$-group is \emph{discrete} if the groups of objects and morphisms are discrete (as topological spaces).
Note that any 2-group $\tg=(\tg_1 \rightrightarrows \tg_0)$ induces a discrete topological $2$-group by endowing $\tg_1$ and $\tg_0$ with the discrete topology. In particular, a crossed module $\CM$ induces the discrete topological $2$-group $\tg_\CM$.

\subsection{Actions of topological 2-groups}\label{sect-action}
A \emph{continuous right action} of a topological 2-group $\tg$ on a topological 2-space~$P$ is a continuous functor $a \co P \times \tg \to P$ such that
$$
a(\id_P \times \mu)=a(a \times \id_\tg)  \quad  \text{and} \quad a(\id_P \times \eta)=1_P,
$$
where the functor $\mu\co \tg \times \tg \to \tg$ is given by the group products of $\tg_0$ and $\tg_1$ and the functor $\eta \co \widebar{\{1\}} \to \tg$ sends $1$ to the unit elements of $\tg_0$ and $\tg_1$. 

A \emph{$\tg$-2-space} is a topological 2-space endowed with a continuous right action of~$\tg$. For example, for any topological 2-space $P$, the product  $P \times \tg$ is a $\tg$-2-space where~$\tg$ acts by right multiplication on the second factor.

A \emph{$\tg$-functor} between two $\tg$-2-spaces is a continuous functor which is $\tg$-equi\-va\-riant in the sense that it (strictly) preserves the actions of $\tg$. 
A \emph{$\tg$-pseudo-inverse} of a $\tg$-functor $f \co P \to Q$ is a $\tg$-functor $g \co Q \to P$ such that there exist continuous natural transformations (actually isomorphisms by Section~\ref{sect-internal-groupoids}) $\varepsilon \co gf \Rightarrow \id_P$ and  $\eta\co \id_Q \Rightarrow fg$ which are $\tg_0$-equivariant in the sense that for all $p \in P_0$, $q \in Q_0$, and~$x \in \tg_0$,
$$
\varepsilon_{p \cdot x}= \varepsilon_{p}\cdot 1_x \quad \text{and} \quad \eta_{q \cdot x}= \eta_{q}\cdot 1_x.
$$ 

A \emph{$\tg$-equivalence} is a $\tg$-functor  between $\tg$-2-spaces which has a $\tg$-pseudo-inverse.
Note that if $f \co P \to Q$ is a $\tg$-equivalence with $\tg$-pseudo-inverse $g \co Q \to P$ and if $\varepsilon \co gf \Rightarrow \id_P$  is a $\tg_0$-equivariant continuous natural isomorphism, then there exists a unique $\tg_0$\ti equivariant continuous natural isomorphism  $\eta\co \id_Q \Rightarrow fg$  making $(g, f, \eta, \varepsilon)$ an adjoint equivalence.

\subsection{Principal 2-bundles over topological spaces}\label{sect-principal-2-bundles}
Let $\tg$ be a topological $2$\ti group and $\B$ be a topological space. 

A \emph{trivializing chart} for a continuous functor $\pi \colon P \to \widebar{\B}$, where $P$ is a $\tg$-2-space and $\widebar{\B}$ is defined in Section~\ref{sect-2-spaces}, is a quadruple $(U,\psi, \phi,\varepsilon)$, where $U$ is an open subset of $\B$, $\psi \co P|_{U} \to \widebar{U} \times \tg$ is a $\tg$-equivalence, $\phi \co  \widebar{U} \times \tg\to P|_{U}$ is a $\tg$-pseudo-inverse of $\psi$, and 
$\varepsilon \co \phi \psi \Rightarrow  \id_{P|_{U}}$ is a  $\tg_0$-equivariant continuous natural isomorphism
such that the diagram 
$$
\xymatrix@R=.7cm @C=.9cm{
P|_{U} \ar@{->}[rd]_-{\pi_{|U}} \ar@{->}[r]^-{\psi} & {\widebar{U}} \times \tg \ar@{->}[r]^-{\phi} \ar@{->}[d]^-{\mathrm{pr}_1} & P|_{U} \ar@{->}[ld]^-{\pi_{|U}}\\
&  {\widebar{U}} &
}
$$
commutes. Here, $P|_{U}$ is the full subgroupoid of $P$ on objects mapped to $U$ under~$\pi$, $\pi_{|U}$ is the restriction of $\pi$, and $\mathrm{pr}_1$ is the projection onto the first factor. Recall from Section~\ref{sect-action} that this data uniquely determines a $\tg_0$-equivariant continuous natural isomorphism  $\eta\co \id_{\widebar{U} \times \tg} \Rightarrow  \psi \phi $  making $(\phi, \psi, \eta, \varepsilon)$ an adjoint equivalence. 

An \emph{atlas} for $\pi$ is a family $\aaa=\{(U_i,\psi_i, \phi_i,\varepsilon_i)\}_{i \in I}$ of trivializing charts such that $\{U_i\}_{i \in I}$ covers $\B$.  A \emph{trivializing cover} for $\pi$ is an open cover of $\B$ which extends to an atlas for~$\pi$.

A \emph{principal $\tg$-bundle} over $\B$ is a continuous functor $\pi \colon P \to \widebar{\B}$, where $P$ is a $\tg$-2-space, which admits an atlas. For example, the first projection $\widebar{\B} \times \tg \to \widebar{\B}$ is a principal $\tg$-bundle over $\B$, called the \emph{trivial $\tg$-bundle over $X$}.

A \emph{$\tg$-bundle morphism} from a principal $\tg$-bundle   $\pi \colon P \to \widebar{\B}$ to a principal $\tg$\ti bundle  $\pi' \colon P' \to \widebar{\B}$ is a $\tg$-functor $f \co P \to P'$  making the  diagram
$$
\xymatrix@R=.8cm @C=.5cm{
P \ar@{->}[rr]^-{f}  \ar@{->}[rd]_-{\pi} & &  P'\ar@{->}[ld]^-{\!\pi'} \\
&  {\widebar{\B}} &
}
$$
commutative. 

Two principal $\tg$-bundles $\pi$ and $\pi'$ over $\B$ are (Morita) \emph{equivalent} if there is a span between them in the category of principal $\tg$-bundles over $\B$, meaning that there exist a principal $\tg$-bundle  $\widetilde{\pi}$ over $\B$ and $\tg$-bundle morphisms $\widetilde{\pi} \to \pi$ and $\widetilde{\pi} \to \pi'$. In particular, if there is a $\tg$-bundle morphism between two principal $\tg$-bundles over~$\B$, then they are equivalent.

A principal $\tg$-bundle over $\B$ is \emph{trivializable} if it is equivalent to the  trivial $\tg$\ti bundle over $\B$.

\subsection{\v{C}ech cohomology}\label{sect-cocyle-descript}
Let  $\tg=(\tg_1 \rightrightarrows \tg_0)$ be a topological 2-group. We endow the kernel $\te \subset \tg_1$  of the source map of $\tg$ with the induced topology. Note that the crossed module $t\co \te \to \tg_0$ (given by the restriction of the target map, see Section~\ref{sect-2-groups}) is continuous. 

Let $\B$ be a topological space and $\uu=\{U_i\}_{i \in I}$ be an open cover of $\B$. For finitely many indices $i_1, \dots,i_n \in I$, denote 
$$
U_{i_1 \cdots i_n}=U_{i_1} \cap  \cdots \cap U_{i_n}.
$$
A \emph{$\tg$-valued \v{C}ech cocycle} subordinate to $\uu$ is a pair 
$$
\Bigl(\x=(\x_{ij} \co U_{ij} \to \tg_0)_{i,j \in I}\, , \, \e=(\e_{ijk}\co U_{ijk} \to \te)_{i,j,k \in I} \Bigr) 
$$
of families of continuous maps such that for all $i,j,k,l \in I$,
$$
t(\e_{ijk})\x_{ij}\x_{jk}=\x_{ik} \, \text{ on $U_{ijk}$} \quad  \text{and} \quad \e_{ikl}\e_{ijk}=\e_{ijl}\bigl(\lact{\x_{ij}}{\!\e_{jkl}}\bigr) \, \text{ on $U_{ijkl}$.}
$$
Two such cocycles $(\x,\e)$ and $(\x',\e')$ are \emph{cohomologous} if there is a pair 
$$
\Bigl(a=(a_i \co U_i \to \tg_0)_{i \in I}\, , \, d=(d_{ij}\co U_{ij} \to \te)_{i,j \in I} \Bigr) 
$$
of families of continuous maps such that\footnote{Given a map $\alpha \co A \to G$ from a set $A$ to a group $G$, we denote by $\alpha^{-1} \co A \to G$ the map sending any $a \in A$ to $\alpha(a)^{-1} \in G$. } for all $i,j,k \in I$,
$$
\x_{ij}'=a_i t(d_{ij})\x_{ij}a_j^{-1}  \, \text{ on $U_{ij}$} \quad \text{and} \quad \e_{ijk}'= \lact{a_i}{\ns\bigl(}d_{ik}\e_{ijk}\lact{\x_{ij}}{\!(}d_{jk}^{-1})d_{ij}^{-1}  \bigr)  \, \text{ on $U_{ijk}$.}
$$
Being cohomologous is an equivalence relation on the set of $\tg$-valued \v{C}ech cocycles  subordinate to $\uu$. Denote the set of equivalence classes by $\check{H}(\uu,\tg)$. Any refinement~$\vv$ of $\uu$ gives rise to a map $\check{H}(\uu,\tg) \to \check{H}(\vv,\tg)$ induced by restriction. The  \emph{\v{C}ech cohomology} of $\B$ with coefficients in $\tg$ is the colimit
$$
\check{H}(\B,\tg) = \varinjlim_{\uu} \check{H}(\uu,\tg)
$$
taken over the open covers of $\B$. It comes with universal cocone maps 
$$
\check{H}(\uu,\tg) \to \check{H}(\B,\tg)
$$
which are compatible  with the restriction maps. Note that if $\uu$ is a good open cover of $\B$ (meaning that all nonempty finite intersections of open sets in $\uu$ are contractible), then the cocone map is a bijection $\check{H}(\uu,\tg) \cong \check{H}(\B,\tg)$. 

Recall that a topological space admits good covers if any open cover has a good open cover that refines it.
Recall also that a topological group $G$ is well-pointed if the inclusion $\{1\} \hookrightarrow G$ is a closed cofibration. The  topological   2-group $\tg$ is said to be \emph{well-pointed} if both~$\tg_0$ and $\te$ are well-pointed. Baez and Stevenson~\cite{BS} proved that if $\tg$ is well-pointed and $\B$ is a paracompact Hausdorff space admitting good covers, then  there is a canonical bijection
$$
\check{H}(\B,\tg) \cong [\B, B|\tg|]
$$
between the \v{C}ech cohomology of $\B$ with coefficients in $\tg$ and the set of homotopy classes of maps $\B \to B|\tg|$. Here $|\tg|$ is the topological group defined as the geometric realization of the nerve of $\tg$ and $B|\tg|$ is its classifying space.

\subsection{From \v{C}ech cocycles to 2-bundles}
Let $\tg=(\tg_1 \rightrightarrows \tg_0)$ be a topological 2-group and $\uu=\{U_i\}_{i \in I}$ be an open cover of a topological space $\B$. We associate to  any  $\tg$-valued \v{C}ech cocycle $c=(\x,\e)$  subordinate to $\uu$ a principal $\tg$-bundle $$\pi_c \co P_c=(P_1 \rightrightarrows  P_0)\to \widebar{\B}$$ defined as follows.
Consider the topological spaces
$$
P_0 = \coprod_{i \in I} U_i \times \tg_0 \quad \text{and} \quad P_1 =\coprod_{i,j \in I} U_{ij} \times \tg_1
$$
where $U_{ij}=U_i \cap U_j$ as in Section~\ref{sect-cocyle-descript}. An object  $(u,x) \in U_i \times \tg_0$ is denoted by~$(u,x)_i$, and a morphism $(v,g) \in U_{ij} \times \tg_1$ is denoted by $(v,g)_{ij}$. The source and target maps $s,t \co P_1 \to P_0$ are defined by 
$$ 
s\bigl((v,g)_{ij}\bigr)= \bigl(v,s(g)\bigr)_{\ns i} \quad \text{and} \quad t\bigl((v,g)_{ij}\bigr)=\bigl(v,\x_{ij}^{-1}(v)t(g)\bigr)_{\ns j}.
$$ 
The product of morphisms is defined by
$$
(v,g)_{ij} \ast (v,h)_{jk}=(v,\ell)_{ik} \quad \text{with} \quad  \ell=\e_{ijk}(v) \bigl(g \ast (1_{\x_{ij}(v)} h)\bigr)  \in \tg_1
$$
for all $v \in U_{ijk}$ and $g,h \in \tg_1$ such that $t(g)=\x_{ij}(v)s(h)$. The units and the inverses are given by
$$
1_{(u,x)_i}=\bigl(u,\e_{iii}^{-1}(u)1_{x}\bigr)_{\ns ii} \quad \text{and} \quad  \widebar{(v,g)}_{ij}=\bigl(v,1_{\x_{ij}^{-1}(v)}\e_{iji}^{-1}(v)\e_{iii}^{-1}(v) \widebar{g}\bigr)_{\ns ji}.
$$
Using the expression of the product and inversion  of $ \tg$ by means of those of the groups~$\tg_0$ and $\tg_1$ (see Section~\ref{sect-2-groups}), we have:
$$
g \ast (1_{\x_{ij}(v)} h)=1_{\x_{ij}(v)} h 1_{t(g)^{-1}}g \quad \text{and} \quad 
\widebar{g}=1_{s(g)} g^{-1} 1_{t(g)} .
$$
The right action of $\tg$ on $P_c$ is defined by
$$
(u,x)_i \cdot y=(u,xy)_i  \quad \text{and} \quad  (v,g)_{ij} \cdot h=(v,gh)_{ij}. 
$$
Finally, the functor $\pi_c \co P_c\to \widebar{\B}$ is defined by 
$$
\pi_c\bigl((u,x)_i\bigr)=u  \quad \text{and} \quad  \pi_c\bigl((v,g)_{ij}\bigr)=v. 
$$

\begin{lem}\label{lem-cocycle-to-principal}
$\pi_c$ is a principal $\tg$-bundle over $\B$ admitting $\uu$ as trivializing cover.
\end{lem}
\begin{proof}
The axioms of a \v{C}ech cocycle imply that $P_c$ is a $\tg$-2-space and $\pi_c \co P_c\to \widebar{\B}$ is a continuous functor. Let us prove that $\uu=\{U_i\}_{i \in I}$ is a trivializing cover for~$\pi_c$. Pick $i \in I$. Consider  $\psi_i \co P_c|_{U_i} \to \widebar{U}_i \times \tg$ and $\phi_i \co  \widebar{U}_i \times \tg\to P_c|_{U_i}$ defined by
\begin{align*}
& \psi_i\bigl ((v,x)_j \bigr )=\bigl(v,\x_{ij}(v)x\bigr), && \phi_i(u,x)=(u,x)_i, \\
& \psi_i\bigl ((w,g)_{jk} \bigr )=\bigl(w,\e_{ijk}(w)1_{\x_{ij}(w)}g\bigr), && \phi_i(u,g)=\bigl(u,\e_{iii}^{-1}(u)g \bigr)_{\ns ii}, 
\end{align*}
for all $j,k,l\in I$,  $u \in U_i$, $v \in U_{ij}$, $w \in U_{ijk}$, $x \in \tg_0$, $g \in \tg_1$. 
Then $\psi_i$ and~$\phi_i$ are $\tg$\ti functors and are $\tg$-pseudo-inverses of each other via the $\tg_0$-equivariant continuous natural isomorphisms 
$\epsilon_i \co \phi_i \psi_i\Rightarrow\id_{P_c|_{U_i}}$ and $\eta_i \co \id_{\widebar{U}_i  \times \tg} \Rightarrow \psi_i \phi_i$ given by
$$
\epsilon_i\bigl((v,x)_j\bigr)=\bigl(v,1_{\x_{ij}(v)x} \bigr)_{\ns ij}  \quad \text{and} \quad  \eta_i(u,x)=\bigl(u,\e_{iii}^{-1}(u)1_x\bigr)
$$
for all $j\in I$,  $u \in U_i$, $v \in U_{ij}$, $x \in \tg_0$. Moreover, $\mathrm{pr}_1  \circ \psi_i=\pi_{c|U_i}$ and $\pi_{c|U_i} \circ \phi_i=\mathrm{pr}_1$, where $\mathrm{pr}_1$ is the projection onto the first factor. Consequently $(U_i,\psi_i,\phi_i,\varepsilon_i)$ is a trivializing chart for $\pi_c$.
\end{proof}

\begin{lem}\label{lem-morphisms-Pc}
Let $c,c'$ be $\tg$-valued \v{C}ech cocycles subordinate to $\uu$. Then $c$ and $c'$ are cohomologous if and only if there is a $\tg$-bundle morphism between $\pi_c$ and $\pi_{c'}$.
\end{lem}
\begin{proof}
Denote  $c=(\x, \e)$ and $c'=(\x',\e')$.  Assume first that $c$ and $c'$  are cohomo\-logous via  some pair $(a,d)$. For all $(u,x)_i \in (P_c)_0$ and $(v,g)_{ij} \in (P_c)_1$, set
$$
f \bigl( (u,x)_i \bigr) = \bigl(u, a_i(u)x\bigr)_{\ns i} \quad \text{and} \quad f \bigl(  (v,g)_{ij}  \bigr) = \bigl(v, 1_{a_{i}(v)} d_{ij}(v) g\bigr)_{\ns ij}.
$$
Then this defines a functor $f \co P_c \to P_{c'}$ which is a $\tg$-bundle morphism $\pi_c \to \pi_{c'}$.

Conversely, assume that there is a $\tg$-bundle morphism $f\co \pi_c \to \pi_{c'}$. The condition $\pi_{c} = \pi_{c'} \circ f$ implies that there exists a map $\alpha \co I \to I$ such that 
$$
U_{i} \subset U_{\alpha(i)}, \quad f \bigl( (u,x)_i \bigr) \in U_{\alpha(i)} \times \tg_0, \quad \text{and} \quad f \bigl( (v,g)_{ij} \bigr)  \in U_{\alpha(i)\alpha(j)} \times \tg_1
$$  
for all  $i,j \in I$, $(u,x)_i \in (P_c)_0$, and $(v,g)_{ij} \in (P_c)_1$.  The continuity and $\tg$\ti equiva\-rian\-ce of $f$ imply that  there are continuous  maps $b_i \co U_i \to \tg_0$ and $h_{ij} \co U_{ij} \to \tg_1$ such that
$$
f \bigl( (u,x)_i \bigr) = \bigl(u,b_i(u)x\bigr)_{\ns \alpha(i)} \quad \text{and} \quad
f \bigl( (v,g)_{ij} \bigr) = \bigl(v, h_{ij}(v)g\bigr)_{\ns \alpha(i)\alpha(j)}. 
$$
Finally consider the continuous maps $a_i \co U_i \to \tg_0$ and $d_{ij} \co U_{ij} \to \te$ defined by
$$
a_i (u)= \bigl(\x_{\alpha(i)i}'(u)\bigr)^{-1}b_i(u), \quad
d_{ij}(v) = 1_{b_i(v)^{-1}} \bigl(\e_{\alpha(i)ij}'(v)\bigr)^{-1} \e_{\alpha(i)\alpha(j)j}(v) h_{ij}(v).
$$
Then, for all $i,j,k \in I$, 
$$
\x_{ij}'= a_i t(d_{ij}) \x_{ij} a_{j}^{-1} \, \text{ on $U_{ij}$,} \quad   \e_{ijk}' = \lact{a_i}{\!\Bigl(} d_{ik} \e_{ijk} \lact{\x_{ij}}{\!\bigl(}d_{jk}^{-1}\bigr) d_{ij}^{-1} \Bigr) \, \text{ on $U_{ijk}$},
$$
and so $c$ and $c'$ are cohomologous via the pair $(a,d)$.
\end{proof}

\begin{lem}\label{lem-restriction-cocycle-principal}
Let $c$ be a $\tg$-valued \v{C}ech cocycle subordinate to $\uu$ and let $\vv$ be a refinement of $\uu$. 
Denote by $c_{|\vv}$ the $\tg$-valued \v{C}ech cocycle subordinate to~$\vv$ induced by restriction of $c$. Then there is a $\tg$-bundle morphism $\pi_{c_{|\vv}} \to \pi_c$.
\end{lem}
\begin{proof}
Denote $c=(\x,\e)$ and $\vv=\{V_a\}_{a\in A}$. Recall that a refinement map for $\vv$ is a map $\alpha \co A \to I$ such that   $V_a \subset U_{\alpha(a)}$ for all $a\in A$. Any such  map $\alpha$  induces the restricted cocycle $c_{|\vv}=c_{|\vv,\alpha}=(\x',\e')$ defined by 
$$
\x'_{ab}(v)= \x_{\alpha(a)\alpha(b)}(v) \quad \text{and} \quad \e'_{abc}(w) = \e_{\alpha(a)\alpha(b)\alpha(c)}(w)
$$
for all $a,b,c \in A$, $v \in V_{ab}$, and $w \in V_{abc}$. Then the $\tg$-bundle morphism $\pi_{c_{|\vv}} \to \pi_c$  is given by the functor $f\co P_{c_{|\vv}} \to P_c$ defined by 
$$
f\bigl( (u,x)_a \bigr) = (u,x)_{\alpha(a)} \quad \text{and} \quad f \bigl( (v,g)_{ab} \bigr) = (v,g)_{\alpha(a)\alpha(b)}
$$
for all $(u,x)_a \in (P_{c_{|\vv}})_0$ and $(v,g)_{ab} \in (P_{c_{|\vv}})_1$.
\end{proof}

\begin{lem}\label{lem-restriction-decomposition}
Let $c$ be a $\tg$-valued \v{C}ech cocycle subordinate to $\uu$ and $d$ be a $\tg$-valued \v{C}ech cocycle subordinate to a refinement $\vv$ of $\uu$. Then any $\tg$-bundle morphism $\pi_d \to \pi_c$ decomposes as a composition $\pi_d \to \pi_{c_{|\vv}} \to \pi_c$ of $\tg$-bundle morphisms.
\end{lem}
\begin{proof} 
Denote $c=(\x,\e)$, $\vv=\{V_a\}_{a\in A}$, and  $d=(\x', \e')$. As in the proof of Lemma~\ref{lem-morphisms-Pc}, a $\tg$-bundle morphism $f\co \pi_d \to \pi_c$ determines a refinement map $\alpha \co A \to I$ and continuous maps $b_a \co V_a \to \tg_0$ and $h_{ab} \co V_{ab} \to \tg_1$ such that
$$
f \bigl( (u,x)_a \bigr) = \bigl(u,b_a(u)x \bigr)_{\ns \alpha(a)} \quad \text{and} \quad f \bigl( (v,g)_{ab} \bigr) = \bigl(v,h_{ab}(v)g \bigr)_{\ns \alpha(a)\alpha(b)}.
$$
By Lemma~\ref{lem-restriction-cocycle-principal}, the map  $\alpha$ determines a $\tg$-bundle morphism $\iota \co \pi_{c_{|\vv}}=\pi_{c_{|\vv,\alpha}} \to \pi_c$. Let $q \co  P_d \to P_{c_{|\vv,\alpha}}$ be the functor defined by
$$
q \bigl( (u,x)_a \bigr) = \bigl(u, b_a(u)x\bigr)_{\ns a} \quad \text{and} \quad q \bigl( (v,g)_{ab} \bigr) = \bigl(v,h_{ab}(v)g\bigr)_{\ns ab}.
$$
Then $q$ is a $\tg$-bundle morphism $\pi_d \to \pi_{c_{|\vv}}$ such that $f=\iota\circ q $.
\end{proof}

\subsection{From 2-bundles to \v{C}ech cocycles}\label{sect-bundle-to-cocycle} 
Let $\tg$ be a topological 2-group and $\pi \co P\to \widebar{\B}$ be a principal $\tg$-bundle over a topological space $\B$. Pick an atlas
$$
\aaa=\bigl\{\bigl(U_i,\, \psi_i \co P|_{U_i} \to \widebar{U}_{\!i} \times \tg , \,  \phi_i \co \widebar{U}_{\!i} \times \tg \to P|_{U_i}, \, \varepsilon_i \co \phi_{i}\psi_{i} \Rightarrow \id_{P|_{U_i}} \bigr)\bigr\}_{i \in I}
$$ 
of $\pi$  (see Section~\ref{sect-principal-2-bundles}). We derive from $\aaa$ a $\tg$-valued \v{C}ech cocycle $c_\aaa$ subordinate to the open cover $\uu=\{U_i\}_{i \in I}$ of $\B$ defined as follows.

For any $i,j \in I$, consider the transition functor
$$
T_{ij} \co \widebar{U}_{\!ij} \times \tg \to \widebar{U}_{\!ij} \times \tg
$$
defined as the composition
$\widebar{U}_{\!ij} \times \tg \xrightarrow{\phi_j} P|_{U_{ij}}  \xrightarrow{\psi_i} \widebar{U}_{\!ij} \times \tg$. 
Since $T_{ij}$ is continuous and commutes with the projection onto $\widebar{U}_{\!ij}$, there is a continuous map $\x_{ij} \co U_{ij} \to \tg_0$ such that for all $u \in U_{ij}$, 
$$
T_{ij}(u,1_{\tg_0})=\bigl(u,\x_{ij}(u)\bigr).
$$
Note that $\x_{ij}$ entirely determines $T_{ij}$ since its functoriality and $\tg$-equivariance imply that for all $u \in U_{ij}$, $x \in \tg_0$, and $g \in \tg_1$,
$$
T_{ij}(u,x)=\bigl(u,\x_{ij}(u)x\bigr) \quad \text{and} \quad T_{ij}(u,g)=\bigl(u,1_{\x_{ij}(u)}g\bigr).
$$

Next, for all $i,j,k \in I$ and $(u,x) \in U_{ijk}\times \tg_0$, set
$$
\gamma_{ijk}(u,x)=\psi_i\bigl(\varepsilon_j\phi_k(u,x)\bigr) \in U_{ijk}\times \tg_1.
$$
Then $\gamma_{ijk}$ is a $\tg$-equivariant continuous natural isomorphism from $T_{ij}T_{jk}$ to $T_{ik}$ both restricted to $U_{ijk}$. Consequently, there is a continuous map $\e_{ijk} \co U_{ijk} \to \te$ such that for all $(u,x) \in U_{ijk}\times \tg_0$, 
$$
\gamma_{ijk}(u,x)=\bigl(u,\e_{ijk}(u)1_{\x_{ij}(u)\x_{jk}(u)x}\bigr). 
$$
Set $\x=\{\x_{ij}\}_{i,j \in I}$ and $\e=\{\e_{ijk}\}_{i,j,k \in I}$.

\begin{lem}\label{lem-principal-to-cocycle}
$c_\aaa=(\x,\e)$ is  a $\tg$-valued \v{C}ech cocycle subordinate to $\uu$.
\end{lem}
\begin{proof}
The fact that the  source of $\gamma_{ijk}(u,x)$ is  $T_{ij}T_{jk}(u,x)$ implies that each $\e_{ijk}$ does land into $\te \subset \tg_1$.
Since the target of $\gamma_{ijk}(u,x)$ is  $T_{ik}(u,x)$, we obtain $t(\e_{ijk})\x_{ij}\x_{jk}=\x_{ik}$ on $U_{ijk}$.
For any $i,j,k,l \in I$, the naturality of $\varepsilon_j$ implies the commutativity of the following diagram restricted to $U_{ijkl}$:
$$
\xymatrix@R=.8cm @C=1.5cm{
T_{ij}T_{jk}T_{kl}\ar@{->}[r]^-{\gamma_{ijk}T_{kl}}\ar@{->}[d]_-{T_{ij}(\gamma_{jkl})} & T_{ik}T_{kl} \ar@{->}[d]^-{\gamma_{ikl}} \\
T_{ij}T_{jl} \ar@{->}[r]_-{\gamma_{ijl}}& T_{il}.
}
$$
This commutativity amounts to the following equality on $U_{ijkl}$:
$$
\bigl ( \e_{ijk}1_{\x_{ij}\x_{jk}\x_{kl}} \bigr) \ast \bigl ( \e_{ikl}1_{\x_{ik}\x_{kl}} \bigr) = 
\bigl (1_{\x_{ij}} \e_{jkl}1_{\x_{jk}\x_{kl}}\bigr) \ast \bigl ( \e_{ijl}1_{\x_{ij}\x_{jl}} \bigr).
$$
We deduce that $\e_{ikl}\e_{ijk}=\e_{ijl}\bigl(\lact{\x_{ij}}{\!\e_{jkl}}\bigr)$ on $U_{ijkl}$ by
using the expression for the product $\ast$ of $\tg$ given in Section~\ref{sect-2-groups}.
\end{proof}
 
By Lemma~\ref{lem-cocycle-to-principal}, the cocycle $c=c_\aaa$ induces the principal $\tg$-bundle $\pi_c \co P_c\to \widebar{\B}$. 
Let  $f \co P_c \to P$ be the functor defined by
$$ 
f\bigl((u,x)_i\bigr)= \phi_i(u,x) \quad \text{and} \quad f\bigl((v,g)_{ij}\bigr)=  \phi_i(v,g) \ast \varepsilon_i \bigl(\phi_j(v,\x_{ij}^{-1}(v)t(g)) \bigr) 
$$
for all $(u,x)_i \in (P_c)_0$ and $(v,g)_{ij} \in (P_c)_1$.

\begin{lem}\label{lem-map-principal-to-cocycle}
The functor $f$ is a $\tg$-bundle  morphism $\pi_c \to \pi$.
\end{lem}
\begin{proof}
Let $(v,g)_{ij} \in (P_c)_1$. Set $x=\x_{ij}(v)^{-1} t(g) \in \tg_0$. Since 
\begin{gather*}
s\varepsilon_i \bigl(\phi_j(v,x) \bigr)  =  \phi_i\psi_i\phi_j(v,x)=\phi_i T_{ij}(v,x)\\
 = \phi_i (v,\x_{ij}(v)x) = \phi_i (v,t(g))=\phi_i t(v,g)=t\phi_i(v,g),
\end{gather*}
we get that $f\bigl((v,g)_{ij}\bigr)=  \phi_i(v,g) \ast \varepsilon_i \bigl(\phi_j(v,x) \bigr) $ is well-defined. Moreover,
\begin{gather*}
sf\bigl((v,g)_{ij}\bigr)=s\phi_i(v,g)=\phi_is(v,g) =\phi_i(v,s(g))
=f\bigl((v,s(g))_i\bigr)=fs\bigl((v,g)_{ij}\bigr),\\
tf\bigl((v,g)_{ij}\bigr)=t\varepsilon_i \bigl(\phi_j(v,x) \bigr)=\phi_j(v,x)
=f\bigl((v,x)_j\bigr)=ft\bigl((v,g)_{ij}\bigr).
\end{gather*}
Thus $f$ is compatible with the source and target maps. The functoriality of $f$ follows from the naturality of $\varepsilon_i$, the functoriality of $\phi_i$ and $\psi_i$, and the definition of the product of $P_c$, by using that
$$
f\bigl((v,g)_{ij}\bigr)=  \phi_i(v,g) \ast \varepsilon_i \bigl(tf((v,g)_{ij}) \bigr).
$$
The continuity of $f$  follows from that of the product $\ast$ of $P$, the functors $\phi_i$, and the natural transformations~$\varepsilon_i$.
The $\tg$-equivariance of $f$ follows from that of $\phi_i$ and~$\varepsilon_i$ and from the functoriality of the $ \tg$-action of $P$. Finally the compatibility of the functors $\psi_i$ and $\phi_i$ with $\pi_{|U_i}$ and $\mathrm{pr_1}$ (see Section~\ref{sect-principal-2-bundles}) implies that $\pi \circ f=\pi_c$. Hence, $f$ is a $\tg$-bundle morphism.
\end{proof}

\subsection{Classification of 2-bundles}\label{sect-classification} 
Let $\tg$ be a topological 2-group and $\B$ be a topological space.
Denote by $[c]$ the cohomology class of a $\tg$\ti valued \v{C}ech cocycle~$c$,  by $[\pi]$ the equivalence class of a principal $\tg$-bundle $\pi$ over~$\B$, and by $\pp(\B,\tg)$ the set of equivalence classes of principal $\tg$-bundles over~$\B$.

Let $\uu$ be an open cover of $\B$. The map 
$$
\check{H}(\uu,\tg) \xrightarrow{\pi_\uu} \pp(\B,\tg), \quad [c] \mapsto [\pi_c]
$$
is well-defined (by Lemma~\ref{lem-morphisms-Pc}) and is compatible with restriction (by Lemma~\ref{lem-restriction-cocycle-principal}):
if $\vv$ is a refinement of $\uu$, then the following diagram 
$$
\xymatrix@R=.1cm @C=1.2cm{ \check{H}(\uu,\tg) \ar@{->}[dd] \ar@{->}[rd]^-{\pi_\uu}  & \\
&  \pp(\B,\tg) \\
\check{H}(\vv,\tg) \ar@{->}[ru]_-{\pi_\vv}  & 
}
$$
commutes, where the vertical arrow is induced by restriction of cocycles. Consequently, since $\check{H}(\B,\tg) = \varinjlim \check{H}(\uu,\tg)$, there exists a unique map
$$
\Gamma \co \check{H}(\B,\tg) \to \pp(\B,\tg)
$$
such that $\pi_\uu=\Gamma \circ \kappa_\uu$ for all open cover $\uu$ of $\B$, where $\kappa_\uu \co \check{H}(\uu,\tg) \to \check{H}(\B,\tg)$ is the cocone map.

\begin{thm}\label{thm-bundle-classification}
The map $\Gamma \co \check{H}(\B,\tg) \to \pp(\B,\tg)$ is surjective. Moreover, if $\B$ admits good covers (see Section~\ref{sect-cocyle-descript}), then it is bijective.
\end{thm}

\begin{proof}
Let~$\pi$ be a principal $\tg$-bundle over $\B$. Pick an atlas $\aaa$ of $\pi$ and denote by $\uu$ the open cover underlying $\aaa$.
Let $c=c_\aaa$ be the $\tg$-valued \v{C}ech cocycle subordinate to $\uu$ associated with $\aaa$ (see Lemma~\ref{lem-principal-to-cocycle}). By  Lemma~\ref{lem-map-principal-to-cocycle}, there is a $\tg$-bundle  morphism from $\pi_c$ to $\pi$. Then 
$$
\Gamma\bigl(\kappa_\uu([c])\bigr)=\pi_\uu\bigl([c]\bigr)=[\pi_c]=[\pi].
$$
Thus $\Gamma$ is surjective.

Assume now that $\B$ admits good covers and let us prove that $\Gamma$ is injective. Let $a,b \in \check{H}(\B,\tg) $ such that $\Gamma(a)=\Gamma(b)$. Pick a good open cover $\uu$ of $\B$. Since the cocone map  $\kappa_\uu$ is then a bijection, there are $\tg$-valued \v{C}ech cocycles $c$ and $c'$ subordinate to $\uu$ such that $a=\kappa_\uu([c])$ and $b=\kappa_\uu([c'])$.
Then 
$$
[\pi_c]=\pi_\uu([c])=\Gamma(a)=\Gamma(b)=\pi_\uu([c'])=[\pi_{c'}].
$$
Consequently, there exist  a principal $\tg$-bundle  $\pi$ over $X$ and $\tg$-bundle morphisms $\pi \to \pi_c$ and $\pi \to \pi_{c'}$. Let $\vv$ be a trivializing good open cover for $\pi$ which refines the cover $\uu$. Pick an atlas for $\pi$ subordinate to $\vv$ and consider the associated  \v{C}ech cocycle  $d$ subordinate to $\vv$ (see Lemma~\ref{lem-principal-to-cocycle}). Composing the $\tg$-bundle  morphism $\pi_d \to \pi$ from Lemma~\ref{lem-map-principal-to-cocycle} with the above morphisms from $\pi$, we obtain  $\tg$-bundle morphisms $\pi_d \to \pi_c$ and $\pi_d \to \pi_{c'}$. By Lemma~\ref{lem-restriction-decomposition}, these morphisms decompose as  $\pi_d \to \pi_{c_{|\vv}} \to \pi_c$ and $\pi_d \to \pi_{c'_{|\vv}} \to \pi_{c'}$. In particular, using Lemma~\ref{lem-morphisms-Pc}, we deduce that $[c_{|\vv}]=[d]=[c'_{|\vv}]$. Consequently 
$$
a=\kappa_\uu([c])=\kappa_\vv\bigl([c_{|\vv}]\bigr)=\kappa_\vv\bigl([c'_{|\vv}]\bigr)=\kappa_\uu([c'])=b.
$$
Hence $\Gamma$ is injective, and so bijective. 
\end{proof}

The next corollary is a direct consequence of Theorem~\ref{thm-bundle-classification} and the Baez-Stevenson theorem (see  Section~\ref{sect-cocyle-descript}):

\begin{cor}\label{cor-principal-homotopy}
Assume that  $\tg$ is well-pointed and that  $\B$ is paracompact Hausdorff and admits good covers. Then there is a canonical bijection 
$$
\pp(\B,\tg) \cong [\B, B|\tg|]
$$
between the set of equivalence classes of principal $\tg$-bundles over~$\B$ and the set of homotopy classes of maps from $\B$ to the classifying space of the topological group~$|\tg|$ defined as the geometric realization of the nerve of $\tg$.
\end{cor}

Note that under the bijection of Corollary~\ref{cor-principal-homotopy}, the equivalence class of a trivializable principal $\tg$-bundle over~$\B$ (see Section~\ref{sect-principal-2-bundles}) corresponds to the homotopy class of a constant map $\B \to  B|\tg|$.

\subsection{Flat 2-bundles}\label{sect-flat}
Let $\tg$ be a topological $2$\ti group. Denote by $\tg^d$ the discrete $2$\ti group associated with the 2-group underlying~$\tg$ (see Section~\ref{sect-top-2-groups}). Note that the identity functor $\tg^d \to \tg$ is continuous and is a morphism of topological $2$\ti groups.  Consequently,  any  $\tg$-action on a 2-space induces a  $\tg^d$-action on this 2-space. In particular, any  $\tg$-2-space $P$ is a  $\tg^d$-2-space in a canonical way.

Let $\B$ be a topological space. Any morphism $\mathcal{K} \to \tg$ of topological $2$\ti groups induces a map $\check{H}(\B,\mathcal{K})\to \check{H}(\B,\tg)$. Composing this map with the surjection $\check{H}(\B,\tg) \to \pp(\B,\tg)$ of Theorem~\ref{thm-bundle-classification} gives rise to the map $\check{H}(\B,\mathcal{K}) \to \pp(\B,\tg)$. In particular, the morphism $\tg^d \to \tg$ leads to the map
\begin{equation}\label{eq-flat}
\check{H}(\B,\tg^d) \to \pp(\B,\tg).
\end{equation}
A  principal $\tg$-bundle  over $\B$ is \emph{flat} if its equivalence class lies in the image of \eqref{eq-flat}. A \emph{flat structure} of a flat principal $\tg$-bundle $\pi$ over $\B$ is an antecedent of $[\pi]$ under~\eqref{eq-flat}.

A $\tg$-valued \v{C}ech cocycle $c=(\x,\e)$ subordinate to an open cover $\uu= \{U_i\}_{i \in I}$ of $\B$ is \emph{locally constant} if the maps $\x_{ij} \co U_{ij} \to \tg_0$ and $\e_{ijk} \co U_{ijk} \to \te$ are locally constant for all $i,j,k \in I$. Observe that a $\tg$-valued \v{C}ech cocycle is locally constant if and only if it is a $\tg^d$-valued \v{C}ech cocycle. Consequently, a principal $\tg$-bundle  over $\B$ is flat if and only if it is equivalent to the principal $\tg$-bundle $\pi_c$ associated with a locally constant $\tg$-valued \v{C}ech cocycle $c$ subordinate to some open cover of~$\B$ (see Lemma~\ref{lem-cocycle-to-principal}).

Note that if the $2$\ti group $\tg$ is discrete (see Section~\ref{sect-top-2-groups}) and $X$ admits good covers, then the map \eqref{eq-flat} is bijective (by Theorem~\ref{thm-bundle-classification}) and so any  principal $\tg$\ti bundle  over~$\B$ is flat and has a   unique flat structure.

Let $\chi \co E \to H$ be a crossed module (see Section~\ref{sect-crossed-modules-def}). The 2-group $\tg_\CM$ associated with $\chi$ (see Section~\ref{sect-2-groups}) is discrete (by endowing $E$ and $H$ with the discrete topology) and so any  principal $\tg_\chi$-bundle  over~$\B$ is flat and has a   unique flat structure. Note that the classifying space of the geometric realization of the nerve of~$\tg_\CM$ is homotopy equivalent to the classifying space $B \chi$ of $\chi$ (see Section~\ref{sect-crossed-modules-classifying-spaces}). Moreover~$\tg_\chi$ is well-pointed (as any discrete 2-group). Consequently, we deduce from
Theorem~\ref{thm-bundle-classification} and Corollary~\ref{cor-principal-homotopy} that if  $\B$ is paracompact Hausdorff and admits good covers, then there are canonical bijections
$$
\pp(\B,\tg_\chi) \cong \check{H}(\B,\tg_\chi) \cong [\B,B\chi].
$$
Note that under these bijections, the equivalence class of a trivializable principal $\tg_\CM$-bundle over~$\B$  corresponds to the 
cohomology class of the unit \v{C}ech cocycle and to the homotopy class of a constant map $\B \to  B\CM$.

\section{Crossed modules of groupoids}\label{sect-crossed-modules-groupoids}
In this appendix, we review the basics of crossed modules of groupoids.

\subsection{Groupoids}\label{sect-groupoids}
A \emph{groupoid} is a small category in which all morphisms are isomorphisms. Given objects $a,b$ of a groupoid $G$, we denote by $G(a,b)$ the set of morphisms from $a$ to $b$ and by $G(a)$ the set of automorphisms of $a$.  We will often refer to the morphisms, identities, and composition of $G$ as the \emph{elements}, the \emph{units}, and the \emph{product} of $G$. As for the fundamental groupoids (for which the product is the concatenation of paths) and contrary to the usual notation for categories, we denote the product of two elements $x \in G(a,b)$ and $y \in G(b,c)$ by $xy\in G(a,c)$, and we denote the unit of an object $a$ by $1_a$. Given a set $S$, a \emph{groupoid over~$S$} is a groupoid whose set of objects is $S$.  A \emph{morphism of groupoids} between two groupoids is a functor between them.

Recall that a category is \emph{totally disconnected} if all of its morphisms are endomorphisms, and is \emph{skeletal} if isomorphic objects are necessarily equal. Clearly, a groupoid is totally disconnected if and only if it is skeletal. Such a groupoid is just a collection of groups indexed by the set of objects. A totally disconnected groupoid $G$ is \emph{abelian} if for any object $a$ of $G$, the group $G(a)$ is abelian.

\subsection{Free groupoids}\label{sect-free-groupoids}
The \emph{free groupoid on an oriented graph} is the groupoid whose objects are the vertices of the graph and whose elements are finite concatenations of the edges of the graph and formal inverses to them.  The only relations between elements are the necessary ones defining the unit of each object, the inverse of each edge, and the associativity of the product. 

A groupoid is \emph{free} if it is isomorphic to the free groupoid on some oriented graph. A \emph{free basis} for a free groupoid ${F}$ is a set $\mathcal{B}$ of elements of $F$
such that~${F}$ is isomorphic (via an isomorphism inducing the identity on objects) to the free groupoid on the oriented graph whose vertices are the objects of ${F}$ and whose edges are  the elements of $\mathcal{B}$.
Note that a functor from a free groupoid to another category is fully determined by its values on a free basis.

\subsection{Actions of groupoids}
Let $G$ be a groupoid with the set of objects~$G_0$. An \emph{action} of $G$ on a groupoid $D$ over~$G_0$ is a family of maps
$$
\bigl\{ G(a,b) \times {D}(b) \to {D}(a), \quad (x,e) \mapsto \lact{x}{\!e}\, \bigr\}_{a,b\in G_0}
$$
satisfying
$$
\lact{(1_c)}{e}=e, \quad \lact{x}{}(\lact{y}{e})=\lact{xy}{e}, \quad \text{and} \quad \lact{y}{(ef)}=(\lact{y}{e})(\lact{y}{f})
$$
for all $a,b,c \in G_0$, $e,f \in {D}(c)$, $x \in G(a,b)$, and $y \in G(b,c)$.
For example, $G$  acts on itself by conjugation: $$\lact{x}{y}=xyx^{-1} \in G(a)$$ for all $a,b \in G_0$, $y \in G(b)$, and $x \in G(a,b)$.

\subsection{Crossed modules of groupoids}\label{sect-crossed-modules-over-groupoids}
Let ${H}$ be a groupoid with the set of objects~${H}_0$.
A \emph{crossed module} over ${H}$ is a morphism of groupoids $$\CM \co {E} \to {H}$$
where ${E}=\{{E}(a)\}_{a \in {H}_0}$ is a totally disconnected groupoid over ${H}_0$ endowed with an action of ${H}$ such that
\begin{itemize}
\item $\CM$ is identity on objects,
\item $\CM$ preserves the ${H}$-action, where ${H}$ acts on itself by conjugation,
\item $\Ima(\CM)$ acts by conjugation on ${E}$.
\end{itemize}
This means that for all $a,b \in H_0$ and all elements $e,f\in E(b)$ and $x \in H(a,b)$,
$$
\CM(a)=a,  \quad \CM(\elact{x}{e})=x\CM(e)x^{-1},  \quad \lact{{\CM(e)}}{\!f}=efe^{-1}.
$$

Note that any crossed module (of groups) in the sense of Section~\ref{sect-crossed-modules-def} is a crossed module in the above sense when viewing a group as a groupoid with a single object.

\subsection{Morphisms of crossed modules of groupoids}\label{sect-morphisms-crossed-modules-over-groupoids}
A \emph{morphism of crossed modules} from a crossed module $\CM \co {E} \to {H}$ to a crossed module $\mu \co F \to K$ is a pair $(\psi \co {E} \to F,\varphi \co {H} \to K)$ of morphisms of groupoids inducing the same map on objects and satisfying
$$
\mu\bigl(\psi(e) \bigr)=\varphi \bigl( \CM(e) \bigr) \quad \text{and} \quad \psi(\elact{x}{e})=\lact{\varphi(x)}{\psi(e)}
$$
for all elements  $e\in E$ and  $x \in H$ with the same target.

\subsection{Free crossed modules}\label{sect-free-crossed-modules-over-groupoids}
Let ${H}$ be a groupoid with the set of objects~${H}_0$.
Consider a set $R$ and a map $\omega \co R \to {H}$ such that for each $r \in R$, the element $\omega(r)$ of $H$ is an endomorphism.
For any $a \in {H}_0$, let $\mathcal{P}_\omega(a)$ be the free group on the set of pairs  $(x,r)$ with $r  \in R$ and $x \in {H}(a,s(\omega(r)))$, where $s\co H \to H_0$ is the source map.
An action of~${H}$ on the totally disconnected groupoid  $\mathcal{P}_\omega=\{\mathcal{P}_\omega(a)\}_{a \in {H}_0}$ is given by $\lact{y}{(x,r)}=(yx,r)$ provided $yx$ is defined. The formula $(x,r) \mapsto x \omega(r) x^{-1}$ induces a morphism of groupoids $\upsilon_\omega \co \mathcal{P}_\omega \to {H}$ which is the identity on objects and preserves the ${H}$-action (where ${H}$ acts on itself by conjugation). Let $a \in {H}_0$. The \emph{Peiffer commutator} of two elements $m,n \in \mathcal{P}_\omega(a)$ is
$$
[\![m,n]\!]=\bigl(\lact{\upsilon_\omega(m)}{n}\bigr) m n^{-1}m^{-1} \in \mathcal{P}_\omega(a).
$$
Denote by $C_\omega(a)$ the subgroup of $\mathcal{P}_\omega(a)$ generated by the Peiffer commutators. Then $C_\omega(a)$ is a normal subgroup of $\mathcal{P}_\omega(a)$. Note that $C_\omega(a)$ can also be defined as the normal subgroup of $\mathcal{P}_\omega(a)$ generated by the basic Peiffer commutators which are Peiffer commutators of generators of $\mathcal{P}_\omega(a)$. Consider the groupoid 
$$
\mathcal{F}_\omega=\{\mathcal{F}_\omega(a)=\mathcal{P}_\omega(a)/C_\omega(a)\}_{a \in {H}_0}.
$$
The morphism $\upsilon_\omega$ and the $H$-action factorize through the quotient map $\mathcal{P}_\omega \to \mathcal{F}_\omega$ and give rise to a crossed module of groupoids $\nu_\omega\co\mathcal{F}_\omega \to {H}$ called the \emph{free crossed module on $\omega$}.

A crossed module $E \to {H}$ is \emph{free} if $H$ is a free groupoid and if it is isomorphic to  the free crossed module on some map to the endomorphisms of ${H}$. A \emph{free basis} for a free crossed module $\CM \co {E} \to {H}$ is a pair $\{\mathcal{B}_2,\mathcal{B}_1\}$ where $\mathcal{B}_1$ is a free basis for~$H$ (see Section~\ref{sect-free-groupoids}) and $\mathcal{B}_2$ is a set of elements of $E$ such that~$\CM$ is isomorphic (via an isomorphism inducing the identity on objects) to the free crossed module on the map $\mathcal{B}_2 \to {H}$ sending $b \in \mathcal{B}_2$ to $\CM(b)$. 

Let  $\CM \co {E} \to {H}$ be a free crossed module, $\mu \co F \to K$ be a crossed module, and $\rho\co H_0 \to K_0$ be a map between the sets of objects of $H$ and $K$. Then a morphism of crossed modules $(\psi,\varphi)\co\CM \to \mu$ inducing the map $\rho$ on objects is fully determined by its value on a free basis  $\{\mathcal{B}_2,\mathcal{B}_1\}$ for $\CM$, that is, by the elements
$\{f_b=\psi(b)\in F\}_{b \in \mathcal{B}_2}$ and   $\{k_c=\varphi(c)\in K\}_{c \in \mathcal{B}_1}$ satisfying
$$
s(k_c)= \rho(s(c)), \quad t(k_c)= \rho(t(c)), \quad t(f_b)= \rho(t(b)), \quad \mu(f_b)=\varphi ( \CM(b))
$$
for all $c \in \mathcal{B}_1$ and $b \in \mathcal{B}_2$, where $s,t$ are the source and target maps.

\section{Homotopy classification theorem}\label{sect-HVThm}
In this appendix, we recall the homotopy classification theorem. To this end, we first review the basics of crossed complexes. For more details, we refer to \cite{BHS}.

\subsection{Modules over groupoids}\label{sect-modules-def}
Let $G$ be a groupoid with the set of objects~$G_0$. A \emph{$G$-module} is an abelian totally disconnected groupoid over $G_0$  endowed with an action of $G$. A morphism of $G$-modules is a morphism of groupoids which is the identity on objects and preserves the $G$-action.

Given a map $\omega \co R \to G_0$, where $R$ is a set, the \emph{free $G$-module  on $\omega$} is the $G$\ti module ${F}(\omega)=\{{F}(\omega)(a)\}_{a \in G_0}$ where ${F}(\omega)(a)$ is the free abelian group on the set of pairs  $(x,r)$ with $r  \in R$ and $x \in G(a,\omega(r))$ and the action of $G$ is given by $\lact{y}{(x,r)}=(yx,r)$ provided $yx$ is defined in $G$.

A $G$-module is \emph{free} if it is isomorphic to a free $G$-module on some map. A \emph{free basis} for a free $G$-module ${M}$ is a set $\mathcal{B}$ of elements of $M$ such that ${M}$ is isomorphic (via an isomorphism inducing the identity on objects) to the free $G$-module on the map $\mathcal{B} \to G_0$ sending $b \in \mathcal{B}$ to the source of $b$.

\subsection{Crossed complexes of groupoids}\label{sect-crossed-complexes-def}
A \emph{crossed complex} $\CO$ is a sequence of groupoid morphisms
\begin{align*}
  \cdots \to \CO_n \xrightarrow{\delta_n} \CO_{n-1} \xrightarrow{ \delta_{n-1}} \cdots \cdots \xrightarrow{\delta_3} \CO_2 \xrightarrow{\delta_2} \CO_1
\end{align*}
where $\delta_2$ is a crossed module of groupoids and for all $n \geq 3$,
\begin{itemize}
\item $\CO_n$ is a $\CO_1$-module on which $\Ima(\delta_2)$ acts trivially,
\item $\delta_n$ is a morphism of groupoids which is the identity on objects and preserves the $\CO_1$-action,
\item $\delta_{n-1}\delta_n$ is the trivial morphism (which sends elements to units).
\end{itemize}
The condition that $\Ima(\delta_2)$ acts trivially on $\CO_n$ for $n\geq 3$ implies that $\CO_1$ acts on~$\CO_n$ through the fundamental groupoid $$\pi_1(\CO)= \Coker(\delta_2) = \CO_1 / \Ima(\delta_2)$$ of $\CO$, thereby making $\CO_n$ into a $\pi_1(\CO)$-module.

\subsection{Morphisms of crossed complexes}\label{sect-morphisms-crossed-complexes}
Let $\CO$ and $D$ be crossed complexes. A \emph{morphism of crossed complexes} $f \co  \CO \to D$ is a family $f=\{f_n \co\CO_n \to D_n\}_{n \geq 1}$ of morphisms of groupoids
$$
  \begin{tikzcd}
  \cdots \arrow[r] &  \CO_n \arrow[r, "\delta_n"] \arrow[d, "f_n"] & \CO_{n-1} \arrow[r, "\delta_{n-1}"] \arrow[d, "f_{n-1}"] & \cdots \arrow[r, "\delta_3"]  & \CO_2 \arrow[r, "\delta_2"] \arrow[d, "f_2"]& \CO_1 \arrow[d, "f_1"] \\
  \cdots \arrow[r] &  D_n \arrow[r, "\delta_n"'] & D_{n-1} \arrow[r, "\delta_{n-1}"'] & \cdots \arrow[r, "\delta_3"']& D_2 \arrow[r, "\delta_2"'] & D_1
  \end{tikzcd}
$$
which induce the same map on objects, denoted $f_0\co C_0 \to D_0$, and satisfy
$$
\delta_n f_n(c)= f_{n-1} \delta_n(c)
\quad \text{and} \quad
f_n(\lact{c_1}{c}) = \lact{f_1(c_1)}{f_n(c)}
$$
for all $n\geq 2$ and all elements $c \in \CO_n$ and $c_1 \in \CO_1$ with the same target. In particular, the pair $(f_2,f_1)$ is a morphism of crossed modules.

Crossed complexes and their morphisms form a category denoted by $\textsf{Crs}$.

\subsection{Free crossed complexes}\label{sect-free-crossed-complexes-def}
A crossed complex 
$$
\CO =\{ \dots \xrightarrow{\delta_4} \CO_3\xrightarrow{\delta_3} \CO_2 \xrightarrow{\delta_2} \CO_1 \}
$$
is \emph{free} if the crossed module~$\delta_2$ is free and for all $n \geq 3$, the $\pi_1(\CO)$-module $\CO_n$ is free. If such is the case, a \emph{free basis for $\CO$} is a family $\mathcal{B}_*=\{\mathcal{B}_n\}_{n \geq 1}$ where  $\{\mathcal{B}_1,\mathcal{B}_2\}$ is a free basis for~$\delta_2$ and $\mathcal{B}_n$ is a free basis for the $\pi_1(\CO)$-module $\CO_n$ for all $n \geq 3$.

Let  $C$ be a free crossed complex, $D$ be a crossed complex, and $f_0\co C_0 \to D_0$ be a map between the sets of objects of $C$ and $D$. Then a morphism of crossed complexes $f\co \CO \to D$ inducing the map $f_0$ on objects is fully determined by its value on a 
free basis $\mathcal{B}_*=\{\mathcal{B}_n\}_{n \geq 1}$ for $C$, that is, by the elements $\{f_n(b)\}_{b \in \mathcal{B}_n}$  with $n\geq 1$ satisfying
\begin{itemize}
  \item $s(f_1(b))= f_0(s(b))$ and $t(f_1(b))= f_0 (t(b))$ for all $b \in \mathcal{B}_1$,
  \item $t(f_n(b))= f_0(t(b))$ and $\delta_n \circ f_n(b)= f_{n-1}\circ \delta_n(b)$ for all $b \in \mathcal{B}_n$ with $n \geq 2$,
\end{itemize}
where $s,t$ are the source and target maps. 

\subsection{Homotopies}\label{sect-homotopy-of-crossed-complexes}
Let $f,g \co \CO \to D$ be two  morphisms of crossed complexes inducing the maps $f_0,g_0 \co C_0 \to D_0$ on objects.
A \emph{homotopy} $H$ from $f$ to $g$ is a family $H= \{ H_n \co \CO_n \to D_{n+1} \}_{n \geq 0}$ of maps
\begin{equation*}
  \begin{tikzcd}
  \cdots \arrow[r] &  \CO_{n+1} \arrow[r, "\delta_{n+1}"] \arrow[d,shift left,"g_{n+1}"]  \arrow[d,shift right,"f_{n+1}"'] &[1.5em] \CO_{n} \arrow[r] \arrow[d,shift left,"g_n"]  \arrow[d,shift right,"f_n"'] \arrow[dl, "\!\!H_n"] & \cdots \arrow[r] & \CO_2 \arrow[r, "\delta_2"] \arrow[d,shift left,"g_2"]  \arrow[d,shift right,"f_2"'] &[1.5em] \CO_1 \arrow[d,shift left,"g_1"]  \arrow[d,shift right,"f_1"'] \arrow[r, "{s,t}"]  \arrow[dl, "\!\!H_1"] &[1.5em] \CO_0 \arrow[d,shift left,"g_0"]  \arrow[d,shift right,"f_0"']  \arrow[dl, "\!\!H_0"]\\
  \cdots \arrow[r] &  D_{n+1} \arrow[r, "\delta_{n+1}"'] & D_{n} \arrow[r] & \cdots \arrow[r]& D_2 \arrow[r, "\delta_2"'] & D_1 \arrow[r, "{s,t}"'] & D_0
  \end{tikzcd}
  \end{equation*}
which satisfy the following conditions:
\begin{itemize}
    \item $t(H_0(a))= g_0(a)$ for all $a \in \CO_0$.
  \item $H_n(c) \in D_{n+1} (g_0(s(c)))$ for all $c \in \CO_n$ and $n \geq 1$.
  \item $H_1$ is a derivation over $g_1$: for all $x,y \in \CO_1$ with $t(x)=s(y)$,
   $$H_1(xy)= H_1(x)\bigl(\lact{g_1(x)}{H_1(y)}\bigr).$$
 \item $H_n$ is multiplicative and preserves the action over $g_1$ when $n \geq 2$:  for all~$x \in \CO_1$ and $c,d \in \CO_n$ with $t(x)=t(c)=t(d)$,
  $$H_n(cd)=H_n(c) H_n(d) \quad \text{and} \quad H_n( \lact{x}{c})= \lact{{g_1(x)}}{H_n(c)}.$$  
  \item $f$ is computed from $H$ and $g$ as follows: for all $c \in \CO_n$,
      \begin{align*}
        f_n(c) = \left\{ \begin{array}{ll}
        s(H_0(c)) & \text{if } n=0,\\
H_0(s(c))\delta_2(H_1(c))g_1(c)H_0(t(c))^{-1} & \text{if } n=1,\\
        \lact{{H_0(s(c))}}{\!{\big( \delta_{n+1}(H_n(c)) H_{n-1}(\delta_n (c)) g_n(c) \big)}} & \text{if } n \geq 2.
        \end{array} \right.
      \end{align*}
\end{itemize}

If the crossed complex $\CO$ is free, then such a homotopy is completely determined by its values on a free basis of $\CO$ (see \cite[Corollary 9.6.6]{BHS}).

\subsection{The groupoid of homotopies}\label{sect-groupoid-of-homotopies}
Let $C$ and $D$ be crossed complexes. Homotopies between  morphisms of crossed complexes $\CO \to D$ form a groupoid. The product
$HK$ of a homotopy $H$ from $f$ to $g$ with a homotopy $K$ from $g$ to $h$, the unit $1_f$ of a morphism $f$, and the inverse $H^{-1}$ of a homotopy $H$ are computed by:
\begin{align*}
& (H\ast K)_0(a)=H_0(a)K_0(a), && (H\ast K)_n(c)=\Bigl(\lact{K_0(s(c))^{-1}}{\! H_n(c)}\Bigr) K_n(c), \\
& (1_f)_0(a)=1_{f_0(a)} , && (1_f)_n(c)=1_{f_0(s(c))}, \\
& (H^{-1})_0(a)=H_0(a)^{-1}, && (H^{-1})_n(c)=\lact{H_0(s(c))}{\bigl(}H_n(c)^{-1} \bigr),
\end{align*}
for all  $a \in \CO_0$ and $c \in \CO_n$ with $n \geq 1$. 

In particular, this implies that being homotopic is an equivalence relation on morphisms of crossed complexes. We denote by $[C,D]$ the set of homotopy classes of morphisms of crossed complexes $C \to D$.

\subsection{Crossed complexes of filtered spaces}\label{exa-crossed-complexes-filtered}
Let $X_*$ be a filtered space, that is, a topological space $X$ equipped with a filtration
\begin{align*}
  X^{\ast}= \{X^0 \subseteq X^1 \subseteq \dots \subseteq X^n \subseteq \dots \subseteq X\}
\end{align*}
consisting in an increasing sequence of subspaces whose union is $X$. Consider the groupoid $\pi_1(X^1,X^0)$ whose product is the concatenation of paths.
The \emph{fundamental crossed complex $\Pi X^{\ast}$} of $X^{\ast}$ over the groupoid $\pi_1(X^1,X^0)$ is given by
\begin{align*}
  \cdots \xrightarrow{\delta_{n+1}} \pi_n( X^n, X^{n-1},X^0) \xrightarrow{\delta_{n}} \pi_{n-1}(X^{n-1}, X^{n-2}, X^0) \to \cdots \\
  \cdots   \xrightarrow{\delta_3} \pi_2(X^2,X^1,X^0) \xrightarrow{\delta_2} \pi_1(X^1,X^0)
\end{align*}
where $\delta_2$ is the usual boundary map while, for any $x \in X^0$ and $n\geq 3$, the restriction of $\delta_n$ to the component $\pi_n(X^n, X^{n-1},x)$ is given by the composition
\begin{align*}
  \pi_n(X^n, X^{n-1},x) \xrightarrow{\partial_n} \pi_{n-1}(X^{n-1},x) \xrightarrow{i_{n-1}} \pi_{n-1}(X^{n-1}, X^{n-2}, x).
\end{align*}
The action of the groupoid $\pi_1(X^1,X^0)$ is given by generalizing the fundamental group action on higher homotopy groups.

The assignment $X^* \mapsto \Pi X^{\ast}$ extends to a functor $\Pi \co \textsf{FTop} \to \textsf{Crs}$ from the category of filtered spaces and filtered maps to the category of crossed complexes (see \cite[Section 7.1.v]{BHS}).

\subsection{Crossed complexes of CW-complexes}\label{exa-crossed-complexes-of-CW-complexes}
Let $X$ be a CW-complex. Denote by $X^*$ its skeletal filtration (i.e.,~$X^n$ is the  $n$-th skeleton).
Then the fundamental crossed complex~~$\Pi X^{\ast}$ (see Section~\ref{exa-crossed-complexes-filtered}) is the free crossed complex with free basis given by the homotopy classes of the characteristic maps of the cells of $X$ (see \cite[Corollary 8.3.14]{BHS}).

\subsection{Homotopy classification theorem}\label{defn-of-pointed-crossed-things}
The functor $\Pi \co \textsf{FTop} \to \textsf{Crs}$ from Section~\ref{exa-crossed-complexes-filtered} has a section (up to natural equivalence) $B\co \textsf{Crs} \to \textsf{FTop}$ which sends a crossed complex~$\CO$ to its \emph{classifying space}  $B\CO$ (defined as the geometric realization of the nerve of~$\CO$). For example, regarding a group as a groupoid with a single object,  a crossed module $\CM\co E \to H$ (of groups) in the sense of Section~\ref{sect-crossed-modules-def} gives rise to the crossed complex $$\dots \to 1 \to 1 \to E \xrightarrow{\CM} H$$ whose classifying space is the classifying space $B \CM$ of $\CM$ described in Section~\ref{sect-crossed-modules-classifying-spaces}.

A key theorem in the theory of crossed complexes is the following \emph{homotopy classification theorem}:
\begin{thm}[{\cite[Theorem A]{BH1}, \cite[Theorem 11.4.19]{BHS}}]\label{thm-HCT}
Let~$\CO$ be a crossed complex and let $M$ be a CW-complex with  its skeletal filtration $M^{\ast}$. Then there is a canonical bijection
$$
[M, B \CO] \cong [\Pi M^{\ast}, \CO]
$$ 
where the left-hand side is the set of homotopy classes of maps of spaces, and the right-hand side is the set of homotopy classes of morphisms of crossed complexes.
\end{thm}

\subsection*{Acknowledgements}
The authors acknowledge the support of the CDP C2EMPI, the CPER WaveTech@HdF, and the NSERC Discovery Grant.

\end{document}